%% file: main.tex
\documentclass{amsart}
\usepackage[USenglish]{babel}
\usepackage[T1]{fontenc}
\calclayout
            
\numberwithin{equation}{section}
\usepackage{caption}
\usepackage{subcaption}
\usepackage{mathtools,amsmath,amsthm,amsfonts,amssymb,enumerate}
\usepackage[colorlinks,allcolors=blue,backref=page]{hyperref}
\renewcommand*{\backrefalt}[4]{%
  \ifcase #1 %
  \or
    {↑#2}%
  \else
    {↑#2, ↑#3}%
  \fi}
\usepackage{url}
\usepackage{euscript}
\usepackage{enumitem}
\usepackage{xcolor}
\usepackage{enumitem}
\usepackage{pdfpages} 
\usepackage{tikz-cd}
\usepackage{graphicx} 
\usepackage{comment}
\usepackage{dutchcal}
\usepackage{float}
\usepackage{wrapfig}

\setlist[itemize,2]{label=$\circ$}

\makeatletter
\def\mathcolor#1#{\@mathcolor{#1}}
\def\@mathcolor#1#2#3{%
  \protect\leavevmode
  \begingroup
    \color#1{#2}#3%
  \endgroup
}
\makeatother

\usepackage{pstricks}
\usepackage[all]{xy}
\SelectTips{cm}{}
\usepackage{pinlabel}

\newtheorem{manualtheoreminner}{Theorem}
\newenvironment{manualtheorem}[1]{%
  \IfBlankTF{#1}
    {\renewcommand{\themanualtheoreminner}{\unskip}}
    {\renewcommand\themanualtheoreminner{#1}}%
  \manualtheoreminner
}{\endmanualtheoreminner}

\newtheorem{manualconjectureinner}{Conjecture}
\newenvironment{manualconjecture}[1]{%
  \IfBlankTF{#1}
    {\renewcommand{\themanualconjectureinner}{\unskip}}
    {\renewcommand{\themanualconjectureinner}{#1}}%
  \manualconjectureinner
}{\endmanualconjectureinner}

\newtheorem{manualpropositioninner}{Proposition}
\newenvironment{manualproposition}[1]{%
  \IfBlankTF{#1}
    {\renewcommand{\themanualpropositioninner}{\unskip}}
    {\renewcommand{\themanualpropositioninner}{#1}}%
  \manualpropositioninner
}{\endmanualpropositioninner}

\newtheorem{manualinsightinner}{Approximate Miracle}
\newenvironment{manualinsight}[1]{%
  \IfBlankTF{#1}
    {\renewcommand{\themanualinsightinner}{\unskip}}
    {\renewcommand{\themanualinsightinner}{#1}}%
  \manualinsightinner
}{\endmanualinsightinner}

\newtheorem{thm}{Theorem}[section]
\newtheorem{lem}[thm]{Lemma}

\newtheorem{prop}[thm]{Proposition}
\newtheorem{cor}[thm]{Corollary}
\newtheorem{claim}[thm]{Claim}
\newtheorem{conj}[thm]{Conjecture}

\theoremstyle{definition}
\newtheorem{defn}[thm]{Definition}
\newtheorem{nota}[thm]{Notation}

\theoremstyle{remark}
\newtheorem{rem}[thm]{Remark}

\newtheoremstyle{named}{}{}{\itshape}{}{\bfseries}{.}{.5em}{\thmnote{#3 }#1}
\theoremstyle{named}

\newcommand{\nc}{\newcommand} 
\nc{\mb}{\mathbb}
\nc{\mc}{\mathcal}
\nc{\mf}{\mathfrak}

\nc{\ic}{\mathbf{IC}}

\mathtoolsset{showonlyrefs}

    \def\AC{{\mathcal{A}}}
    
    \def\CC{{\mathcal{C}}}
    \def\DC{{\mathcal{D}}}

    \def\HC{{\mathcal{H}}}

\usepackage{mathrsfs}

\def\a{\alpha}

\def\d{\delta}

\def\l{\lambda}

\def\th{\theta}

\def\t{\tau}

\usepackage{bbm}

\def\1{\mathbbm{1}}

\newcommand{\id}{{\rm id}}

\newcommand{\Aut}{\textrm{Aut}}

\usepackage{bbm}

\def\1{\mathbbm{1}}
\def\cen{\operatorname{cen}}
\def\Pgn{\operatorname{Pgn}}

\def\cPgn{\operatorname{cPgn}}
\def\Cone{\operatorname{Cone}}
\def\Par{\operatorname{Par}}
\def\Sgm{\operatorname{Sgm}}
\def\cSgm{\operatorname{cSgm}}
\def\DD{\mathcal{D}}

\newcommand{\norm}[1]{\left\lVert#1\right\rVert}

\usepackage{cite}

\definecolor{airforceblue}{rgb}{0.36, 0.54, 0.66}
\definecolor{lilac}{rgb}{0.78, 0.64, 0.78}

\usepackage[normalem]{ulem}
\usepackage{cancel}

\title[Shape and class of Bruhat Intervals]{Shape and class of Bruhat Intervals \\[0.5em] \normalfont{\footnotesize{Towards a classification via Euclidean geometry}}}

\author{Gaston Burrull}
\address{(Gaston Burrull) \newline \indent Beijing International Center for Mathematical Research, Peking University, No.\@ 5 Yiheyuan Road, Haidian District, Beijing 100871, China}
\email{gaston(at)bicmr(dot)pku(dot)edu(dot)cn}

\author{Nicolas Libedinsky}
\address{(Nicolas Libedinsky) \newline \indent Universidad de Chile, Departamento de Matem\'aticas, Las Palmeras 3425, Casilla 653, \newline \indent Santiago, Ñuñoa, 7800003, Chile}
\email{nlibedinsky@gmail.com}

\author{Rodrigo Villegas}
\address{(Rodrigo Villegas)}
\email{rod(dot)villegasg(at)gmail(dot)com}
\date{}
\usepackage{microtype}

\begin{document}
\begin{abstract}
    We study Bruhat intervals in affine Weyl groups by viewing them as regions of alcoves.
    In type $\widetilde{A}_2$ we show that each interval coincides with a generalized permutohedron minus a star-shaped polygon, and we prove a subtler version inside the dominant chamber of type $\widetilde{A}_n$.
    Motivated by this geometry, we conjecture that whenever two Bruhat intervals are isomorphic, there exists an isomorphism realized by a piecewise isometry.
    We prove this when both endpoints are dominant in $\widetilde{A}_2$ and obtain partial results in $\widetilde{A}_n$. In the course of proving these results, we made the surprising observation that much of the information contained in a Bruhat interval is already encoded in a tiny portion of it.
\end{abstract}
\maketitle

\input{Sections/Introduction}
\input{Sections/Preliminaries}
\input{Sections/Geometry}
\input{Sections/Translation}
\input{Sections/Sides}
\input{Sections/Proof}
\input{Sections/Applications}
\input{Sections/Towards_a_general_picture}
\bibliographystyle{abbrv}
\bibliography{biblio}
\end{document}

%% file: Sections/Introduction.tex
\section{Introduction}
Consider the following selection of fundamental and interconnected open problems in Lie theory, listed in roughly decreasing order of complexity:

\begin{enumerate}
    \item\label{problem: characters} Determine the characters of simple representations of reductive algebraic groups over fields of positive characteristic.
    \item\label{problem: p-canonical} Compute the $p$-canonical basis for affine Weyl groups.
    \item\label{problem: class} Classify all Bruhat intervals modulo poset isomorphism.
    \item\label{problem: projectors} Construct projectors onto indecomposable objects in the (anti-spherical) Hecke category in characteristic zero.
    \item\label{problem: closed KL} Find closed formulas for Kazhdan--Lusztig polynomials in affine Weyl groups.
    \item\label{problem: structure of Bruhat intervals} At the more accessible end: counting and analyzing the geometry of Bruhat intervals.
\end{enumerate}
We view Problems~(\ref{problem: class}) and~(\ref{problem: structure of Bruhat intervals}) as especially compelling because they may offer critical insights into the other problems listed above (see Remark~\ref{rem: euclideanalcovic},  Section~\ref{ss: stabilization}, and Section~\ref{ss: dihedral} for some examples), by illuminating structural patterns that might be obscured in richer categorical or representation-theoretic settings.

\subsection{The geometry of Bruhat intervals}
Let $W$ be a crystallographic Coxeter group.
The main result of the important paper \emph{On the Shape of Bruhat Intervals} by Björner and Ekedahl~\cite{BE09} is the following: for any $y \in W$, let $f_i^y$ denote the number of elements of length $i$ in the Bruhat interval $[\mathrm{id}, y]$.
Then,
\begin{equation*}
    0 \leq i < j \leq \ell(y) - i \implies f_i^y \leq f_j^y.
\end{equation*}
These inequalities constrain the ``combinatorial shape'' of the lower Bruhat interval $[\mathrm{id}, y]$.
In this paper, we are interested in the  ``geometric shape'' of Bruhat intervals in the context of affine Weyl groups and general intervals $[x, y]$: by viewing an affine Weyl group as a collection of alcoves in Euclidean space, we aim to describe $[x, y]$ in geometric terms---for instance, as unions of a small number of convex polytopes or similar structures.

Let us begin with the affine Weyl group $W$ of type~$\widetilde{A}_2$.
This is the simplest affine Weyl group that already exhibits a rich combinatorial structure.
Despite its apparent simplicity, some of the fundamental questions listed above remain open in this setting.
For example, Problem~(\ref{problem: characters}) is open for $\mathrm{SL}_3$ and Problem~(\ref{problem: p-canonical}) is open for $\widetilde{A}_2$, though there is the \emph{billiards conjecture} by Lusztig and Williamson~\cite{LW18} that if proved, would give a partial solution to these problems.
Problems~(\ref{problem: projectors}) and~(\ref{problem: closed KL}), concerning projectors in the Hecke category and Kazhdan--Lusztig polynomials, have been settled only recently for $\widetilde{A}_2$ by the second author and Patimo~\cite{LP23}.

For ease of exposition, we now focus on \emph{dominant elements}, i.e., those lying in the dominant chamber $C_+$, which corresponds to the light-blue region\footnote{Note that the red alcove represents the identity element.
This convention is commonly used in representation theory but differs from the standard Bourbaki labeling~\cite{Bou68}.} in Figure~\ref{subfig: hexagon 1}.
We aim to give a geometric visualization of the interval $[x, y]$, where both $x$ and $y$ lie in~$C_+$.

We begin by choosing an alcove $y$ as shown in Figure~\ref{subfig: hexagon 1}.
We then construct a hexagon $C_y$ as the convex hull of the orbit of the center of this alcove under the action of the finite Weyl group $W_f \cong S_3$, which is generated by reflections across the thick black lines in the figure.

\begin{figure}[H]
    \centering
    \begin{subfigure}{.38\linewidth}
        \includegraphics[width=\textwidth]{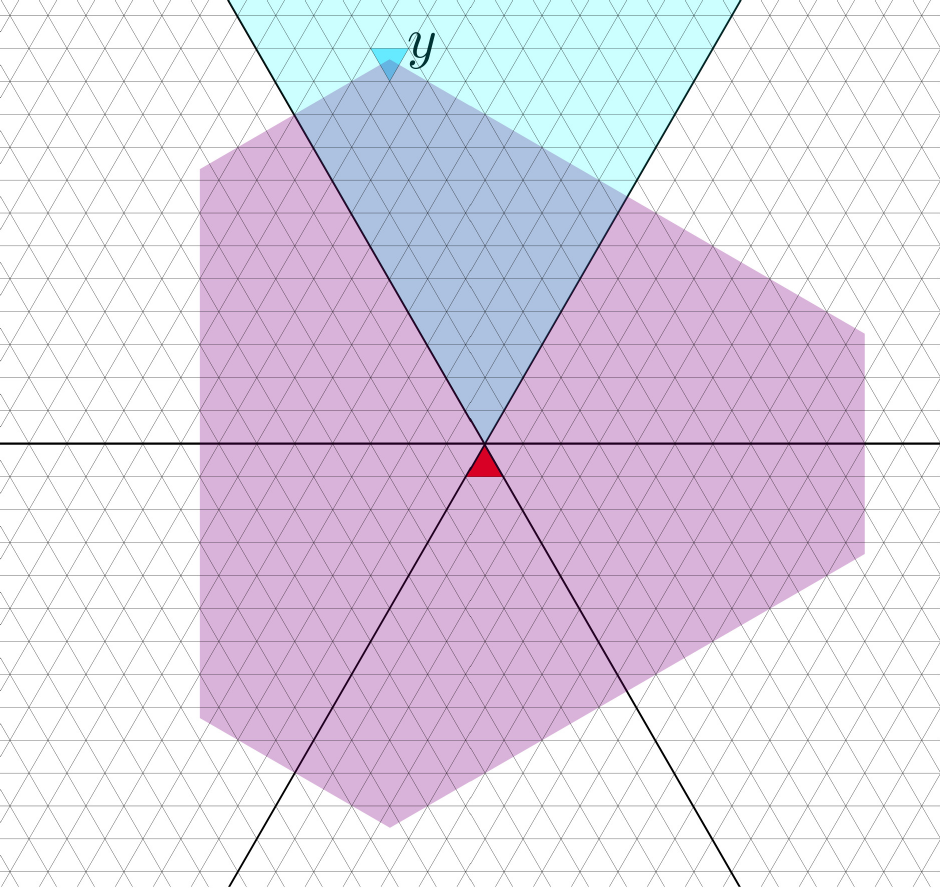}
        \caption{The dominant chamber and the hexagon $C_y=\operatorname{Conv}({W_f}\cdot y)$.}
        \label{subfig: hexagon 1}
    \end{subfigure}
    \hspace{0.7cm} 
    \begin{subfigure}{.38\linewidth}
        \includegraphics[width=\linewidth]{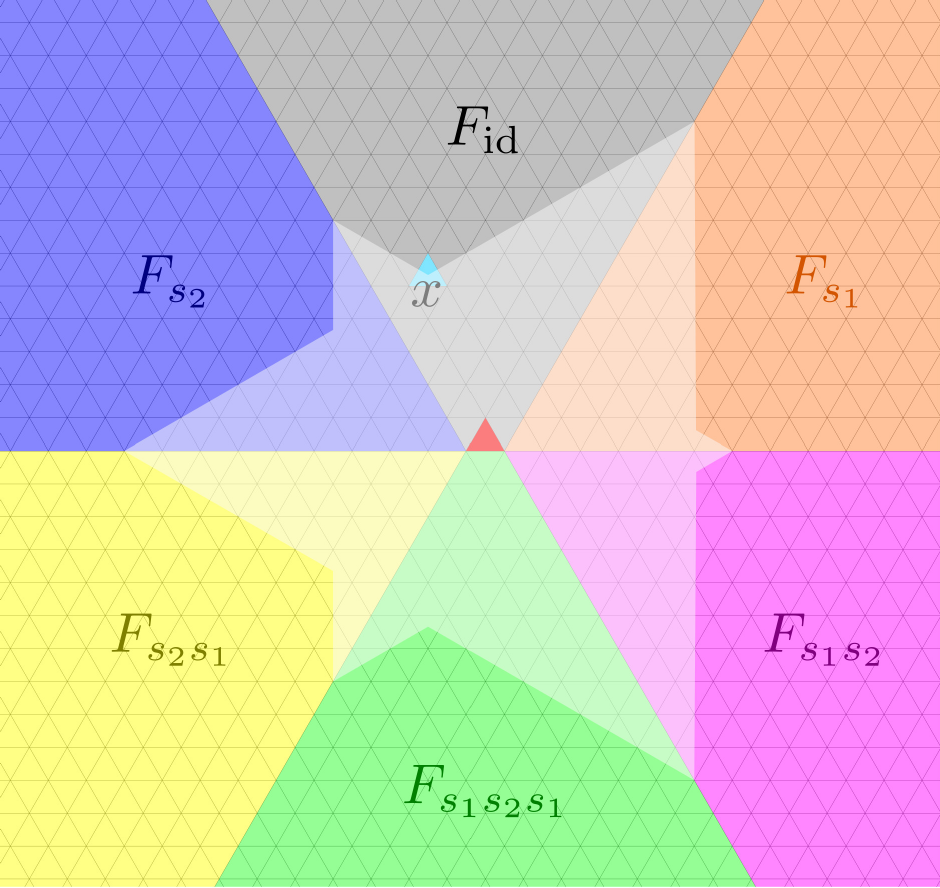} 
        \caption{The star $\mathcal{St}(x)$ and the six zones $F_w$ (for $w\in W_f$.)}
        \label{subfig: star and Fw zones intro}
    \end{subfigure}\\
    \vspace{0.3cm} 
    \caption{}
    \label{fig: hexagon and star intro}
\end{figure}

Now, fix another alcove $x$.
We define a star-shaped region $\mathcal{St}(x)$ passing through $x$, illustrated in Figure~\ref{subfig: star and Fw zones intro}.
This star is the union of two equilateral triangles, denoted $\triangleleft$ and $\triangleright$, constructed as follows:
\begin{itemize}
    \item $\triangleleft$ is a left-facing triangle centered at the bottom-right vertex of the red triangle.
    Its boundary contains the center of~$x$.
    \item $\triangleright$ is a right-facing triangle centered at the bottom-left vertex of the red triangle, and its boundary also contains the center of~$x$.
\end{itemize}

For any $x, y\in W$ (i.e., not necessarily in the dominant chamber), we also define a convex polygon (usually a hexagon) $C_y$ passing through the center of $y$ and a star-shaped region $\mathcal{St}(x)$ whose boundary contains the center of~$x$ (see Section~\ref{subsec: geometry of general intervals}).

\begin{manualtheorem}{A}\label{thm: A}
    An element $z \in W$ belongs to the interval $[x, y]$ if and only if the center of its alcove lies in $C_y$, and does not lie in the interior of~$\mathcal{St}(x)$.
\end{manualtheorem}
The following figure illustrates this theorem for the example above.
\begin{figure}[H]
    \centering
    \begin{subfigure}{.38\linewidth}
        \includegraphics[width=\linewidth]{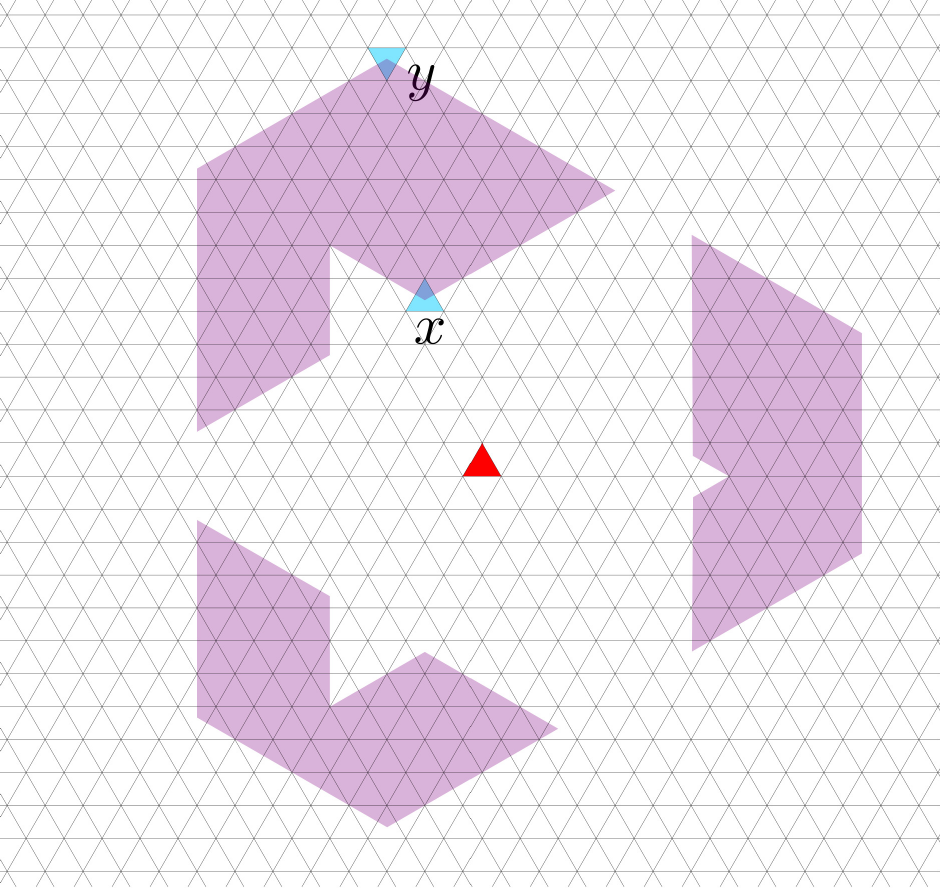} 
        \caption{The set $C_y\setminus\mathcal{St}^\circ (x)$.}
        \label{subfig: mixing hexagon and star}
    \end{subfigure}
    \hspace{0.7cm} 
    \begin{subfigure}{.38\linewidth}
        \includegraphics[width=\linewidth]{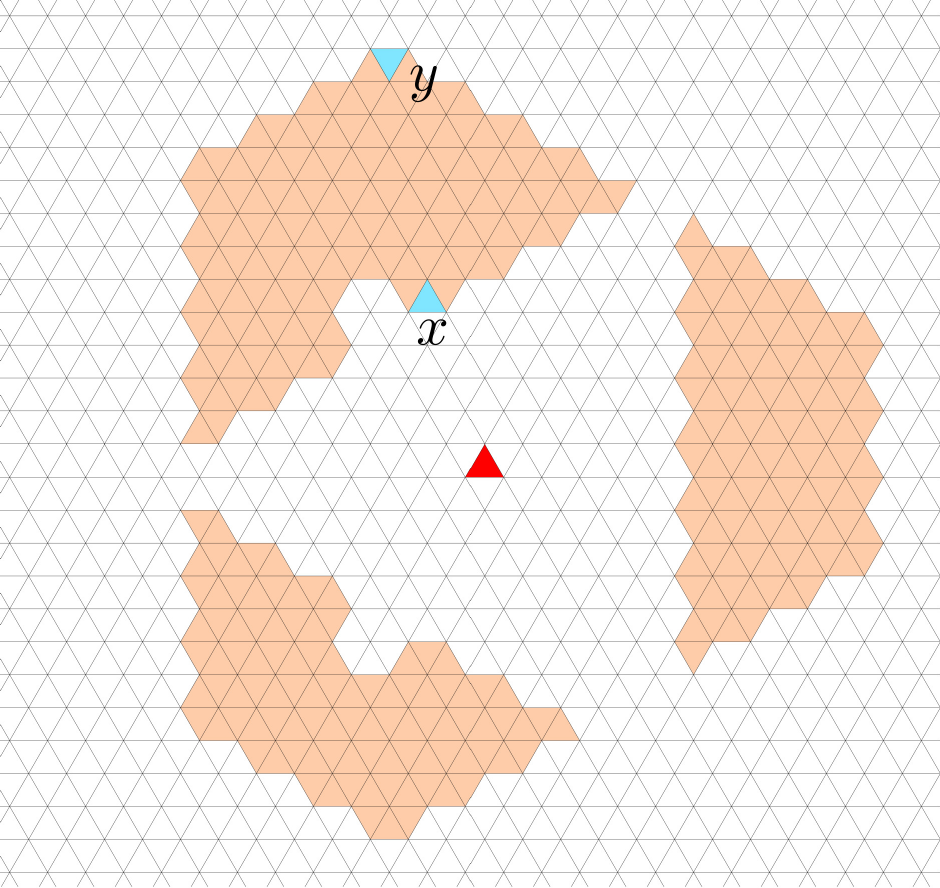} 
        \caption{Alcoves in~$[x,y]$.}
        \label{subfig: interval hexagon and star}
    \end{subfigure}
    \caption{}
    \label{fig: geometric realization of a Bruhat interval}
\end{figure} 
To generalize the previous result, we begin by reinterpreting the construction within the dominant chamber $C_+$.
For $x \in W$, let $A_x$ denote the alcove corresponding to $x$, and let $\mathrm{cen}(x)$ denote the center of~$A_x$.
Let $K \coloneqq \mathrm{cone}(\alpha_1, \alpha_2)$, where $\alpha_1$ and $\alpha_2$ are the simple roots for $\widetilde{A}_2$ (see Figure~\ref{Fig: teselado} for a visual reference).
Suppose that $x,y$ lie in~$C_+$.
Then Theorem~\ref{thm: A}, when restricted to $C_+$, is equivalent to the following:
\vspace{0.1cm}
\begin{equation}
    \centering
    z \in [x, y] \cap C_+ \ \ \iff \ \  \mathrm{cen}(z) \in C_+ \cap \left( \mathrm{cen}(x) + K \right) \cap \left( \mathrm{cen}(y) - K \right).
\end{equation}
We now generalize this result to higher ranks.
Let $W$ be the affine Weyl group of type~$\widetilde{A}_n$ with simple roots $\Delta = \{ \alpha_1, \ldots, \alpha_n \}$, and define the convex cone $K \coloneqq \mathrm{cone}(\Delta)$.

Let $\operatorname{V}(A_{\mathrm{id}})$ be the vertex set of $A_{\mathrm{id}}$.
For any $x \in W$ and $v \in \operatorname{V}(A_{\mathrm{id}})$, let $v(x)$ denote the unique vertex of~$A_x$ such that the difference $v - v(x)$ lies in the coroot lattice, i.e.,
\begin{equation*}
    v - v(x) \in \sum_{i=1}^n \mathbbm{Z} \alpha_i.
\end{equation*}
The next proposition is proved in Section~\ref{subsec: results in An tilde}.
\begin{manualproposition}{B}\label{prop: An intro}
    Let $W$ be the affine Weyl group of type~$\widetilde{A}_n$.
    Then:
    \begin{equation*}
        z \in [x, y] \cap C_+ \  \iff \  v(z) \in \overline{C_+} \cap \left( v(x) + K \right) \cap \left( v(y) - K \right) \  \text{for all } v \in \operatorname{V}(A_{\mathrm{id}}).
    \end{equation*}
\end{manualproposition}
Although this proposition is not as strong as Theorem~\ref{thm: A} in the $\widetilde{A}_2$ case, it captures the essence of a potential generalization of Theorem~\ref{thm: A} to affine type~$A$ and possibly beyond.
For other types, we expect the situation to be more subtle, yet still governed by rich Euclidean-geometric structures.
  
\subsection{Classification of Bruhat intervals} 
We now turn to the classification of Bruhat intervals up to poset isomorphism (Problem~(\ref{problem: class})).
This question has been studied in the context of finite Weyl groups (see, e.g.,~\cite{Dy91, hul03, HulPhd2003, Jan79}; see also~\cite[Section 2.8]{BB05} for a review).
In the symmetric group~$S_6$, for instance, researchers have identified 24 distinct isomorphism classes of intervals of length at most~4, and 36 classes of intervals of length at most~5.
The main challenge with this approach is the apparent lack of structure or pattern, which renders a full classification seemingly intractable.

Rather than pursuing this classification problem in finite Weyl groups, we focus on affine Weyl groups---specifically, on those intervals that lie in the ``complement'' of the finite Weyl group embeddings within the affine Weyl group.

Returning to the affine Weyl group~$\widetilde{A}_2$, we partition the plane into six colored zones (gray, orange, purple, green, yellow, and blue), each associated with an element of~$W_f$ (Figure~\ref{subfig: star and Fw zones intro}).
Let $\mathrm{Gray}\subset \mathbbm{R}^2$ be the gray zone.
For $x,y$ \emph{dominant elements} of~$W$ (i.e., with $A_x,A_y\subset C_+$) define the polygon 
\begin{equation*}
    \Pgn_{x,y}\coloneqq \mathrm{Gray}\cap \left( \mathrm{cen}(x) + \mathrm{cone}(\alpha_1, \alpha_2) \right) \cap \left( \mathrm{cen}(y) - \mathrm{cone}(\alpha_1, \alpha_2)\right).
\end{equation*}
We focus on intervals~$[x, y]$ that are \emph{thick}, meaning they satisfy the condition $x + 2\alpha_i,\, y - 2\alpha_i \in \operatorname{Gray} \cap\, [x,y]$ for each $i \in \{1,2\}$.
This excludes ``thin'' intervals arising from $\widetilde{A}_1$, whose poset structures fail to reflect the geometry of the polygon~$\Pgn_{x,y}$.
In contrast, for thick intervals, the polygon~$\Pgn_{x,y}$ is determined---up to isometry---by the poset structure of $[x,y]$. 
\begin{manualtheorem}{C}\label{thm-main intro}
    Let $W$ be the affine Weyl group of type~$\widetilde{A}_2$ and let $x,x',y,y'$ be dominant elements with $[x,y]$ and $[x',y']$ thick.
    Then $[x,y]$ and $[x',y']$ are isomorphic as posets if and only if $\Pgn_{x,y}$ and $\Pgn_{x',y'}$ are isometric.
\end{manualtheorem}
\begin{rem}\label{rem: euclideanalcovic}
    We view Theorems~\ref{thm: A} and~\ref{thm-main intro} as strong evidence for an emerging philosophy: that Bruhat intervals in affine Weyl groups are deeply connected with Euclidean geometry.
    Further manifestations of this idea appear in~\cite{CdLP23, CdLP25}, where the authors explain---partially conjecturally---that the cardinalities of most lower Bruhat intervals in affine Weyl groups are linear combinations of the volumes of the faces of certain polytopes.
    Similarly, in~\cite{BGH23}, it is shown that the volumes of hyperplane sections of these polytopes---when intersected with $C_+$---approximate the length-counting sequences of lower intervals (of dominant elements) intersected with $C_+$.
    In both cases, a convex polytope governs the number of elements in a given Bruhat interval, providing partial answers to Problem~(\ref{problem: structure of Bruhat intervals}).
\end{rem}
We now provide a visual explanation of the isomorphisms appearing in Theorem~\ref{thm-main intro}, which fall into two distinct categories.
The first type is classical, arising from Coxeter theory~\cite[Corollary 2.3.6]{BB05}, and is induced by an automorphism of the Dynkin diagram.
Geometrically, this corresponds to a reflection (or ``flip'') across the vertical axis through the origin.
The second type is new and more subtle: it consists of piecewise translations.
We illustrate this phenomenon with two examples.

Figure~\ref{subfig: star and Fw zones intro} shows the plane partitioned into six distinct “zones”: $F_\mathrm{id}, F_{s_1}, F_{s_1s_2}$, and so on.
Each zone is assigned a specific color---gray, orange, yellow, etc.---for clarity.
When $\Pgn_{x,y}$ is a parallelogram, we can translate it freely within the gray zone, say by a weight~$\lambda$; in the zone~$F_w$, the corresponding figure is translated by the weight~$w(\lambda)$.
This is illustrated in Figure~\ref{subfig: traslacion paralelogramo intro}, which depicts the poset isomorphism $[x, y] \mapsto [x + \lambda, y + \lambda]$, where $\lambda = 5\varpi_1 + \varpi_2$, and $x, y$ are the alcoves shown in the figure.

In contrast, when $\Pgn_{x,y}$ is a pentagon with a side parallel to~$\varpi_i$ for some $1 \leq i \leq 2$ (as in Figure~\ref{subfig: traslacion pentagono intro}), the allowed piecewise translations are more constrained: the weight~$\lambda$ must be a multiple of~$\varpi_i$.
This is again illustrated by a piecewise translation representing the poset isomorphism $[x, y] \mapsto [x + \lambda, y + \lambda]$, where $\lambda = 5\varpi_1$, and $x, y$ are the alcoves shown in Figure~\ref{subfig: traslacion pentagono intro}.

Finally, when $\Pgn_{x,y}$ is a hexagon, there exist no piecewise translations that induce poset isomorphisms.
\begin{figure}[H]
    \centering
    \begin{subfigure}{.38\linewidth}
        \includegraphics[width=\linewidth]{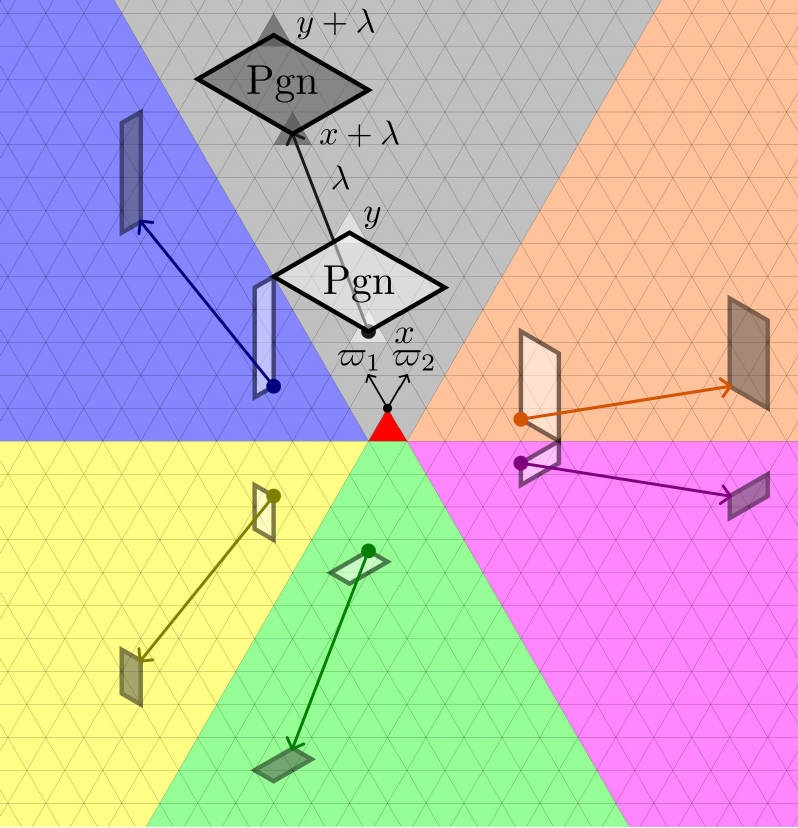} 
        \caption{Piecewise translation of a parallelogram.}
        \label{subfig: traslacion paralelogramo intro}
    \end{subfigure}
    \hspace{0.7cm} 
    \begin{subfigure}{.38\linewidth}
        \includegraphics[width=\linewidth]{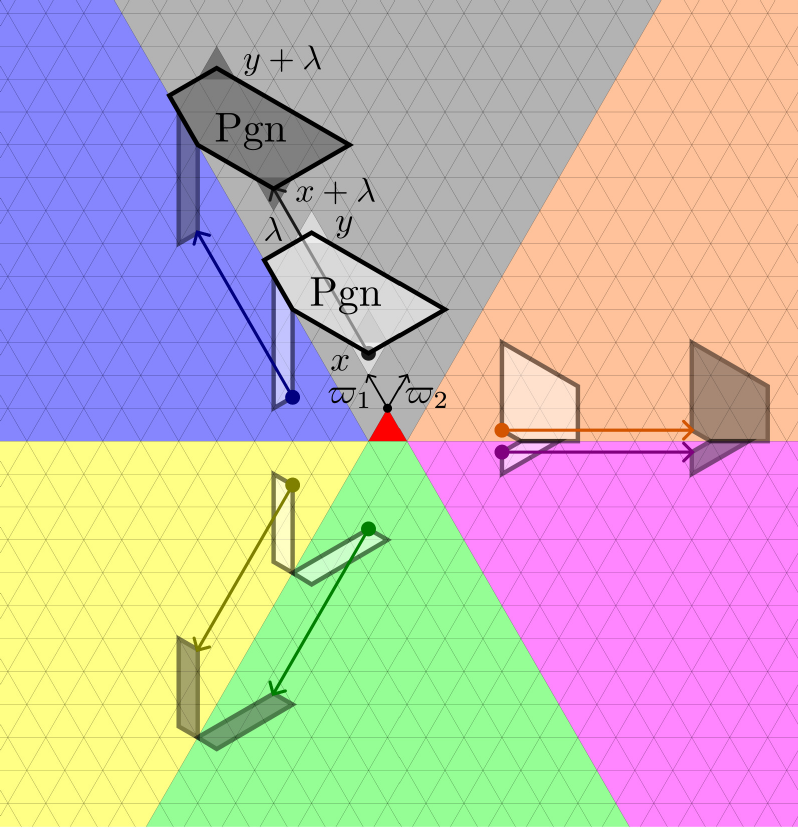} 
        \caption{Piecewise translation of a pentagon.}
        \label{subfig: traslacion pentagono intro}
    \end{subfigure}
    \caption{}
    \label{fig: translations intro}
\end{figure} 

\begin{rem}
    Let $\lessdot$ denote the covering relation in the Bruhat order.
    We say that an interval $[x, y]$ is \emph{full} if, for any $z \in W$ satisfying $x \lessdot z$ or $z \lessdot y$, one has $z \in [x, y]$.
    Such intervals are fairly generic---for instance, every thick interval in~$\widetilde{A}_2$ is full (Lemma~\ref{lem: thick implies full}).
    In Conjecture~\ref{conj: a2general}, we propose an analog of Theorem~\ref{thm-main intro} for arbitrary full intervals $[x, y]$, where $x$ and $y$ are not assumed to be dominant.
    This conjecture closely mirrors Theorem~\ref{thm-main intro} in that the only nontrivial poset isomorphisms arise from suitably defined piecewise translations.
\end{rem}
In Section~\ref{subsec: results in An tilde}, we will partially generalize this result to $\widetilde{A}_n$ when restricted to $C_+$.
This is proved in Proposition~\ref{prop: dominant translations in An tilde}.
\begin{manualproposition}{D}\label{prop: dominant translations in An tilde intro}
    Let $W$ be the affine Weyl group of type~$\widetilde{A}_n$, and let $x$ and $y$ be dominant elements.
    Suppose there exists $i \in \{1, \ldots, n\}$ such that
    \begin{equation*}
        (\alpha_i, v) \neq 0 \quad \text{for every vertex } v \text{ of every dominant alcove } A_z \text{ with } z \in [x, y].
    \end{equation*}
    Then the translation by $\varpi_i$ induces a poset isomorphism
    \begin{equation*}
        [x, y] \cap C_+ \cong [x + \varpi_i, y + \varpi_i] \cap C_+.
    \end{equation*}
\end{manualproposition}

This proposition suggests a broader phenomenon, which we now explain.
Every Coxeter system \((W,S)\) has a well-understood group \( G \) of automorphisms that preserve the Bruhat order ~\cite[Corollary~2.3.6]{BB05}.
This group is generated by the inversion map \( x\mapsto x^{-1} \) and the automorphisms induced by the symmetries of the Dynkin diagram.

Let \( W \) be an affine Weyl group, and let \( x, y \in W \).
We write \( y - x \in \Lambda^\vee \) if there exists a coweight \( \lambda \in \Lambda^\vee \) such that \( A_x + \lambda = A_y \).
\begin{manualconjecture}{E}\label{conj: weak conj intro}
    Suppose \( x, x', y, y' \in W \) are dominant elements such that the intervals \( [x, y] \) and \( [x', y'] \) are full and isomorphic as posets.
    Assume further that there exists \( g_0 \in G \) such that \( x - g_0x' \in \Lambda^\vee \).
    Then there exists \( g \in G \) such that
    \begin{equation*}
        y - g y' = x - g x' \in \Lambda^\vee.
    \end{equation*}
\end{manualconjecture}
In Section~\ref{subsec: weak conjecture}, we describe the extent of our computational evidence supporting the conjecture.
Note that by Theorem~\ref{thm-main intro}, this conjecture is established for thick intervals in type~$\widetilde{A}_2$.

Conjecture~\ref{conj: weak conj intro} essentially asserts that, modulo the ``easy'' poset isomorphisms common to all Coxeter groups, any poset isomorphism between two Bruhat intervals must restrict to a translation on the endpoints $\{x, y\}$.

Ideally, we would have liked to formulate a stronger conjecture: under the same hypotheses, there exists $g\in G$ and a piecewise translation $PT$ such that $ g\circ PT$ is a poset isomorphism from $[x, y]$ to 
$[x', y']$.
However, two main obstacles prevent us from formulating this stronger conjecture.
First, for general affine Weyl groups, we do not yet know the precise analog of the six zones for arbitrary elements, which makes it difficult to define a piecewise translation in this broader setting.
Second, verifying such a statement computationally is considerably more difficult.
That said, within the dominant cone, translations are well understood, so the first issue does not arise when restricting to this region.

\subsection{Applications and insights}
\subsubsection{Combinatorial invariance conjecture}
The Combinatorial Invariance Conjecture~(CIC), introduced by Lusztig and Dyer in the 1980s, asserts that if two Bruhat intervals (potentially from different Coxeter groups) are poset isomorphic, their Kazhdan--Lusztig polynomials must be equal.
This conjecture has divided expert opinion, largely due to the geometric interpretation of Kazhdan--Lusztig polynomials as dimensions of intersection cohomology spaces.
In analogous geometric settings, this type of statement is false, which makes the conjecture particularly surprising\footnote{The analogous statement is that the poset defining the stratification determines the dimensions of the corresponding local intersection cohomology spaces.}.

An \emph{idée-force} of this paper is that the CIC holds for what might be called ``silly reasons''---namely, that poset isomorphisms between Bruhat intervals are exceedingly rare.
\begin{rem}
    The most impressive evidence in favor of the CIC comes from its validity on lower intervals~\cite{dC03, BCM06}.
    In this direction, Proposition~\ref{prop: main lower} shows that for~$W$ of type~$\widetilde{A}_2$ and any elements $x, y \in W$, the intervals $[\mathrm{id}, x]$ and $[\mathrm{id}, y]$ are poset isomorphic if and only if $x \in G \cdot y$.
    We stress that this proposition applies to \textbf{arbitrary} elements $x$ and $y$ in the group.
\end{rem} 
When the hypotheses of Theorem~\ref{thm-main intro} are satisfied, we prove the CIC in Proposition~\ref{prop: CIC for thick}.
This result is not new---indeed, in~\cite{BLP23}, the first and second authors, together with David Plaza, established the CIC for the entire group.
What is novel here is not the result itself but the method.
We show that the longest dihedral subintervals starting at $x$ give approximately the Euclidean distances in~$\Pgn_{x,y}$ from $x$ to its nearest vertices.
A similar statement holds for $y$.
The key word here is ``approximately''---not ``exactly''---and this subtlety is the main reason this paper spans over 60~pages instead of just 30.
Nevertheless, we believe that the underlying phenomenon holds for all affine Weyl groups, suggesting that the methods developed here are far more generalizable than those used in~\cite{BLP23}.

Another notable difference is the nature of the argument: the proof in this paper is conceptually transparent---we explicitly describe the set of intervals isomorphic to a given one and directly verify that their Kazhdan–Lusztig polynomials coincide.
In contrast, the approach in~\cite{BLP23} relies on a collection of carefully chosen invariants, making it more opaque despite its elegance.

\subsubsection{Stabilization of Bruhat intervals}\label{ss: stabilization}
One of the main observations of this paper is that the polygon $\Pgn_{x,y}$ determines the poset structure of the interval $[x,y]$.
If an analogous statement were true in higher rank---namely, that a corresponding polytope $\operatorname{Ptp}_{x,y}$ encodes the structure of $[x,y]$---then, using the fact that this polytope stabilizes when translated sufficiently far from the origin (see Lemma~\ref{lem: Pgn stabilizes} for a proof in $\widetilde{A}_2$), it would follow that the interval $[x,y]$ also stabilizes.
In other words, we conjecture that for dominant coweights $\lambda$, the interval $[x + n\lambda, y + n\lambda]$ stabilizes as $n \to \infty$ (see Conjecture~\ref{conj: stabilization of dominant Bruhat intervals}).
We prove this conjecture in the case of $\widetilde{A}_2$ (Proposition~\ref{prop: dominant Bruhat intervals stabilize in A2}).

We note that if both this conjecture and CIC hold, they would together imply a known result: that the Kazhdan--Lusztig polynomial $P_{x + n\lambda, y + n\lambda}(q)$ eventually stabilizes as~$n \to \infty$ (see~\cite{Lus80, Kat85}).

\subsubsection{}\label{ss: dihedral}
We reserve the end of this introduction for the most speculative part of the paper.
As mentioned earlier, in the case of~$\widetilde{A}_2$, we observed that when $[x,y]$ is a full interval, the set $D_x$---consisting of elements $z \in [x,y]$ such that $[x,z]$ is a dihedral interval (i.e., isomorphic to an interval in some dihedral group)---almost recovers the  Euclidean distances from $x$ to its two nearest vertices in the polygon $\Pgn_{x,y}$.
A similar statement holds for $y$ and the similarly defined set $D_y$.
Remarkably, the subposet $D_x \cup D_y \subset [x,y]$ almost completely determines the isometric type of~$\Pgn_{x,y}$.
Once again, the crucial nuance lies in the word ``almost.''

What makes this insight particularly striking is that the following double implication, though not strictly true, comes surprisingly close:
\begin{manualinsight}{F}\label{insight: dihedral}
    $[x,y] \cong [x',y'] \iff D_x \cup D_y \cong D_{x'} \cup D_{y'}$.
\end{manualinsight}
If this equivalence held, the CIC would shift from a remarkable and mysterious conjecture to one with practical computational consequences---potentially providing a pathway to directly compute Kazhdan--Lusztig polynomials via combinatorial data.
This is because the information carried by the full interval $[x,y]$ is absurdly richer than that contained in the comparatively microscopic subposet $D_x \cup D_y$.
In all the examples we have examined so far, either the equivalence holds exactly, or---when it doesn’t---the intervals $[x,y]$ and $[x',y']$ have nearly identical invariants (cardinality, Kazhdan--Lusztig polynomials, Betti numbers, etc.), providing an excellent approximation.

One possible (though not necessarily optimal) conjectural fix is the following.
Let $x, x' \in C_+$ or $x, x' \in s_0 C_+$, and let $y, y' \in C_+$ or $y, y' \in s_0 C_+$.
Consider the set $\geq x\coloneqq \{ z \in W \mid z \geq x \}$ of elements greater than or equal to $x$ in the Bruhat order.
Then
\begin{equation}
    [x,y] \cong [x',y'] \text{ and both intervals are full} \iff 
    \begin{cases}
        D_x \cup D_y \cong D_{x'} \cup D_{y'}, \\
        W \setminus (\geq x) \not\subset [\mathrm{id}, y], \\
        |[x,y]| = |[x',y']|.
    \end{cases}
\end{equation}
We have verified this conjecture extensively in type~$\widetilde{A}_2$, for pairs of full intervals $[x, y]$ with $\ell(x, y) \leq 9$ and $9 \leq \ell(x) \leq 12$.
Additional checks include intervals with $\ell(x, y) \leq 8$ and $16 \leq \ell(x) \leq 20$, as well as intervals with $21 \leq \ell(x) \leq 22$ and $\ell(x, y) \leq 7$.
In type~$\widetilde{A}_3$, due to computational constraints, we have only tested full intervals with $\ell(x, y) \leq 8$ and $9 \leq \ell(x) \leq 12$.
Altogether, these computations cover more than 500 full intervals, and around $\binom{100}{2}$ pairs among them.

\subsection{Overview of the Paper}
Below, we summarize the contents of each section.
\begin{itemize}
    \item In Section~\ref{section: prelim}, we collect relevant definitions and notations, and review basic facts about the affine Weyl group of type~$\widetilde{A}_2$.
    We conclude the section by proving Proposition~\ref{prop description of all lower intervals}, which provides a geometric description of every lower interval in~$\widetilde{A}_2$.
    
    \item In Section~\ref{section: geometry of intervals}, we give a geometric description of the set $\geq x$ for dominant elements~$x$ (Proposition~\ref{prop: no mayores geometrico}).
    From this geometric perspective, we derive several local properties of the Bruhat order when restricted to the zones of the form $F_w$ for $w\in W_f$.
    In Proposition~\ref{prop: Dih-resumen}, we analyze the subsets $\DC^u(y)$ and $\DC^l(x)$ of~$W$, related to the subsets $D_x$ and $D_y$ of $[x,y]$, using this geometric framework.
    In Section~\ref{subsec: polygon}, we define the polygon $\operatorname{Pgn}_{x,y}$ associated with $[x,y]$, a key ingredient in Theorem~\ref{thm-main intro}.
    We conclude the section by proving Theorem~\ref{thm: A}, thereby obtaining a Euclidean geometric description of every Bruhat interval in~$\widetilde{A}_2$.
    \item In Section~\ref{section: translations}, we introduce the piecewise translation $\tau_\lambda$ on~$W$.
    The guiding philosophy from this point onward is that the polygon $\Pgn_{x,y}$ encodes combinatorial information about the poset $[x,y]$.
    After establishing some basic properties of the map $\tau_\lambda$, we prove Proposition~\ref{prop: cardinality under translation}, which provides conditions (formulated in terms of the polygon $\Pgn_{x,y}$) under which the posets $[x,y]$ and $\tau_\lambda([x,y])$ have the same cardinality.
    We then observe that whenever the map $\tau_\lambda$ restricts to a bijection on $[x,y]$, it induces a poset isomorphism (Proposition~\ref{prop: trasl-por-dominat}).
    \item In Section~\ref{section: sides}, we continue studying the relationship between Euclidean geometry (through the polygon $\Pgn_{x,y}$) and Bruhat order combinatorics (through the interval $[x,y]$).
    In the first part of the section, we prove Proposition~\ref{prop: length of all sides}, which gives an exact description of the Euclidean side lengths of~$\Pgn_{x,y}$ in terms of structural data associated with the interval~$[x,y]$, not all of which is purely combinatorial.
    In the second part, we explore the notion of thick intervals, which allows us to reduce the problem to a more manageable class of intervals---namely, those to which our proof of Theorem~\ref{thm-main intro} applies.
    The main result of the section is Proposition~\ref{prop: seg max fixed under iso}, which shows that poset isomorphisms between thick intervals are subject to strong Euclidean geometric constraints, related to the subsets $D_x$ and $D_y$.
    \item In Section~\ref{subsec: congruence}, we prove the invariance of~$\Pgn_{x,y}$ under a poset isomorphism (Proposition~\ref{prop: iso implies congruent polygons}). 
    The proof is lengthy and relies on several technical results. 
    The key idea is that a poset isomorphism $\phi\colon [x,y] \to [x',y']$ satisfies $\phi(x+\rho) = x' + \rho$, and, together with Proposition~\ref{prop: seg max fixed under iso}, induces a bijection $f_{\phi}\colon \mathrm{V}(\Pgn_{x,y}) \to \mathrm{V}(\Pgn_{x',y'})$ between the vertex sets of the polygons.  
    In Section~\ref{subsec: main result}, we prove the main theorem of the paper: Theorem~\ref{thm-main intro}.
    \item In Section~\ref{section: applications}, we present three applications of our results:  
    a non-recursive proof of the combinatorial invariance conjecture for thick intervals (Proposition~\ref{prop: CIC for thick});  
    a stabilization result for thick intervals (Proposition~\ref{prop: dominant Bruhat intervals stabilize in A2});  
    and a classification of lower intervals, where the isomorphisms are induced by~$G$ (Theorem~\ref{prop: main lower}).
    \item In Section~\ref{section: towards a general picture}, we explore possible extensions of our results to other affine Weyl groups.  
    In Section~\ref{subsec: weak conjecture}, we present computational evidence for Conjecture~\ref{conj: weak conj intro}.  
    We conclude the paper by proving Propositions~\ref{prop: An intro} and~\ref{prop: dominant translations in An tilde intro}, which partially generalize results from Sections~\ref{section: geometry of intervals} and~\ref{section: translations} to the affine Weyl group of type~$\widetilde{A}_n$.
\end{itemize}

\subsection*{Acknowledgements}
The first author is grateful to David Kazhdan for a year of stimulating weekly discussions related to this work.
We also thank Luis Arenas-Carmona, David Plaza, and Paolo Sentinelli for their detailed comments on a previous version of the paper.

%% file: Sections/Preliminaries.tex
\section{Preliminaries}\label{section: prelim}
\subsection{Basic notions in Coxeter systems}\label{subsec: Coxeter systems}
In this section, we recall some basic notions about Coxeter systems.
We mainly follow~\cite[Chapter 2]{BB05} for notation and terminology with some minor differences.

Let $(W,S)$ be a Coxeter system with length $\ell$ and Bruhat order $\leq$.
We denote $\ell(u,w)\coloneqq\ell(w)-\ell(u)$.
A \emph{Bruhat interval} is a set $[x,y]=\{z\in W\mid x\leq z\leq y\}$, with  $x,y\in W$.
Given an element $x\in W$, the \emph{upper set} of~$x$ is the set
\begin{equation*}
    \geq x \coloneqq\{z\in W\mid x\leq z\},
\end{equation*}
and the \emph{lower interval} of~$x$ is the set
$[\id,x]$.

If $u,z\in W,$ we say that $z$ \emph{covers} $u$, and we denote it by $u\lessdot z$, if $u<z$ and there is no element $z'$ such that $u<z'<z$.
If $x,y\in W$, we say that $z\in [x,y]$ is an  \emph{atom} (resp.\@ \emph{coatom}) of the interval  $[x,y]$ if $x \lessdot z$ (resp.\@ $z\lessdot y$).

Let $x,y\in W$, with $x\leq y$.
The \emph{length counting sequence} $\operatorname{LC}(x,y)$ of the interval $[x,y]$ is the finite sequence
\begin{equation*}
   \operatorname{LC}(x,y)= (f_0^{x,y},f_1^{x,y},\hdots,f_{\ell(x,y)}^{x,y})\in \mathbbm{N}^{\ell(x,y)+1}
\end{equation*} 
given by $f_i^{x,y}:=|\{z\in [x,y]\mid \ell(x,z)=i\}|$.
\begin{lem}[\cite{BB05} pp.\@ 35]\label{lem-graded} 
    The interval $[x,y]$ is a graded poset, with the rank function given by the length function $z\mapsto \ell(x,z)$.
\end{lem}

\begin{lem}\label{lem: betti invariante} 
    Let $\varphi\colon[x,y]\to [x',y']$ be an isomorphism of posets.
    Then, $\varphi$ is an isomorphism of graded posets.
    In other words, it preserves the rank function or, in formulas, $\ell(x,z)=\ell(\varphi(x),\varphi(z))$.
    In particular, $\operatorname{LC}(x,y)=\operatorname{LC}(x',y')$.
\end{lem}
\begin{proof}
    It is clear that if $z$ {covers} $u$ then $\varphi(z)$ {covers} $\varphi(u)$.
    If $z\in [x,y]$ is such that $\ell(x,z)=k$, by~\cite[Lemma 2.2.1]{BB05} there is a chain $x=z_0\lessdot z_1\lessdot\cdots \lessdot z_k=z$, so applying $\varphi$ we obtain a chain $\varphi(x)=\varphi(z_0)\lessdot \varphi(z_1)\lessdot\cdots \lessdot \varphi(z_k)=\varphi(z)$.
    By Lemma~\ref{lem-graded} this implies that $\ell(\varphi(x),\varphi(z))=k$.
\end{proof}

An interval $[x,y]$ in~$W$ is \emph{dihedral} if it is isomorphic to an interval of a dihedral Coxeter group.
For each $y\in W$, we define $\mathcal{D}(y):=\{z\in [\id,y]\mid  [z,y]\text{ is dihedral}\}$.
\begin{prop}[{\cite[Theorem 7.25]{Dyer87}}]\label{prop: Dyer dihedral} 
    Let $x,y\in W$ with $\ell(x,y)>1$.
    The following statements are equivalent:
    \begin{enumerate}
        \item $[x,y]$ has two atoms;
        \item $[x,y]$ has two coatoms;
        \item $[x,y]$ is dihedral.
    \end{enumerate}    
\end{prop}
The following lemma is immediate from Proposition~\ref{prop: Dyer dihedral} and Lemma~\ref{lem: betti invariante}.
\begin{lem}\label{lem: iso-preserva-dih} 
    Let $\phi \colon [x, y] \to [x', y']$ be an isomorphism of posets.
    If $[u,v] \subset [x,y]$ is  dihedral, then $\phi([u,v])$ is dihedral.
\end{lem}
\begin{defn}\label{def-lower-cover-upper-cover}
    Given $y\in W$, we define the sets
    \begin{align*}
        \operatorname{L}(y):=&\{z\in W\mid z\lessdot y\,\}\\
        \operatorname{U}(y):=&\{z\in W\mid y\lessdot z\,\}
    \end{align*}
    of \emph{lower covers} of~$y$ and \emph{upper covers} of~$y$ respectively.
\end{defn}
\begin{defn}\label{def-full-interval}
    An interval $[x,y]$ is \emph{full} if $\operatorname{U}(x),\operatorname{L}(y)\subset [x,y].$ 
\end{defn}
As a consequence of the \textit{lifting property} of the Bruhat order~\cite[Proposition 2.2.7]{BB05}, the following result holds.
\begin{lem}\label{prop: intervals are union by right}
    Let $y \in W$ and $s\in S$ be such that $y<sy$.
    Then
    \begin{equation*}
        [\id,sy]=[\id, y]\cup s[\id,y],
    \end{equation*}
    where $s[\id,y]\coloneqq \{sx \mid x\leq y\}$.
\end{lem}

\subsection{The affine Weyl group of type~\texorpdfstring{$\widetilde{A}_2$}{A2~} and alcoves} \label{subsec: alcoves}
From this point forward, $W$ denotes the affine Weyl group of type~$\widetilde{A}_2.$ As a Coxeter system it has generators $S=\{s_0,s_1,s_2\}$.
For simplicity, we will often denote $s_0,s_1,s_2$ by $0,1,2$ respectively.
We will use  ``label mod~3'' notation.
For example, $1456$ stands for $s_1s_1s_2s_0$.

A \emph{reflection} of~$W$ is an element belonging to the set 
\begin{equation*}
    T=\bigcup_{z\in W} zSz^{-1}.
\end{equation*}
The finite subgroup $W_f\coloneqq\langle s_1,s_2\rangle\leq W$ is the \emph{Weyl group} of type~$A_2$.
It is isomorphic to the dihedral group of order $6$.
The longest element of~$W_f$ is denoted by $w_0=s_1s_2s_1$.

The Dynkin diagram of~$W$ has six symmetries (it can be seen as an equilateral triangle).
Each one of them induces an automorphism of~$W$.
Let $\delta\colon S\to S$ be the map $\delta(i)=i+1$, i.e., the map given by $\delta(s_0) = s_1, \delta(s_1) = s_2$, and $\delta(s_2) = s_0$.
Similarly, we consider $\sigma\colon S \to S$ the map that fixes $s_0$ and permutes $s_1$ and $s_2$.
The automorphisms of~$W$ induced by~$\delta$ and~$\sigma$ will be denoted by the same symbols.
The subgroup $\langle \delta, \sigma\rangle $ of~$\Aut(W)$ generated by~$\delta$ and $\sigma$ is isomorphic to the dihedral group of order $6$.
Let $\iota\colon W \to W$ be the inversion group anti-automorphism that sends $x\mapsto x^{-1}$.
We define the group~$G$ of \emph{generic automorphisms} as the subgroup of the symmetric group of~$W$ generated by $\delta,\sigma$, and $\iota$.
The order of~$G$ is $12$.
\begin{lem}[{\cite[Proposition 8.8]{Hum90}}]\label{lem: elements of G are isomorphism of poset}
    Suppose $g\in G$.
    The map $z\mapsto g(z)$ induces an isomorphism of posets $[x,y]\cong [g(x), g(y)]$.
\end{lem}
Let $(E, (-,-))$ be the real Euclidean plane and let $\Phi\subset E$ be the root system of type~$A_2$, with simple roots $\{\alpha_1,\alpha_2\}$ spanning the root lattice $\mathbbm{Z}\Phi,$ and fundamental weights $\{\varpi_1,\varpi_2\}$ spanning the weight lattice $\Lambda$.
The elements of~$\Lambda$ are called \emph{weights}.
A \emph{dominant weight} is a weight that has non-negative coefficients when written as a linear combination of fundamental weights.
Let $\Phi_+=\{\alpha_1,\alpha_2,\alpha_1+\alpha_2\}$ be the  set of positive roots.
For each $\alpha\in \Phi$ and $k\in \mathbbm{Z}$ consider the affine line
\begin{equation*}
    H_{\a,k}\coloneqq\{v\in E\mid (\a, v)=k\}.
\end{equation*}
Let $s_{\a,k}$ be the orthogonal reflection through $H_{\a,k}$.
We denote by $\mathcal{H}$ the collection of all affine lines $H_{\a,k}$, with $\a\in\Phi,k\in\mathbbm{Z}$.
\begin{defn}\label{def: alcoves}
    The set $\mathcal{A}$ of \emph{alcoves} of~$\widetilde{A}_2$ is the set of connected components of 
    \begin{equation*}
        E-\bigcup_{H\in \mathcal{H}} H
    \end{equation*}
    The \emph{fundamental alcove} is  the following open subset of~$E$ 
    \begin{equation}\label{eq: fundamental eq}
        A_+\coloneqq\{v\in E\mid -1< ( \a,v)< 0 \text{ for all } \alpha\in \Phi_+\}\in \mathcal{A}.
    \end{equation}
\end{defn}
Let us recall a classical result.
\begin{prop}[{\cite[Chapitre VI, $\mathsection{2}$, $n^o \,3$]{Bou68}}]\label{prop: bijection between group and alcoves}
    Let $\mathcal{A}$ be the set of alcoves.
    The map $w\mapsto w\cdot A_+$ gives a one-to-one correspondence between $W$~and~$\mathcal{A}$.
\end{prop}
The red triangle in Figure~\ref{Fig: teselado} is the fundamental alcove $A_+$ of~$\widetilde{A}_2$.
Each little triangle in the figure is an alcove.
Via Proposition~\ref{prop: bijection between group and alcoves}, we identify $A_+$ with the identity $\mathrm{id}\in W$.
The colors on the edges of the triangles represent the following reflections: blue is $s_0$, green is $s_1$, and red is $s_2$.
\begin{figure}[!ht]
    \centering
    \begin{subfigure}[b]{0.48\textwidth}
        \centering
        \includegraphics[width=0.9\textwidth]{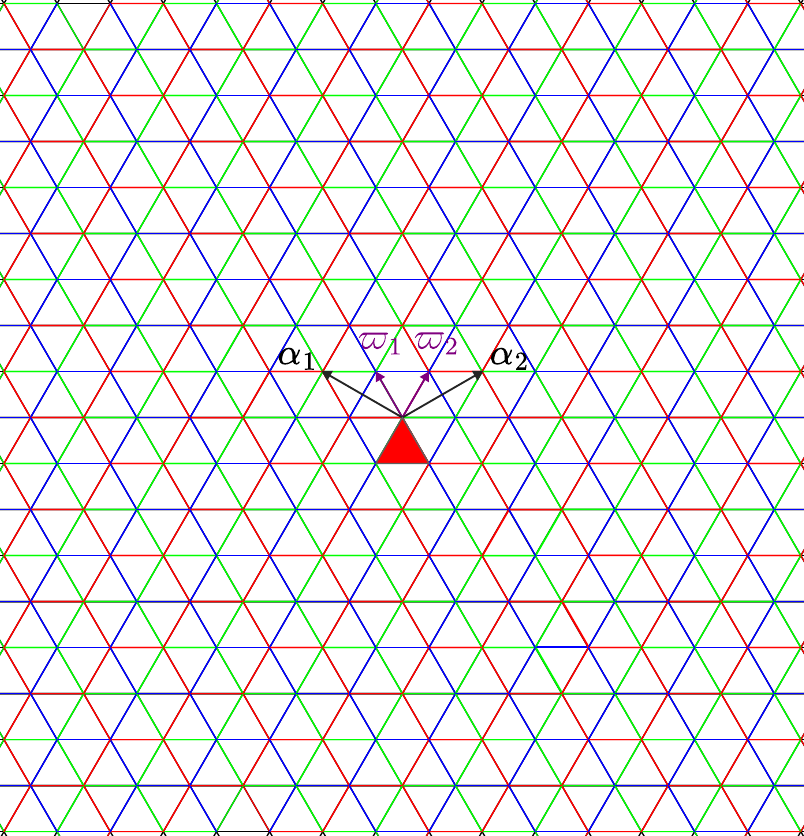}
        \caption{The Coxeter complex of~$W$.}
        \label{Fig: teselado}
    \end{subfigure}
        \hfill
    \begin{subfigure}[b]{0.48\textwidth}
        \centering 
        \includegraphics[width=0.9\textwidth]{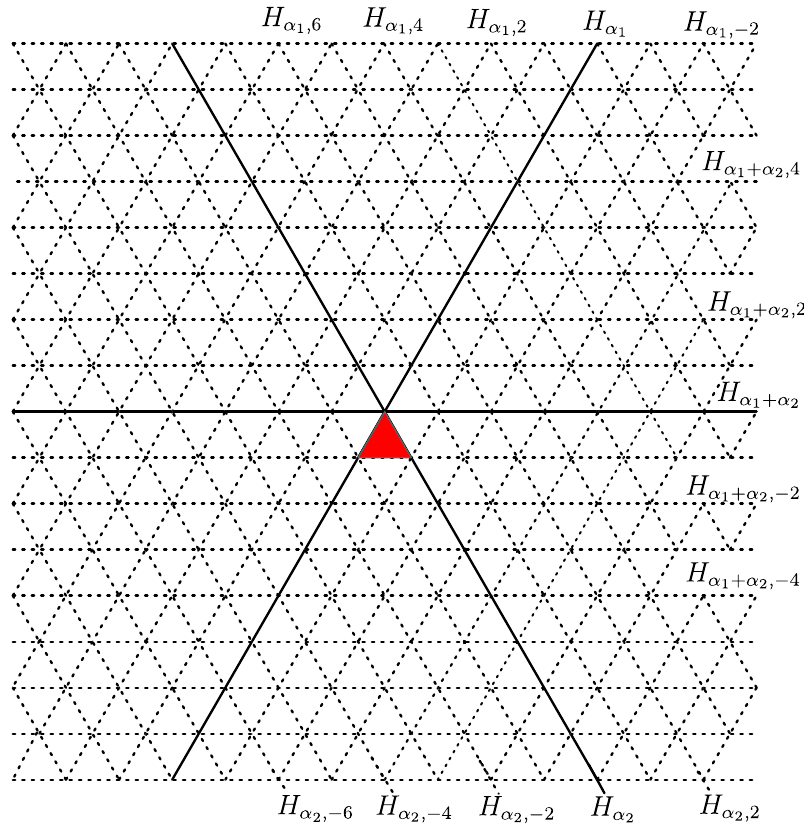}
        \caption{The set $\HC$ of affine lines.}
        \label{fig: Teselado con raiz}
    \end{subfigure}
    \caption{}
    \label{Fig: Raices e hiperplanos}
\end{figure}

We define the function $\operatorname{cen}\colon W\to E$, which maps an element $z\in W$ to the \emph{center} of its alcove $zA_+$.
When discussing the vertices, sides, faces, or area of an element $z\in W$, we refer to the corresponding property on the closure of its alcove, $\overline{zA_+}$.
To avoid cumbersome notation, for $z\in W$ and a point $v\in E$ in the plane, we write $v\in z$ as a shortcut for $v\in \overline{zA_+}$.
Since the left action of~$W$ on~$\AC$ is simply transitive, the equality $xA_+=yA_+$ implies that $x=y$.
In particular, $\cen(x)=\cen(y)$ implies that $x=y$.
An element $z\in W$ can be seen as a bijection $z\colon E\to E$, and its restriction $z\vert_{\cen(W)}\colon \cen(W)\to z(\cen(W))=\cen(W)$ is a bijection.

For any weight $\lambda\in\Lambda$, it is not hard to see that $\cen(W)+\lambda=\cen(W)$.
Indeed, it is enough to check this for $\lambda=\varpi_1$ (by symmetry, this implies the case $\lambda=\varpi_2$), and one can check this directly in Figure~\ref{Fig: teselado}.

We say that two alcoves $A$ and $B$ have the same \emph{orientation} if $A$ can be obtained from $B$ by a translation.
If $A$ has the same orientation as $A_+$, we say it is \emph{up-oriented}; otherwise, we say \emph{down-oriented}.
The \emph{orientation} of an element $z\in W$ is the orientation of its alcove~$zA_+$.

\begin{lem}\label{lem: length determines orientation} Let~$z\in W.$
    \begin{itemize}
        \item If $\ell(z)$ is odd, then $z$ is down-oriented.
        \item If $\ell(z)$ is even, then $z$ is up-oriented.
    \end{itemize}
\end{lem}

\begin{proof}
    It is standard that for $z\in W$, the length $\ell(z)$ measures the minimal number of affine lines of the form $H_{\alpha,k}$ that one must cross to go from the identity alcove to the alcove~$zA_+$.
    Each crossing changes the orientation, allowing one to conclude by induction on~$\ell(z)$.
\end{proof}

\begin{nota}\label{nota: centroide mas vector}
    For $z\in W$ and a vector $v\in \Lambda$, we denote by $z+v$ the unique element in~$W$ with center $\operatorname{cen}(z)+v$.
    In other words, $\operatorname{cen}(z+v) = \operatorname{cen}(z)+v$.
\end{nota}

\subsection{The \texorpdfstring{$X\Theta$-partition}{XTheta-partition} of \texorpdfstring{$\widetilde{A}_2$}{A2~}}\label{subsec: xtheta partition}
In this section, we follow~\cite{BLP23} and~\cite{LP23} with slight modifications to their original notation.

\begin{defn}
    Let $\underline{y} = (r_1, r_2,\cdots , r_k)$ be any expression of~$y$ (not necessarily reduced).
    If $1<i<k$, we say that there is a \emph{braid triplet in position} $i$ if $r_{i-1}= r_{i+1} \neq r_i$.
    We define the \emph{distance} between a braid triplet in position $i$ and a braid triplet in position $j>i$ to be the number $j-i-1$.
\end{defn}

\begin{lem}[{\cite[Lemma 1.1]{LP23}}]\label{lem: reduced expresion}
    An expression of an element in~$W$ without adjacent simple reflections is reduced if and only if the distance between any two braid triplets is~odd.
\end{lem}
 
For $n\in \mathbbm{Z}_{\geq 1}$, consider the element $\mathrm{\mathbf{x}}_n\coloneqq 123\cdots n$.
Recall that this denotes the product $s_1s_2s_3\cdots s_n$, where the indices are considered modulo~3.
By Lemma~\ref{lem: reduced expresion}, the expression $(1,2,\hdots,n)$ is reduced, as there are no braid triplets, 
so $\ell(\mathrm{\mathbf{x}}_n)=n$.
Recall the group $G$ from Section~\ref{subsec: alcoves}.
Let us define 
\begin{equation}
    X^{\mathrm{odd}} \coloneqq \{g(\mathrm{\mathbf{x}}_n)\mid\mbox{$n$ is odd, $g\in G$}\} \quad\mbox{and}\quad  X^{\mathrm{even}}\coloneqq \{g(\mathrm{\mathbf{x}}_n)\mid\mbox{$n$ is even, $g\in G$}\}, 
\end{equation} 
and we denote $X\coloneqq X^\mathrm{odd}\uplus X^\mathrm{even}$, where $\uplus$ is the disjoint union.

For $m,n\in \mathbbm{Z}_{\geq 0}$, we define 
\begin{equation}\label{theta}   
    \theta(m,n)\coloneqq 1234\cdots(2m+1)(2m+2)(2m+1)\cdots(2m-2n+1)
\end{equation} 
Since $(2m+1)(2m+2)(2m+1)$ is the unique braid triplet, Lemma~\ref{lem: reduced expresion} implies that the expression in \eqref{theta} is reduced.
In particular, we have that $\ell(\theta(m,n))=2m+2n+3$.
We define
\begin{equation*}
    \Theta\coloneqq \{g(\theta(m,n))\mid m,n\geq0, g\in G\}.
\end{equation*}
The elements in this set are precisely those elements in~$W$ having two simple reflections in both their left and right descent sets.
In particular, for any $\theta(m,n)$ there exists a unique $s_{m,n}\in S$ such that 
\begin{equation}\label{eq-smn}
    \theta(m,n)<\theta(m,n)s_{m,n}.
\end{equation}
To simplify formulas, we will use the notation $\theta^s(m,n)\coloneqq \theta(m,n)s_{m,n}$.
On the other hand, $s_0$ is the only simple reflection not included in the left descent set of~$\theta(m,n)$.
We define, 
\begin{align*}
    \Theta^s&\coloneqq\{g(\theta^s(m,n))\mid m,n\geq0, g\in G\},\\
    {}^s\Theta^s&\coloneqq \{g(s_0\theta^s(m,n))\mid m,n\geq0, g\in G\}.
\end{align*}
Note that $\ell(\theta^s(m,n))=2m+2n+4$ and $\ell(s_0\theta^s(m,n))=2m+2n+5$.
\begin{rem}\label{rem: facts of theta-partition}
    \begin{enumerate}[label=(\roman*)]
        \item\label{remark-item: length of the elements} The elements in 
        $X^{\mathrm{even}}\uplus \Theta^s$ have even length, while the elements in 
        $X^{\mathrm{odd}}\uplus \Theta\uplus {^s\Theta^s}$ have odd length.
        \item\label{rem-item: difference of thetas as combination of fundamental weights} The element   $\theta(m,n)$ can be defined as the unique element in the affine Weyl group, such that $\theta(m,n)A_+=w_0A_++m\varpi_1+n\varpi_2.$
        Thus, for $m,m',n,n'\geq 0$ and $v=(m'-m)\varpi_1+(n'-n)\varpi_2\in\Lambda$, we have $\theta(m,n)+v=\theta(m',n')$.
        This last equation implies that $\theta^s(m,n)+v=\theta^s(m',n')$, because $\cen(\theta^s(m,n))=\cen(\theta(m,n))+(\varpi_1+\varpi_2)/3$.
        \item\label{rem-item: difference of wall elements as combination of fundamental weights} For $k,k'\geq 0$ and $v=(k'-k)\varpi_1\in\Lambda$, we have $\mathrm{\mathbf{x}}_{2k}+v=\mathrm{\mathbf{x}}_{2k'}$ and $\mathrm{\mathbf{x}}_{2k+1}+v=\mathrm{\mathbf{x}}_{2k'+1}$.
    \end{enumerate}
\end{rem}

The \emph{$X\Theta$-partition} of~$W$ is the decomposition $W\setminus\{\id\}=X\uplus\Theta\uplus\Theta^s\uplus {^s\Theta^s}$.
This last equality is illustrated in Figure~\ref{Fig: Theta-p}.
\begin{figure}
    \includegraphics[width=0.5\textwidth]{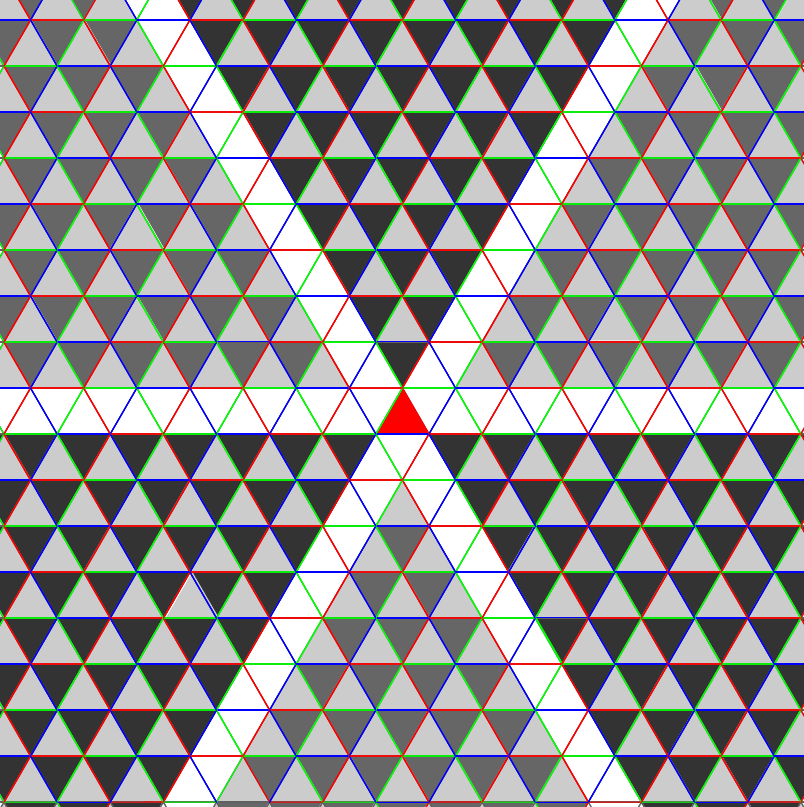}
    \caption{From lightest (white) to darkest gray:  $X,\Theta^s,{^s\Theta^s}$, and $\Theta$.
    The red triangle is $A_+$.}
    \label{Fig: Theta-p}
\end{figure}
\begin{defn}
    We say that an element $x\in W$ is \emph{dominant} if its alcove has a non-empty intersection with the dominant chamber (the light blue region in Figure~\ref{subfig: hexagon 1}).
    Equivalently, $x$ is dominant if it belongs to the set $D\coloneqq\{\theta(m,n),\theta^s(m,n)\mid m,n\geq 0\}.$ 
\end{defn} 
\begin{rem}\label{rem: remark tontos}
    For any $y\in W\setminus\{\id \}$, there is an element $g\in \langle \delta, \sigma\rangle$ such that 
    \begin{equation*}
        gy\in \{\theta(m,n),s_0\theta(m,n),\theta^s(m,n),s_0\theta^s(m,n),\mathrm{\mathbf{x}}_k\mid m,n\geq 0,k\geq 1\}.
    \end{equation*}
    Such a $g$ is not necessarily unique for a fixed $y$.
    For example, $\sigma(\theta(3,4))=\theta(4,3)$.
\end{rem}
Using the equation $\ell(\mathrm{\mathbf{x}}_n)=n$ and Remark~\ref{rem: facts of theta-partition}\ref{rem-item: difference of wall elements as combination of fundamental weights}, we obtain the following lemma.
\begin{lem}\label{lem: difference of lengths for alcoves in the wall}
    If $x,y\in \{\mathrm{\mathbf{x}}_n \mid \mbox{$n$ is odd}\}$, then $\cen(y)-\cen(x)=({\ell(x,y)}/{2})\varpi_1$.
\end{lem}

\subsection{Upper and lower covers} 
In this section we study the sets $\operatorname{U}(z)$ and $\operatorname{L}(z)$ for $z\in D\cup s_0 D$.

For $y\in D$, we define the following subset of~$D$:
\begin{equation*}
    \operatorname{L}_D(y)\coloneqq 
    \left\{
    \begin{array}{@{}l@{\thinspace}l}
        \{\theta^s(m-1,n),\theta^s(m,n-1)\}&\mbox{ if $y=\theta(m,n)$,} \medskip\\
        \{\theta(m,n),\theta(m-1,n+1),\theta(m+1,n-1)\}&\mbox{ if $y=\theta^s(m,n)$.}
    \end{array}  
    \right.
\end{equation*}
If either $m-1<0$ or $n-1<0$, then we omit the corresponding element of~$\operatorname{L}_D(y)$.
For instance,  $\operatorname{L}_D(\theta^s(3,0))=\{\theta(3,0),\theta(2,1)\}$.

The following lemma describes the sets $\operatorname{L}(y)$ for $y\in D\cup s_0D$.
The case $y=s_0\theta^s(m,n)$ with $n,m$ positive integers is proven in~\cite[Lemma 2.2]{BLP23}.
The remaining cases follow similar arguments, so we omit the details.
\begin{lem} \label{lem: set of lower covers}
    Let $y\in D$, then  $\operatorname{L}(y)=\{s_1y,s_2y\}\uplus \operatorname{L}_D(y).$ In particular, $\operatorname{L}(y)\cap D=\operatorname{L}_D(y)$.
    Furthermore, we have that $\operatorname{L}(s_0y)=s_0\operatorname{L}(y)\uplus \{y\}$.
\end{lem}
An immediate consequence of Lemma~\ref{lem: set of lower covers} and the fact that $\operatorname{L}(gy)=g\operatorname{L}(y)$ for $g\in \langle\sigma,\delta\rangle$ is the following.
\begin{lem}\label{lem: cardinality of lower covers}
    Let $y\in W\setminus  (X\cup\{\id\})$, we have three cases:
    \begin{enumerate}
        \item If $y\in \Theta$, then $|\operatorname{L}(y)|\in\{2,3,4\}$.
        \item If $y\in \Theta^s$, then $|\operatorname{L}(y)|\in\{3,4,5\}$.
        \item If $y\in {^s\Theta^s}$, then $|\operatorname{L}(y)|\in\{4,5,6\}$.
    \end{enumerate}
\end{lem}
\begin{lem}\label{lem: cardinality of upper covers}
    Let $x\in W\setminus (X\cup\{\id\})$, we have three cases:
    \begin{enumerate}
        \item If $x\in \Theta$, then $|\operatorname{U}(x)|=6$.
        \item If $x\in \Theta^s$, then $|\operatorname{U}(x)|=5$.
        \item If $x\in {^s\Theta^s}$, then $|\operatorname{U}(x)|=4$.
    \end{enumerate}
\end{lem}
\begin{proof}
    To prove the statement, we distinguish two cases.
    First, suppose that $x\in D+\rho$ or $s_0x\in D+\rho$.
    Let $z\in \operatorname{U}(x)$.
    We claim that $z\not\in X$, which follows from the explicit description\footnote{This description was not introduced earlier, as it plays no direct role in the main results.
    Its proof is similar in spirit to that of Lemma~\ref{lem: set of lower covers}, and we omit it to avoid unnecessary repetition.} for $\operatorname{L}(\mathrm{\mathbf{x}}_k)$
    \begin{equation}\label{eq: set lower xk}
        \operatorname{L}(\mathrm{\mathbf{x}}_k) = 
        \begin{cases}
            \{\mathrm{\mathbf{x}}_{k-1},s_1\mathrm{\mathbf{x}}_k, \theta(\tfrac{k-4}{2},0),\delta^2(\theta(\tfrac{k-4}{2},0))\} & \text{if } k \text{ is even and } k \geq 4,\medskip\\
            \{\mathrm{\mathbf{x}}_{k-1},s_1\mathrm{\mathbf{x}}_k, \delta(s_0\theta(\tfrac{k-5}{2},0)),\delta^2(\theta^s(\tfrac{k-5}{2},0))\} & \text{if } k \text{ is odd and } k \geq 5.
        \end{cases}
    \end{equation}
    This implies that $\operatorname{L}(z)$ can be described using Lemma~\ref{lem: set of lower covers} and the fact that $\operatorname{L}(gz)=g\operatorname{L}(z)$ for $g\in \langle\sigma,\delta\rangle$.
    Let us illustrate this with one example.
    Let $z=s_0\theta(m-1,n+1)$.
    By Lemma~\ref{lem: set of lower covers}, we have that $s_0s_1\theta(m-1,n+1)\in \operatorname{L}(z)$, but we note that $s_0s_1\theta(m-1,n+1)=\delta^2(\theta(m,n))$.
    Thus we get $\delta(s_0s_1\theta(m-1,n+1))=\theta(m,n)\in \operatorname{L}(\delta z)$.
    
    Therefore, we obtain
    \begingroup
    \allowdisplaybreaks
    \begin{align*}
        \operatorname{U}(\theta(m,n)) &= 
            \{ \theta^s(m,n), \theta^s(m{+}1,n{-}1), \theta^s(m{-}1,n{+}1), \\
            &\phantom{=\{\;} s_0\theta(m,n), \delta(s_0\theta(m{-}1,n{+}1)), \delta^2(s_0\theta(m{+}1,n{-}1)) \}, \\[0.5ex]
        \operatorname{U}(\theta^s(m,n)) &= 
            \{ \theta(m{+}1,n), \theta(m,n{+}1), s_0\theta^s(m,n), \\
            &\phantom{=\{\;} \delta(s_0\theta^s(m{-}1,n{+}1)), \delta^2(s_0\theta^s(m{+}1,n{-}1)) \}, \\[0.5ex]
        \operatorname{U}(s_0\theta^s(m,n)) &= 
            \{ s_0\theta(m,n{+}1), s_0\theta(m{+}1,n), \\
            &\phantom{=\{\;} \delta(\theta^s(m,n{+}1)), \delta^2(\theta^s(m{+}1,n)) \}.
    \end{align*}
    \endgroup
    It is easy to see that the elements in each of these sets are mutually distinct, so the lemma follows for $x\in (D+\rho)\uplus s_0(D+\rho)$.
     
    Otherwise, suppose that $x\in D\setminus (D+\rho)$ or $s_0x\in D\setminus (D+\rho)$.
    This case follows by similar reasoning combining Lemma~\ref{lem: set of lower covers}, Equation \eqref{eq: set lower xk}, and the fact that $\operatorname{L}(gz)=g\operatorname{L}(z)$, for $g\in \langle\sigma,\delta\rangle$.
    We omit the proof, as we will not use this case in any of the paper's main results.
    Removing it from the hypothesis would make reading several lemmas' statements more cumbersome.
\end{proof}
\begin{cor}\label{cor: iso intvls same theta partition}
    Let $[x,y],[x',y']$ be full intervals and let $x,x',y,y'$ be dominant elements.
    If $[x,y]$ and $[x',y']$ are isomorphic as posets, then $x$ and $x'$ belong to the same part of the $X\Theta$-partition, and similarly, $y$ and $y'$ belong to the same part of the $X\Theta$-partition.
    In particular, there are $\lambda,\mu\in\Lambda$ such that $x+\lambda=x'$ and $y+\mu=y'$.
\end{cor}
\begin{proof} 
    We may assume that $y,y'\not\in\{\theta(0,0),\theta^s(0,0)\}$, otherwise, it is easy to see that the intervals $[x,y]$ and $[x',y']$ are not full.

    Since $[x,y]$ is full, the set of atoms of~$[x,y]$ is $\operatorname{U}(x)$, and the set of coatoms of~$[x,y]$ is $\operatorname{L}(y)$.
    This implies that $f_1^{x,y}=|\operatorname{U}(x)|$ and $f_{\ell(x,y)-1}=|\operatorname{L}(y)|.$
    By Lemma~\ref{lem: cardinality of upper covers} and the fact that $f_1^{x,y}=|\operatorname{U}(x)|$, this quantity determines the part of the $X\Theta$-partition containing $x$.
    By Remark~\ref{rem: facts of theta-partition}\ref{remark-item: length of the elements}, this determines whether $\ell(x)$ is odd or even.
    Now we can determine the $X\Theta$-partition containing $y$.
    There are two cases: If $\ell(x)$ and $\ell(x,y)$ have the same parity, then $\ell(y)$ is even, so by Remark~\ref{rem: facts of theta-partition}\ref{remark-item: length of the elements} we have $y\in \Theta^s$ and we are done.
    If $\ell(x)$ and $\ell(x,y)$ have different parity, then $\ell(y)$ is odd so by Remark~\ref{rem: facts of theta-partition}\ref{remark-item: length of the elements} we have $y\in \Theta$.
    The first part follows from Lemma~\ref{lem: betti invariante}.
    The second part follows from Remark~\ref{rem: facts of theta-partition}\ref{rem-item: difference of thetas as combination of fundamental weights}.
\end{proof}

\subsection{Basic notions in planar convex geometry} \label{subsec: basic notions convex}
A convex bounded polygon $P$ is the convex hull of a finite set of points in the plane.
If such a set is minimal, we call it the \emph{vertex set} of~$P$, denoted by $\operatorname{V}(P)$.
An \emph{edge} of~$P$ is a closed segment obtained by intersecting $P$ with a closed half-plane.
Two vertices $a,b$ are said to be \emph{adjacent} if the segment $ab$ is an edge of~$P$.

For any subset $Y$ of the plane, we denote its convex hull by $\operatorname{Conv}(Y)$.

To avoid cumbersome notation, and when the context is clear, we usually do not distinguish between elements of~$W$ and the centers of their respective alcoves.
For $A \subset W$, we define $\operatorname{Conv}(A) \coloneqq \operatorname{Conv}(\operatorname{cen}(A))$.
Similarly, if $\operatorname{V}(P) = \operatorname{cen}(A)$, we write $\operatorname{V}(P) = A$.

If $x, y \in W$, we denote by $\operatorname{Sgm}(x, y)$ the line segment in the plane with endpoints $\operatorname{cen}(x)$ and $\operatorname{cen}(y)$---that is,
\begin{equation*}
    \operatorname{Sgm}(x, y) = \operatorname{Conv}(\{\operatorname{cen}(x), \operatorname{cen}(y)\}).
\end{equation*}
If $x \in W$ and $v \in E$, we write $\operatorname{Sgm}(x, v)$ (or $\operatorname{Sgm}(v, x)$) for the segment between $\operatorname{cen}(x)$ and~$v$.

We adopt similar conventions for the inner product~$(-,-)$: for instance, $(v, x)$ denotes $(v, \operatorname{cen}(x))$, where $v \in E$ and $x \in W$.

If $x, y, z \in W$, we write $z \in \operatorname{Sgm}(x, y)$ instead of $\operatorname{cen}(z) \in \operatorname{Sgm}(x, y)$.
More generally, if $Y \subset E$ and $x \in W$, we may write $x \in Y$ instead of $\operatorname{cen}(x) \in Y$.

Finally, we define the Euclidean distance between two points as
\begin{equation*}
    d(a, b) := \lVert \operatorname{Sgm}(a, b) \rVert,
\end{equation*}
where $a, b$ may be points in~$E$, elements of~$W$, or a combination of both.

\subsection{Geometric description of lower intervals} \label{subsec geometric description lower}
In this section, we generalize the geometric descriptions of lower intervals in~\cite{LP23, BLP23}.
Applying these results, we give formulas for the cardinalities of all lower intervals.
\begin{prop}\label{prop description of all lower intervals} 
    For any $y\in W$ there is a bounded convex polygon $\CC_y$ such that 
    \begin{equation}
        [\id,y]=\{x\in W\mid \operatorname{cen}(x)\in \CC_y\}.
    \end{equation}
\end{prop}
\begin{proof}
    We will consider two cases, $\ell(y)\leq 1$ and $\ell(y)\geq2$.
    For $\ell(y)\leq 1$, the polygons $\CC_{\mathrm{id}}:=\operatorname{Conv}(\{\mathrm{id}\})$ and $\CC_{s}:=\operatorname{Conv}(\{\mathrm{id}, s\})$ ($s\in S$) satisfy the proposition.
    We suppose $\ell(y)\geq 2$.
    Consider the following two statements.
    \begin{itemize}
        \item If $y$ is dominant, the hexagon $\CC_y$ with vertex set $W_f\cdot y$  satisfies the proposition.
        \item Let $k\geq 2$ and $y=\mathrm{\mathbf{x}}_k$.
        The quadrilateral $\CC_{\mathrm{\mathbf{x}}_k}$ with vertex set
        \begin{equation*}
            \{\mathrm{\mathbf{x}}_k,s_1\mathrm{\mathbf{x}}_k,s_2s_1\mathrm{\mathbf{x}}_k,s_1s_2s_1\mathrm{\mathbf{x}}_k\}
        \end{equation*}
        satisfies the proposition.
    \end{itemize}
    The proof of the first bullet for $y=\theta(m,n)$ is done in~\cite[Lemma 1.4]{LP23}.
    The remaining case of the first bullet, i.e., $y=\theta^s(m,n)$ and the second bullet, i.e., $\mathrm{\mathbf{x}}_k$, have proofs similar to that of~\cite[Lemma 1.4]{LP23}, so we omit them.
    We now focus on the case $y\in {}^s\Theta^s$.
    
    For $m$ and $n$ non-negative integers, let $y=s_0\theta^s(m,n)$.
    Define $z:=\theta^s(m,n)$.
    By Lemma~\ref{prop: intervals are union by right}, we have $[\id,y]=[\id,z]\cup s_0[\id,z].$
    As $z$ belongs to the cases treated in the first bullet above, this is equivalent to 
    \begin{equation*}
        [\id,y]=\{x\in W\mid \operatorname{cen}(x)\in\mathcal{C}_{z}\cup s_0\mathcal{C}_{z}\}.
    \end{equation*}
    \begin{claim}\label{claim: fac1 Cz}
        $s_0\CC_z=\CC_z-\rho$, where $\rho=\varpi_1+\varpi_2.$ 
    \end{claim}
    \begin{proof}
        We have
        \begin{align}
            \operatorname{V}(\CC_z)&=\{z,s_1z,s_2z, s_1s_2z,s_2s_1z,s_1s_2s_1z\}, \label{eq: Claim hex 1}\ \mathrm{and}\\
            \operatorname{V}(s_0\mathcal{C}_{z})&=\{s_0z,s_0s_1z,s_0s_2z, s_0s_1s_2z,s_0s_2s_1z,s_0s_1s_2s_1z\}.\label{eq: Claim hex 2}
        \end{align}
        By definition, $s_0= s_{\rho,-1}$ and one can easily check that $s_{\rho,0}= s_1s_2s_1$ (for example, check what both elements do to the alcove $A_+$).
        This implies that for any $x\in W,$  we have $s_0x=s_1s_2s_1x-\rho.$
        This last equation implies the claim as $s_0z=s_1s_2s_1z-\rho$, $s_0s_1z=s_1s_2z-\rho$, $s_0s_2z=s_2s_1z-\rho$, etc.
    \end{proof}
    \begin{claim}
        $\cen^{-1}(\CC_z\cup s_0\CC_z)=\cen^{-1}(\operatorname{Conv}(\{z,s_1z,s_2z, s_0z, s_0s_1z,s_0s_2z\}))$
    \end{claim}
    \begin{proof} 
        By Equation \eqref{eq: Claim hex 1} and as $s_{\rho,0}= s_1s_2s_1$, we have that two edges of the hexagon $\CC_z$ are parallel to the vector $\rho$, and the same applies to the hexagon $s_0\CC_z$.
        
        Let $a=s_1z, b=s_1s_2z, c=s_2z, d=s_2s_1z,$ so that  $a,b,c,d$ are four of the six vertices of~$\CC_z$ and the segments $\Sgm(a,b)$ and $\Sgm(c,d)$ are the edges of~$\CC_z$  parallel to  $\rho$.
        By Claim~\ref{claim: fac1 Cz}, we have that $a-\rho, b-\rho, c-\rho, d-\rho\in \operatorname{V}(s_0\CC_z)$.
        Note that it is possible for $a-\rho$ (resp.\@ $c-\rho$) to either lie on~$\Sgm(a,b)$ (resp.\@ $\Sgm(c,d)$) or not.
        To clarify in which case we are, we will now provide a necessary and sufficient condition to determine if $a-\rho$ (resp.\@~$c-\rho$) lies on~$\Sgm(a,b)$ (resp.\@ ~$\Sgm(c,d)$).
        
        As $s_{\rho,0}=s_1s_2s_1$, we have that  $b=s_{\rho,0}a$ and $d=s_{\rho,0}c$.
        This implies that $|\Sgm(a,b)|=2d(a,H_{\rho,0})$ (resp.\@ $|\Sgm(c,d)|=2d(c,H_{\rho,0})$).
        Then $a-\rho \in\Sgm(a,b)$ if and only if $2d(a,H_{\rho,0})\geq\norm{\rho}$, i.e., $d(a,H_{\rho,0})\geq\frac{1}{2}\norm{\rho}$.
        
        On the other hand, the Euclidean distance between $a$ and $H_{\rho,0}$ is given by $d(a,H_{\rho,0})=\frac{\left| (a,\rho)\right|}{\norm{\rho}}$.
        Thus, $a-\rho \in\Sgm(a,b)$ if and only if $\left|(a,\rho)\right| \geq\frac{\norm{\rho}^2}{2}$.
        Since $\rho=\varpi_1+\varpi_2$ and, as  we have that $\norm{\varpi_i}^2=\frac{2}{3}$ and the angle between $\varpi_1$ and $\varpi_2$ is $\frac{\pi}{3}$, we conclude $\norm{\rho}=\sqrt{2}$.
        Hence, $a-\rho \in\Sgm(a,b)$ if and only if $\left|(a,\rho)\right| \geq 1$.
        
        Additionally, since  $a=s_1z$ and $\rho=s_1\alpha_2$, we have that $\left|(a,\rho)\right| \geq1$ 
        if and only if $(\cen(z),\alpha_2)\geq1$.
        We can rewrite the above condition as follows:  
        \begin{equation}\label{eq: condition a_2}
            a-\rho \in\Sgm(a,b)\quad\text{if and only if} \quad (\cen(z),\alpha_2)\geq  1
        \end{equation}
        Similarly, we have that  
        \begin{equation} \label{eq: condition a_1}
            c-\rho \in\Sgm(c,d)\quad\text{if and only if} \quad (\cen(z),\alpha_1)\geq  1.
        \end{equation}
        
        We split the rest of the proof into three cases.
        Recall that by definition, $z=\theta^s(m,n)$.
        \begin{enumerate}
            \item Let us assume that $m,n\geq 1$ (see Figure~\ref{fig: demo hexa1}).
            In this case, one can check that $(\cen(z),\alpha_1)\geq  1$ and $(\cen(z),\alpha_2)\geq  1$.
            To see this, it is enough to calculate these inequalities with $m=n=1$  and that one is easy to calculate, as $\cen(\theta^s(1,1))=\cen(w_0)+\frac{4}{3}\rho$ (recall that $w_0=s_1s_2s_1$ denotes the longest element in~$W_f$).
            By Equation \eqref{eq: condition a_2} and Equation \eqref{eq: condition a_1}, we have $a-\rho \in \Sgm(a,b)$ and $c-\rho \in \Sgm(c,d)$.
            This implies that
            $b\in\Sgm(a-\rho,b-\rho)$ and $d\in \Sgm(c-\rho,d-\rho)$.
            In other words, we have \begin{align}        \Sgm(a,b)\cup \Sgm(a-\rho,b-\rho)&=\Sgm(a,b-\rho)\\ \Sgm(c,d)\cup \Sgm(c-\rho,d-\rho)&=\Sgm(c,d-\rho).\end{align}
            Furthermore, as $z=\theta^s(m,n)$ with $m,n\geq 1$, one can see that $z-\rho\in D$, so  $z-\rho \in \CC_z$.
            We also have that  $s_1s_2s_1z\in s_0\CC_z$, because $s_1s_2s_1z-\rho \in \operatorname{V}(s_0\CC_z)$.
            So we have $ \{z-\rho,a-\rho,c-\rho\}\subset\CC_z$ and $\{b,d,s_1s_2s_1z\}\subset s_0\CC_z $.
            From this two inclusions, together with Equation \eqref{eq: Claim hex 1}, Equation \eqref{eq: Claim hex 2}, Claim \eqref{claim: fac1 Cz} and the fact that $\Sgm(a,b),\Sgm(c,d)$ are parallel  to $\rho$, we conclude that $\CC_z\cup s_0\CC_z$ is a hexagon with vertex set $\{z,s_0z,a,c,b-\rho,d-\rho\}$.
            Therefore, we have 
            \begin{equation}\label{eq: igualdad de hexagonos}
                \CC_z\cup s_0\CC_z=\operatorname{Conv}(\{z,s_1z,s_2z, s_0z, s_0s_1z,s_0s_2z\}),
            \end{equation} 
            so the claim follows.
            \begin{figure}[ht!]
                \centering
                \includegraphics[width=0.44\textwidth]{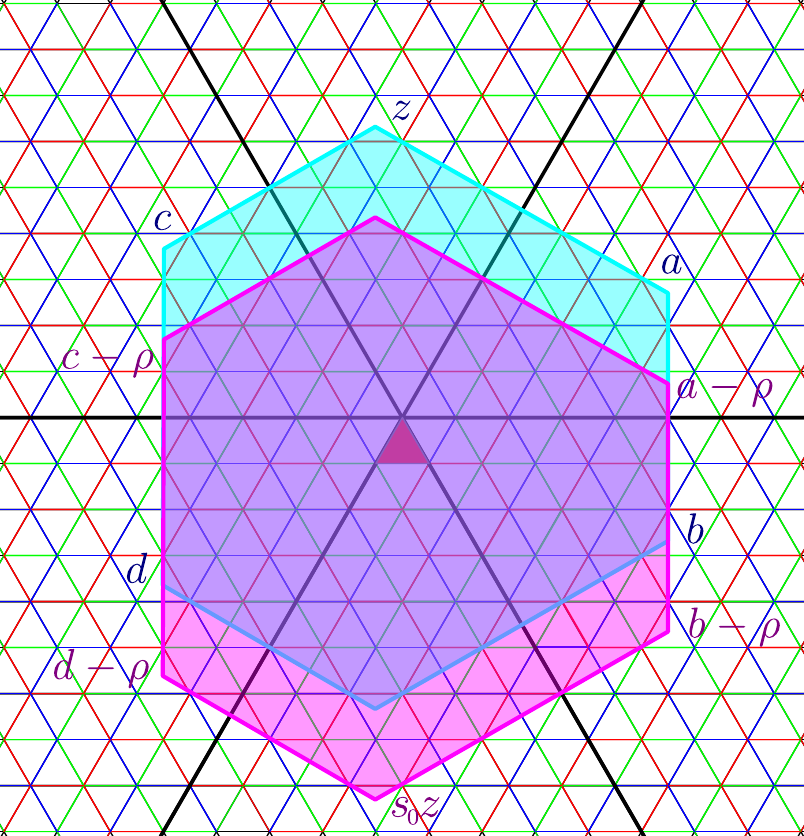}
                \caption{$\CC_z$ and $s_0\CC_z$, for $z=\theta^s(m,n)$.}
                \label{fig: demo hexa1}
            \end{figure}
            \item In this case we assume $m=0$ and $n\geq 1$ or the symmetric case $n=0$ and $m\geq 1,$ see Figure~\ref{subfig: demo hexa 2} for an illustration.
            Let us prove the case $z=\theta^s(m,0)$ with $m\geq 1$, the other being symmetric.
            We have $(\cen(z),\alpha_1)\geq  1$ and $0<(\cen(z),\alpha_2)<  1$.
            By Equation \eqref{eq: condition a_1}, we have that $c-\rho\in \Sgm(c,d)$, so $d\in\Sgm(c-\rho,d-\rho)$.
            However, by Equation \eqref{eq: condition a_2}, we have $a-\rho\notin\Sgm(a,b)$.
            Thus, Equation \eqref{eq: igualdad de hexagonos} does not hold,
            although the strict inclusion $\subsetneq$ still holds.
            The difference between the set on the right and the set on the left of Equation \eqref{eq: igualdad de hexagonos} is a ``small'' equilateral triangle $T$, with two vertices of~$T$ being $a-\rho$ and $b$.
            It is easy to see that $T$ contains no center of alcoves other than $a-\rho$ and $b$, as $a-\rho$ and $b$ are the centers of two alcoves (call them $A$ and $B$), the closures  $\overline{A}$ and $\overline{B}$  share an edge, and $T\subset \overline{A}\cup \overline{B}.$
            This proves the claim for $n=0$ and $m\geq 1$.
            \item Let us assume $m=n=0$, see Figure~\ref{subfig: demo hexa 3}.
            We have that $0<(\cen(z),\alpha_1)<1$ and $0<(\cen(z),\alpha_2)<1$.
            This case is similar to the previous one.
            However, instead of obtaining just one equilateral triangle, we now obtain two: one with $a-\rho$ and $b$ as two of its vertices and the other with $c-\rho$ and $d$ as two of its vertices.\qedhere
            \begin{figure}[!ht]
                \centering
                \begin{subfigure}[b]{0.45\textwidth}
                    \centering
                    \includegraphics[width=0.9\textwidth]{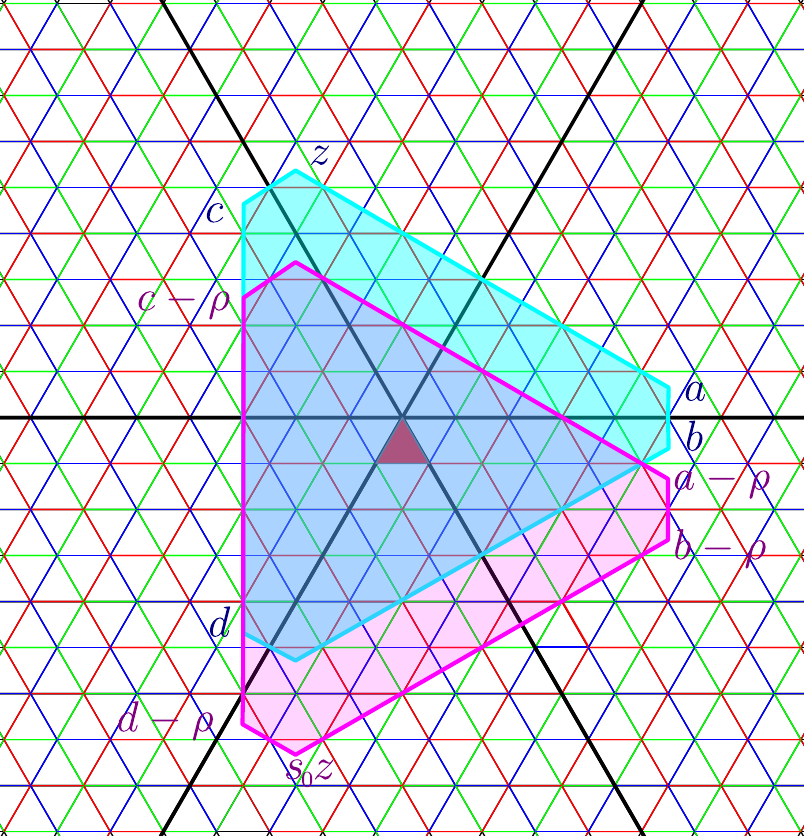}
                    \caption{$\CC_z$ and $s_0\CC_z$, for $z=\theta^s(m,0)$.}
                    \label{subfig: demo hexa 2}
                \end{subfigure}
                    \hfill
                \begin{subfigure}[b]{0.45\textwidth}
                \centering 
                \includegraphics[width=0.9\textwidth]{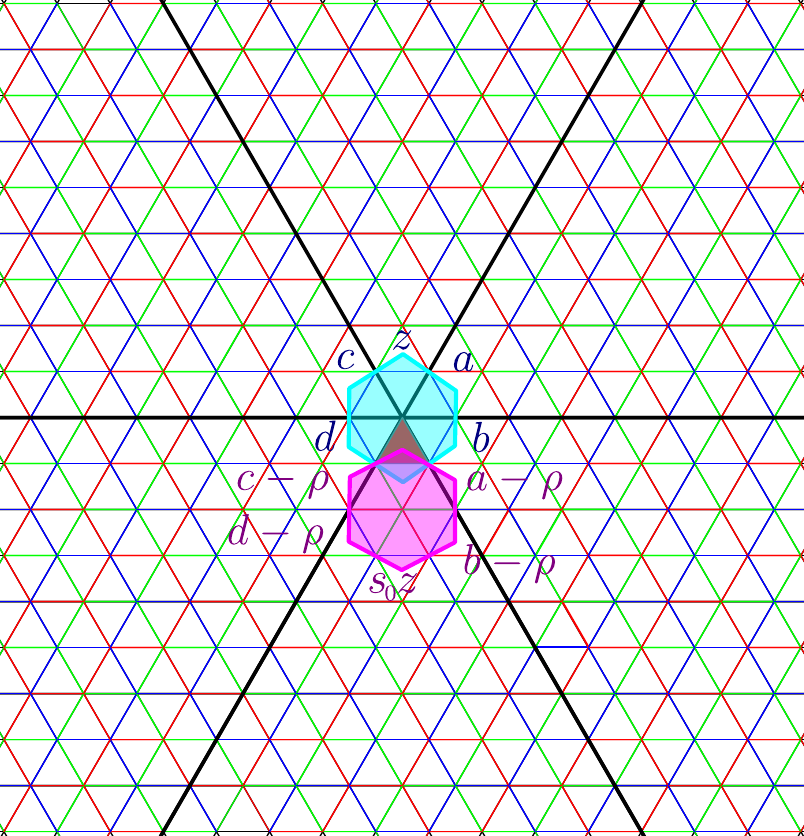}
                \caption{$\CC_z$ and $s_0\CC_z$, for $z=\theta^s(0,0)$.}
                \label{subfig: demo hexa 3}
                \end{subfigure}
                \caption{}
            \end{figure}
        \end{enumerate}
    \end{proof}
    Back to the proof of the proposition, we conclude that the hexagon $\CC_y$, for $y=s_0\theta^s(m,n)$ has vertex set $\{z,s_1z,s_2z, s_0z, s_0s_1z,s_0s_2z\}$, where $z=\theta^s(m,n)$.
    
    The case $y = s_0\theta(m,n)$ can be proved in the same way as the case $y = s_0\theta^s(m,n)$, simply replacing $z = \theta^s(m,n)$ with $z = \theta(m,n)$ in the proof, so we omit the details.
    Finally, consider $y$ and $g$ as in Remark~\ref{rem: remark tontos}, then $x=gy$ belongs to one of the cases above.
    Since $g$ acts as an isometry, for any set $Y$, we have that $g\operatorname{Conv}(Y)=\operatorname{Conv}(gY)$.
    So $\CC_y:=g^{-1}\CC_{x}$ satisfies the proposition.
\end{proof}
\begin{cor}\label{cor: characterization of the Bruhat order}
    Let $y\in W$.
    Let $\CC_y$ be the unique bounded convex polygon such that
    \begin{equation*}
        \operatorname{V}(\CC_y)=
        \begin{cases}
            \delta^iW_f \d^{-i}y&\mbox{ if $y\in \delta^iD$,}\\
            \{y, s_iy,s_{1+i}s_{i}y,s_{2+i}s_{i}y,s_{i}s_{1+i}s_iy,s_is_{2+i}s_iy\}&\mbox{ if $y\in \delta^i s_0D$,}\\
            g\{\mathrm{\mathbf{x}}_k,s_1\mathrm{\mathbf{x}}_k,s_2s_1\mathrm{\mathbf{x}}_k,s_1s_2s_1\mathrm{\mathbf{x}}_k\}&\mbox{ if $y=g\mathrm{\mathbf{x}}_k$.}
        \end{cases}
    \end{equation*}
    Where $D=\{\theta(m,n),\theta^s(m,n)\mid m,n\geq 0\}$.
    Then, $\CC_y$ satisfies Proposition~\ref{prop description of all lower intervals}.
    In other words, $x\leq y$ if and only if $x\in \CC_y$.
\end{cor}
Throughout the remainder of the paper, we use the notation $\CC_y$ for $y\in W$ to denote the convex polygon defined in Corollary~\ref{cor: characterization of the Bruhat order}.
\begin{rem}
    The previous corollary does not generalize in this exact form to all affine Weyl groups.
    For example, one can find instances where it fails in type~$\widetilde{A}_3$.
    The closest we know of a geometric description of the Bruhat order for general affine Weyl groups is the ``lattice formula''~\cite[Theorem A]{CdLP23}, which applies only for elements of the form $\theta(\lambda)$ where $\lambda$ is a dominant weight.
\end{rem}
We end the section by giving formulas for $|[\mathrm{id}, y]|$, for every $y\in W$.
They are easy to obtain from Corollary~\ref{subsec geometric description lower}.
\begin{prop}\label{prop: cardinal lower}
    For all $m,n\in \mathbbm{Z}_{\geq0}$, we have 
    \begin{align}
        |[\id,\theta(m,n)]|&=3m^2+3n^2+12mn+9m+9n+6.\\
        |[\id,\theta^s(m,n)]|&=3m^2+3n^2+12mn+15m+15n+12.\\
        |[\id,s_0\theta^s(m,n)]|&=3m^2+3n^2+12mn+21m+21n+22.
    \end{align}
    Additionally, for $k\geq1$, we have $|[\mathrm{id}, \mathrm{\mathbf{x}}_{2k}]|=3k^2+k$ and $|[\mathrm{id}, \mathrm{\mathbf{x}}_{2k+1}]|=3k^2+5k$.
\end{prop}
These formulas already appear in~\cite[Lemma 1.5]{LP23} and~\cite[Equations 2.10, 2.11, and 2.12]{BLP23}.

%% file: Sections/Geometry.tex
\section{The geometry of Bruhat intervals in \texorpdfstring{$\widetilde{A}_2$}{A2⁓}}\label{section: geometry of intervals}
The purpose of the section is to prove Theorem~\ref{thm: A}.
We first focus on the case where \( x \) and \( y \) are dominant elements, 
and then extend the argument to the general case in Section~\ref{subsec: geometry of general intervals}.

\subsection{Description of \texorpdfstring{$\geq x$}{≥ x}}\label{subsec: description de mayores que x}
For \( w \in W_f \), we define \( F_w \) as the closed, convex, unbounded subset of \( E \) illustrated in Figure~\ref{subfig: star and Fw zones intro}.
Each zone will be either called by its label or by its color.
Some zones are contiguous to others and delimited by unbounded rays.
There are $6$ of these rays.
Each of them is contained in a line of~$\HC$.
The corresponding affine reflections of these lines are $s_{\alpha_1,-1}, s_0, s_2, s_1, s_0$, and $s_{\alpha_2,-1}$.
For convenience, we denote them by $r_1,r_2,r_3,r_4,r_5$, and $r_6$ respectively.
We use ``label mod~$6$'' notation, so for example $r_{11}=r_5$.
Note that $r_2=r_5=s_0$.
We also relabel the zones by $F_1,\ldots, F_6$ (again ``label mod~$6$'' notation) in a ``clockwise'' way starting from $F_{\mathrm{id}}$, in other words, $F_1$ is gray, $F_2$ is orange, $F_3$ is pink, $F_4$ is green, $F_5$ is yellow, and $F_6$ is blue.
From Figure~\ref{subfig: star and Fw zones intro}, it is immediate that $r_i\cdots r_1(D)\subset F_{i+1}$ for $i\geq 1$.
We use the convention $A_+\in F_\id$

\begin{defn}\label{def: estrella}[Construction of~$\mathcal{St}(x)$]
    See Figure~\ref{subfig: star section 3} for an example of the construction that follows.
    Let $x\in D\subset W$.
    Define $x_1:=\cen(x)\in E$ and $x_{i+1}:=r_ix_i$ for all $1\leq i\leq 6$.
    As noted above, $x_i\in F_i$.
    The segment $\overline{x_ix_{i+1}}$ divides the ray corresponding to $r_i$ into two distinct parts: bounded and unbounded.
    A unique point $u_{i+1}$ satisfying that $x_iu_{i+1}x_{i+1}$ is an equilateral triangle, lies in the ray's unbounded part.
    
    Note that $x_7=x_1$ because $r_6\cdots r_1=\mathrm{id}$ (this is clear from Figure~\ref{subfig: star and Fw zones intro} since $r_6\cdots r_1$ fixes $A_+$).
    For convenience, we use label mod~$6$ also for $x_i$ and $u_i$.
    The union of the segments $x_iu_{i+1}$ and $u_{i+1}x_{i+1}$, for $1\leq i\leq 6$, forms a closed loop which is the boundary of a (non-convex) bounded $12$-sided polygon which we denote by $\mathcal{St}(x)$.
    It is easy to see that $\mathcal{St}(x)$ is a $6$-pointed star with acute inner angles $\pi/3$ and obtuse outer angles $2\pi/3$.
\end{defn}

We relabel the $x_i$'s in the following way: For $w\in W_f$, there is a unique $i\mod{6}$ such that $F_w=F_i$, so we write $x_w:=x_i$.
Let us denote by $\mathrm{IVS}_x:=\{x_w\mid w\in W_f\}$ and $\mathrm{EVS}_x:=\{u_i\mid 1\leq i\leq 6\}$ the sets of inner and exterior vertices of~$\mathcal{St}(x)$ respectively.

\begin{prop} \label{prop: no mayores geometrico}
    If $x$ is dominant
    \begin{equation*}
        {\geq x}=\{z\in W\mid  \operatorname{cen}(z)\not\in\mathcal{St}^{\circ}(x)\},
    \end{equation*}
    where $\mathcal{St}^{\circ}(x)$ denotes the interior of the star-shaped polygon $\mathcal{St}(x)$.
\end{prop}

Before we prove this proposition, we need some preliminary lemmas.
For the rest of the section, we fix an element $x\in D$.

Consider the following regions (see Figure~\ref{subfig: G_i})
\begin{gather*}
    H_1:=F_1\setminus\{v\in E\mid -1\leq (v,\alpha_2)< 0\},\\
    H_2:=F_2, H_3:=F_3,\\
    H_4:=F_4\setminus\{v\in E\mid -1\leq (v,\alpha_1)< 0\}.
\end{gather*}
\begin{figure}
    \centering
    \begin{subfigure}[b]{0.45\textwidth}
        \centering
        \includegraphics[width=0.9\textwidth]{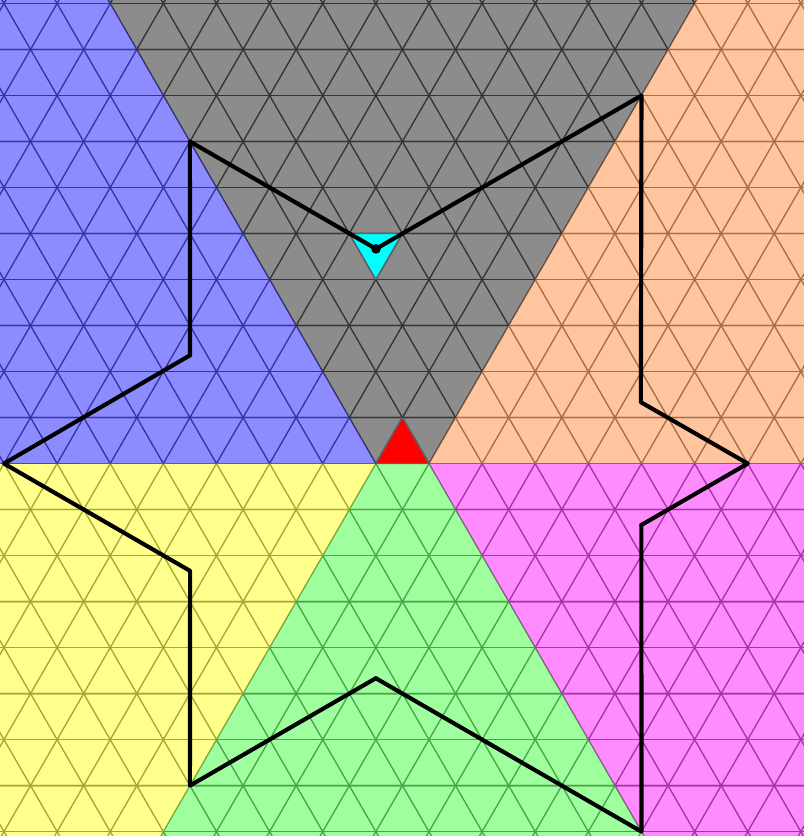}
        \caption{The set $\partial\mathcal{St}(\th(2,1))$.}
        \label{subfig: star section 3}
    \end{subfigure}
        \hfill
    \begin{subfigure}[b]{0.45\textwidth}
        \centering 
        \includegraphics[width=0.9\textwidth]{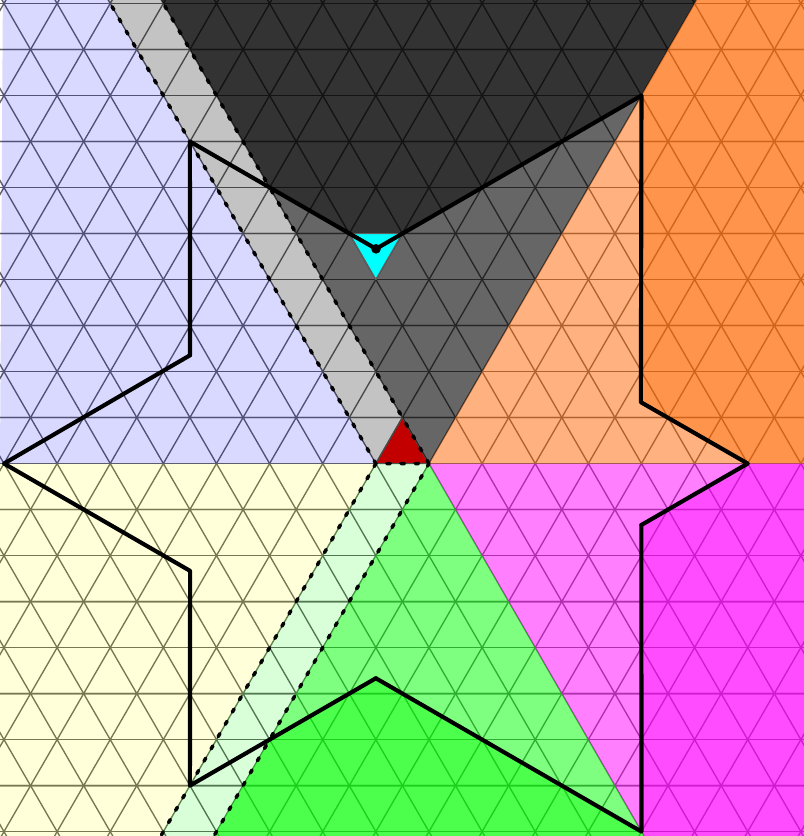}
        \caption{ Regions
        $H_i^{in},H_i^{ex}, 1\leq i\leq 4$.}
        \label{subfig: G_i}
    \end{subfigure}
    \caption{Illustration of~$\partial\mathcal{St}(x)$ and the regions $H_i^{in},H_i^{ex}$ for  $x=\th(2,1)$, $1\leq i\leq 4$.
    The teal-colored alcove corresponds to $x$.
    In the figure on the right, the complement of~$H$ is colored lighter than in the figure on the left, and the four exterior regions $H_i^{\mathrm{ex}}$ are depicted with darker colors than their corresponding interior regions $H_i^{\mathrm{in}}$.}
\end{figure}
For $1\leq j\leq 4$, the set $H_j\setminus \partial(\mathcal{St}(x))$ has two connected components, one in the interior of the star and one in the exterior of the star.
Let us call $H_j^{\mathrm{in}}$ and $H_j^{\mathrm{ex}}$ the connected components in the interior and in the exterior, and  $H_j^{\mathrm{bor}}:= H_j\cap \partial(\mathcal{St}(x))$.
We emphasize that these sets depend on~$x$, but we omit $x$ from the notation.
From the description of~$\mathcal{St}(x)$,
one can deduce that for $1\leq j\leq 3$ the reflection $r_j\in W$ induces a bijection $r_j\colon H_j\to H_{j+1}$ such that $r_j(H_j^{*})=H_{j+1}^{*}$,  for $*\in \{\mathrm{in},\mathrm{ex},\mathrm{bor}\}$.

Let us denote the union of these sets  $H:=\biguplus_{j=1}^4 H_j$.
Let $v\in H_i$.
If $i=1,$ we define $v(1)=v\in H_1$.
If not, we define $v(1)=r_1r_2\cdots r_{i-1}v\in H_1$.
For $2\leq k\leq 4$ we define $v(k):=r_{i-1}\cdots r_2r_1(v(1))\in H_k$.

The following lemma is immediate.
\begin{lem}\label{lem: fact1}
    We have $H_j^*(1)=H_1^*$ for every $1\leq j\leq 4$ and $*\in \{\mathrm{in},\mathrm{ex},\mathrm{bor}\}$.
\end{lem}
\begin{lem}\label{lem: Fact2}
    Let $z\in W$.
    If $z\in H$, then $z(1)\in \mathrm{V}(\CC_z)$.
\end{lem}
\begin{proof} 
    We split the proof into four cases.
    \begin{itemize}
        \item Case $z \in H_1$.
        By Corollary~\ref{cor: characterization of the Bruhat order}, we have $z(1)=z\in \mathrm{V}(\CC_z)$.
        \item Case $z\in H_2$.
        There are two possibilities:
        \begin{itemize}
            \item $\delta z\in s_0D$.
            Then $z\in\delta^2 (s_0D)$.
            By Corollary~\ref{cor: characterization of the Bruhat order}, we have  $z(1)=s_2s_0s_2z\in \mathrm{V}(C_z)$.
            \item $z\in X$.
            Then $z=\delta \mathrm{\mathbf{x}}_k$ for some $k\geq 2$.
            By Corollary~\ref{cor: characterization of the Bruhat order}, we have $z(1)=s_2s_0s_2z=\delta(s_1s_2s_1\mathrm{\mathbf{x}}_k)\in \mathrm{V}(C_z)$.
        \end{itemize} 
        \item Case $z\in H_3$.
        By Corollary~\ref{cor: characterization of the Bruhat order}, we have $z(1)=s_2s_0z\in \mathrm{V}(C_z)$.
        \item Case $z\in H_4$.
        The proof is similar to the second case.\qedhere
    \end{itemize}
\end{proof}
\begin{defn} 
    Let $w\in W_f$.
    Let $\operatorname{Cone}^w(\alpha_1,\alpha_2)$ be the positive cone generated by $w(\alpha_1)$ and $w(\alpha_2)$.
    In formulas, $\operatorname{Cone}^w(\alpha_1,\alpha_2):=\{aw(\alpha_1)+bw(\alpha_2)\mid a,b\in \mathbbm{R}_{\geq 0}\}$.
    For simplicity, we write $\operatorname{Cone}(\alpha_1,\alpha_2):=\operatorname{Cone}^{\mathrm{id}}(\alpha_1,\alpha_2)$.
\end{defn}
\begin{lem}\label{lem: Fact 3} 
    Let $z\in W$.
    If $z\in H$, then 
    \begin{equation*}
        \CC_z\cap H_1\cap \cen(W)=\Bigl( z(1)-\Cone(\alpha_1,\alpha_2)\Bigr)\cap H_1\cap \cen(W).
    \end{equation*}
\end{lem}
\begin{proof} 
    Let $X^{\mathrm{ne}}\subset E$ (resp.\@ $X^{\mathrm{e}}, X^{\mathrm{se}}\subset E$) be the strip of the plane starting at the identity and going northeast (resp.\@ east, southeast) 
    intersected with $H_1$ (resp.\@ $H_2, H_4$).
    In formulas,
    \begin{align*}
        X^{\mathrm{ne}}&\coloneqq\{v\in H_1\mid-1\leq(v,\alpha_1)<0)\},\\
        X^{\mathrm{e}}&\coloneqq\{v\in H_2\mid -1\leq (v,\rho)<0\},\\
        X^{\mathrm{se}}&\coloneqq\{v\in H_4\mid -1\leq (v,\alpha_2)<0\}.
    \end{align*}
    Let $Y^{\mathrm{ne}}\subset H_1$ and $ Y^{\mathrm{se}} \subset H_4$ be the strips directly to the left of~$X^{\mathrm{ne}}$ and $X^{\mathrm{se}}$ respectively.
    In formulas,
    \begin{align}
        Y^{\mathrm{ne}}&\coloneqq\{v\in H_1\mid 0\leq (v,\alpha_1)<1\},\label{equation: def Yne}\\
        Y^{\mathrm{se}}&\coloneqq\{v\in H_4\mid -2\leq (v,\alpha_2)<-1\}.\nonumber
    \end{align}
    By Corollary~\ref{cor: characterization of the Bruhat order} and case-by-case inspection, if $z\in H\setminus (X^{\mathrm{ne}}\cup X^{\mathrm{e}}\cup X^{\mathrm{se}}\cup Y^{\mathrm{ne}}\cup Y^{\mathrm{se}})$, then  $\CC_z$ is a hexagon with only one vertex in~$H_1$ (for instance, the blue hexagon in Figure~\ref{fig: proofa}).
    If $z\in X^{\mathrm{e}}$, then $\CC_z$ is a quadrilateral with only one vertex in~$H_1$ (for instance, the yellow quadrilateral in Figure~\ref{fig: proofa}).
    In both cases, we have
    \begin{equation}\label{eq: onthenose}
        \CC_z \cap H_1=\Bigl(z(1)-\Cone(\alpha_1,\alpha_2)\Bigr)\cap H_1,
    \end{equation}
    so the lemma follows.
    \begin{figure}[ht!]
        \centering
        \includegraphics[width=0.44\textwidth]{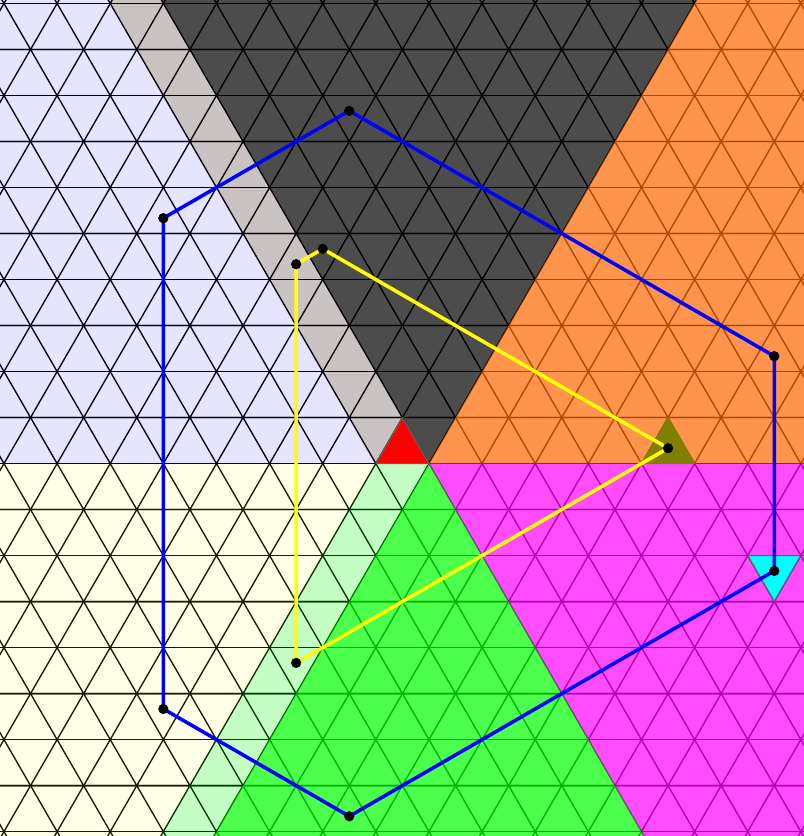}
        \caption{$\partial C_z$ for $z\notin X^{\mathrm{ne}}\cup X^{\mathrm{se}}\cup Y^{\mathrm{ne}}\cup Y^{\mathrm{se}}$.} 
        \label{fig: proofa}
    \end{figure}

    It remains to prove the case $z\in X^{\mathrm{ne}}\cup X^{\mathrm{se}}\cup Y^{\mathrm{ne}}\cup Y^{\mathrm{se}}$.
    Suppose that $z$ is as in Figures~\ref{fig: proofb} or~\ref{fig: proofc}.
    By Corollary~\ref{cor: characterization of the Bruhat order} and case-by-case inspection, Equation \eqref{eq: onthenose} does not hold, although the strict inclusion $\subsetneq$ always holds in this case.
    The difference between the two sets in this strict inclusion (denoted by $\operatorname{Dif}_z$) is a right triangle minus its hypotenuse.
    The set $\operatorname{Dif}_z$ contains no center of alcoves (although along its hypotenuse, there are either one or two elements of~$\operatorname{cen}(W)$.)
    This proves the lemma for the elements $z$ appearing in Figures~\ref{fig: proofb} and~\ref{fig: proofc}.
    Let us argue why this is enough to prove the lemma.

    In Figure~\ref{fig: proofb}, we see a yellow quadrilateral corresponding to a specific element $z\in X^{\mathrm{se}}$.
    If one chooses a different element $z$ in~$X^{\mathrm{se}}$, then  Corollary~\ref{cor: characterization of the Bruhat order} gives a quadrilateral $\CC_z$ with the same inner angles as the yellow quadrilateral.
    In this case, $z(1)$ also belongs to $ X^{\mathrm{ne}}$.
    Modulo a translation of the plane, the right triangle 
    $\operatorname{Dif}_z$ depends only on the parity of~$\ell(z)$.
    If the parity of~$\ell(z)$ is different from the one considered in Figure~\ref{fig: proofb}, the area of~$\operatorname{Dif}_z$ is much smaller (it is one-sixth of the area of an alcove) than in the case depicted there.
    In both cases, $\operatorname{Dif}_z \cap \cen(W)=\emptyset$.
    The remaining three cases, $z\in X^{\mathrm{ne}}\cup Y^{\mathrm{ne}}\cup Y^{\mathrm{se}}$ are similar.
    This concludes the proof of the Lemma.
    \begin{figure}[ht!]
         \centering
         \begin{subfigure}[b]{0.45\textwidth}
            \centering
             \includegraphics[width=0.9\textwidth]{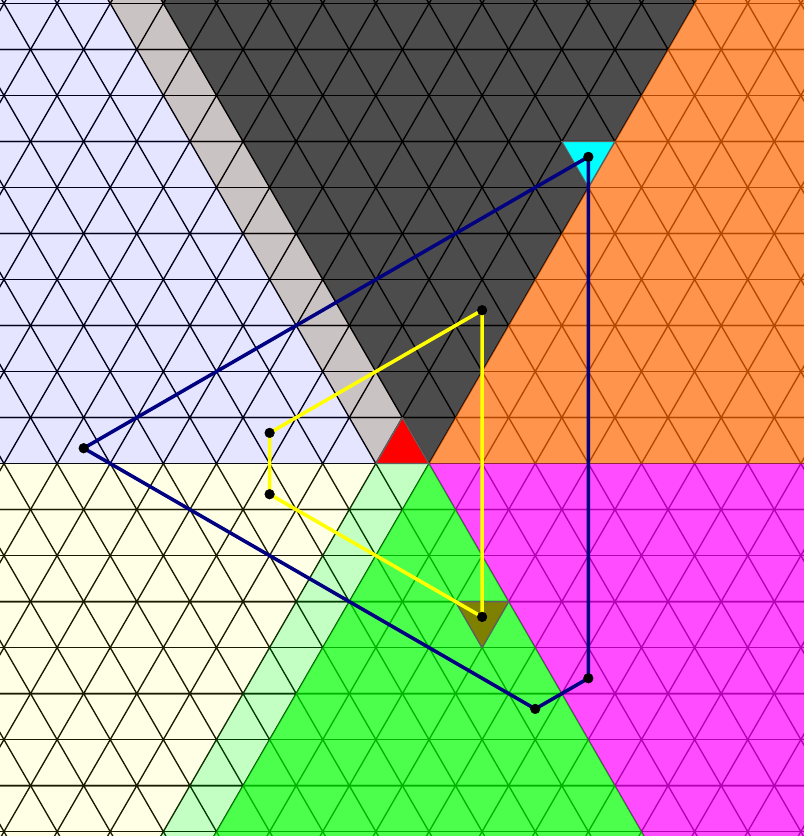}
             \caption{$z\in X^{\mathrm{ne}}\cup X^{\mathrm{se}}$.} 
             \label{fig: proofb}
         \end{subfigure}
         \hfill
         \begin{subfigure}[b]{0.45\textwidth}
             \centering
             \includegraphics[width=0.9\textwidth]{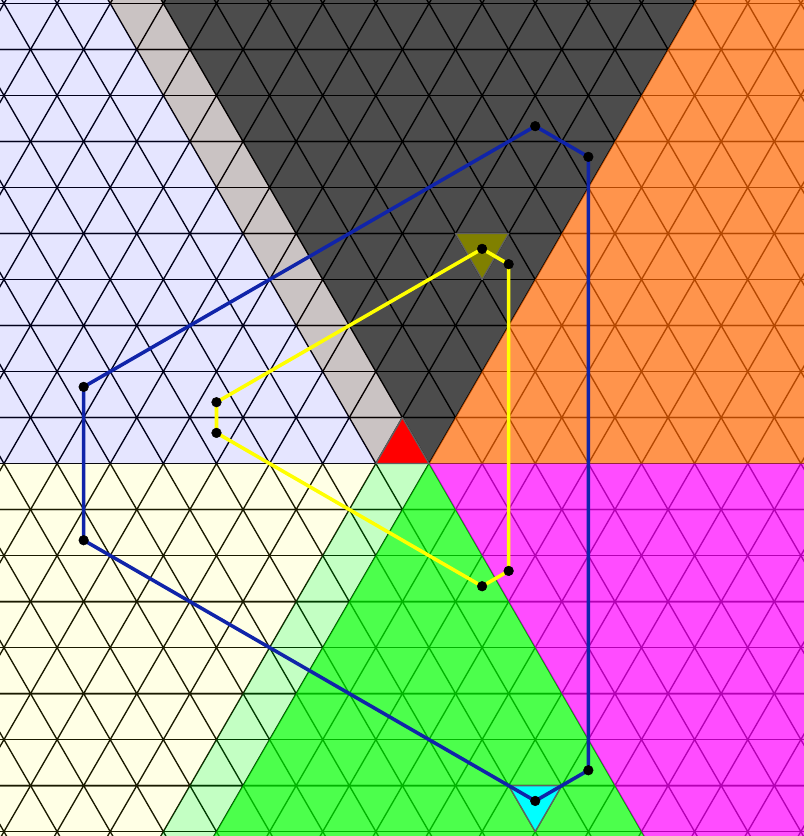}
             \caption{$z\in Y^{\mathrm{ne}}\cup Y^{\mathrm{se}}$.}
             \label{fig: proofc}
         \end{subfigure} 
         \caption{Other examples of~$\partial C_z$.}
         \label{fig: otros casos dem}
    \end{figure}
\end{proof}
The following lemma is the analog of Equation \eqref{eq: onthenose} for the set $E\setminus \mathcal{St}(x)$ in place of the set $\CC_z$.
It follows directly from the definition of~$\mathcal{St}(x)$, and we will omit its proof.
Unlike Equation \eqref{eq: onthenose}, an equality holds in each possible case.
\begin{lem}\label{lem: Fact 4} 
    We have $H_1\setminus \mathcal{St}(x)^{\circ}=H_1\cap(x+\operatorname{Cone}(\alpha_1,\alpha_2))$.
    In particular, if $z\in W$ and $z\in H_1$, then $z\notin \mathcal{St}(x)^{\circ}$ if and only if $x \in z-\Cone(\alpha_1,\alpha_2)$.
\end{lem}
\begin{proof}[Proof of Proposition~\ref{prop: no mayores geometrico}]
    Let us first prove the proposition for $z\in W$ with $\cen(z)\in H_i$ for some $1\leq i\leq 4$.
    By Corollary~\ref{cor: characterization of the Bruhat order}, $x\leq z$ if and only if $x\in \CC_z$.
    Note that $x\in H_1$.
    By Lemma~\ref{lem: Fact 3}, $x\in \CC_z$ if and only if $x\in z(1)-\Cone(\alpha_1,\alpha_2)$.
    By Lemma~\ref{lem: Fact 4}, $x\in z(1)-\Cone(\alpha_1,\alpha_2)$ if and only if $z(1)\notin \mathcal{St}(x)^{\circ}$ if and only if $z(1)\in H_1^{\mathrm{ex}}\cup H_1^{\mathrm{bor}}$.
    By Lemma~\ref{lem: fact1}, $z(1)\in H_1^{\mathrm{ex}}\cup H_1^{\mathrm{bor}}$ if and only if $z\in H_i^{\mathrm{ex}}\cup H_i^{\mathrm{bor}}$ if and only if $z\notin \mathcal{St}(x)$.
    This proves the proposition for $z\in H$.
    
    If one replaces everywhere in the argument $H$ by $\sigma(G)$ (swap $s_1$ and $s_2$, $v$ and $\sigma v$, etc.\@) we obtain a proof for the proposition for $z\in W$ such that $z\in \sigma H$.
    The only elements $z\in W$ such that $z\notin H\cup \sigma H$ are $\mathrm{id}$ and $s_0$.
    However $\mathrm{id}, s_0\notin (\geq x)$.
    The proposition is proved.
\end{proof}

\subsection{Local Bruhat order} Now we will present results similar to Lemma~\ref{lem: Fact 3}, which 
provide a local description of the Bruhat order in each zone $F_w$.
These results will be used in Section~\ref{subsec:Trasl-Preserv-Bruhat-order} and Section~\ref{subsec: thicks}.

\begin{lem}\label{lem: local bruhat order}
    Let $z\in F_\id$.
    Then $\CC_z \cap F_\id \cap \cen(W)=\big( z-\Cone(\alpha_1,\alpha_2)\big)\cap F_\id\cap \cen(W)$.
\end{lem}
\begin{proof}
    Note that either $z\in H_1=F_{\id}\setminus\{v\in E\mid -1\leq (v,\alpha_2)< 0\}$ or $z\in\sigma(H_1)$.
    The case $z\in H_1$ follows by Lemma~\ref{lem: Fact 3} with $z(1)=z$.

    The remaining case $z=\mathrm{\mathbf{x}}_n$ is proved in a similar way as in the part of the proof of Lemma~\ref{lem: Fact 3} following Figure~\ref{fig: proofb}.
    This is because, although a different $z$ is considered, the difference again has the following properties: it is a right triangle minus its hypotenuse, it is included in~$\overline{zA_+}$, and it contains no element of~$\cen(W)$.
\end{proof}
\begin{cor}\label{cor: equivalence of being in the hexagon in the gray region}
    Let $z,z'\in F_\id$.
    Then  $z\in \CC_{z'}$ if and only if  $(\varpi_1,\cen(z')-\cen(z))\geq 0$ and $(\varpi_2,\cen(z')-\cen(z))\geq 0$.
\end{cor}
\begin{proof}
    By Lemma~\ref{lem: local bruhat order},   $z\in \CC_{z'}$ if and only if $\cen(z)=\cen(z')-t_1\alpha_1-t_2\alpha_2$, for some $t_1,t_2\in \mathbbm{R}_{\geq 0}$ if and only if for $i\in \{1,2\}$
    \begin{equation*}
        (\varpi_i,\cen(z')-\cen(z))=(\varpi_i,t_1\alpha_1+t_2\alpha_2)=t_i\geq 0.\qedhere
    \end{equation*}
\end{proof}
\begin{lem}\label{lem: gral local bruhat order}
    Let $z$ be dominant and $w\in W_f$.
    We have,
    \begin{equation*}
        \CC_z\cap F_w\cap\cen(W)=\big(wz-\Cone^w(\alpha_1,\alpha_2)\big)\cap F_w\cap \cen(W).
    \end{equation*}
\end{lem}
\begin{proof}
    If each element of~$\mathrm{V}(\CC_z)$ belongs to a different zone, then it is clear that $\CC_z\cap F_w=\big(wz-\Cone^w(\alpha_1,\alpha_2)\big)\cap F_w,$ and we are done.
    This happens if $z\in D+\rho$.
    The complement by this set in~$D$ is composed of two strips, and by symmetry, we can restrict to the case $z\in \{\theta(0,n),\theta^s(0,n)\, \vert\, n\in \mathbbm{N}\}$.
    For the case $z=\theta(0,n)$ see Figure~\ref{fig: proofc}, and use a similar argument as in the part of the proof of Lemma~\ref{lem: Fact 3} following Figure~\ref{fig: proofa}.
    The case $z=\theta^s(0,n)$ is similar.
\end{proof}
The next lemma is an analogue of Lemma~\ref{lem: gral local bruhat order} for  \(z \in \{\mathrm{\mathbf{x}}_k \mid k > 3\}\).
\begin{lem}\label{lem: gral local bruhat order for x}
    For $k\geq 4$, we have that
    \begin{equation*}
            \CC_{\mathrm{\mathbf{x}}_k}\cap F_w\cap\cen(W)=
            \begin{cases}
                \big(w\mathrm{\mathbf{x}}_k-\Cone^w(\a_1,\rho)\big)\cap F_w\cap\cen(W)& \mbox{ if } w\not\in\{s_2,s_1s_2\},
                \\
                \big(ws_2\mathrm{\mathbf{x}}_k-\Cone^{ws_2}(\a_1,\rho)\big)\cap F_w\cap\cen(W)& \mbox{ if } w\in\{s_2,s_1s_2\}.
            \end{cases}
    \end{equation*}
\end{lem}
\begin{proof}
    By Corollary~\ref{cor: characterization of the Bruhat order}, we have  $V(\CC_{\mathrm{\mathbf{x}}_k})=\{\mathrm{\mathbf{x}}_k,s_1\mathrm{\mathbf{x}}_k,s_2s_1\mathrm{\mathbf{x}}_k,s_1s_2s_1\mathrm{\mathbf{x}}_k\}$.
    Recall that $s_1=s_{\a_1,0}, s_2=s_{\a_2,0}$ and $s_1s_2s_1=s_{\rho,0}$.
    Note that
    \begin{equation}\label{eq: regiones de vert para x_k}
        s_1\mathrm{\mathbf{x}}_k\in F_{s_1},\quad s_2s_1\mathrm{\mathbf{x}}_k\in F_{s_1s_2s_1},\quad 
    s_1s_2s_1\mathrm{\mathbf{x}}_k\in F_{s_2s_1}.
    \end{equation}
    Moreover, we have the following      
    \begin{align}
        \begin{split}\label{eq: vert para x_k dem}  
            s_1(\mathrm{\mathbf{x}}_k)&=\mathrm{\mathbf{x}}_k-2t_1\a_1,\\
            s_2s_1(\mathrm{\mathbf{x}}_k)&=s_1(\mathrm{\mathbf{x}}_k)-2t_2\a_2,\\
            s_1s_2s_1(\mathrm{\mathbf{x}}_k)&=s_2s_1(\mathrm{\mathbf{x}}_k)+2t_3\a_1,\\
            s_1s_2s_1(\mathrm{\mathbf{x}}_k)&=s_{\rho,0}(\mathrm{\mathbf{x}}_k)=\mathrm{\mathbf{x}}_k-2t\rho.
        \end{split}
    \end{align}
     for some $t_1,t_2,t_3,t\in\mathbbm{R}_{\geq 0}.$ 
    
    We will first prove the equation in the second case, i.e., $w\in \{s_2,s_1s_2\}$.
    From Equations \eqref{eq: regiones de vert para x_k}, it follows that there is no vertex on \( F_w \).
    Also, we have that the zones adjacent to $ F_{s_2} $ are $F_\id $ and $F_{s_2s_1}.$
    By Equation \eqref{eq: regiones de vert para x_k}, the vertices of~$\CC_{\mathrm{\mathbf{x}}_k}$ in~$F_\id$ and $F_{s_2s_1}$ are $\mathrm{\mathbf{x}}_k$ and  $s_1s_2s_1 \mathrm{\mathbf{x}}_k$, respectively.
    By Equation \eqref{eq: vert para x_k dem}, we have that $s_1s_2s_1\mathrm{\mathbf{x}}_k= \mathrm{\mathbf{x}}_k-2t\rho$, so 
    \begin{equation*}
        \CC_{\mathrm{\mathbf{x}}_k}\cap F_{s_2}= \left(\mathrm{\mathbf{x}}_k-\Cone(\a_1,\rho)\right)\cap F_{s_2},
    \end{equation*}
    and after intersecting both sides with $\cen(W)$ this  concludes the proof of the lemma in the case $w=s_2.$
    
    The adjacent zones to $F_{s_1s_2}$ are $F_{s_1}$ and $F_{s_1s_2s_1}$.
    By Equation \eqref{eq: regiones de vert para x_k}, the vertices of~$\CC_{\mathrm{\mathbf{x}}_k}$ in~$F_{s_1}$ and $F_{s_2s_1s_2}$ are $s_1\mathrm{\mathbf{x}}_k$ and  $s_2s_1 \mathrm{\mathbf{x}}_k$, respectively.
    By Equation \eqref{eq: vert para x_k dem}, we have that $s_2s_1\mathrm{\mathbf{x}}_k= s_1\mathrm{\mathbf{x}}_k-2t_2\alpha_2$ and $s_1(\mathrm{\mathbf{x}}_k)=\mathrm{\mathbf{x}}_k-2\t_1\a_1$, so
    \begin{align}
         \CC_{\mathrm{\mathbf{x}}_k}\cap F_{s_1s_2}&= \left(s_1\mathrm{\mathbf{x}}_k-\Cone(-\a_1,\a_2)\right)\cap F_{s_1s_2}\\
         &= \left(s_1\mathrm{\mathbf{x}}_k-\Cone^{s_1}
         (\a_1,\rho)\right)\cap F_{s_1s_2},
    \end{align}
    and after intersecting both sides with $\cen(W)$ this  concludes the proof of the lemma in the case $w=s_1s_2.$
    
    The result for the remaining cases follows similarly, with the caveat that for $F_{s_2s_1}$ and $F_{s_1s_2s_1}$ the equation is not true before intersecting with \(\cen(W)\).
    As in several previous arguments, a small triangle---without intersection with $\cen(W)$---is the difference between the left-hand side and the right-hand side that after intersecting with $\cen(W)$ disappears.
    For brevity, we omit further details.
\end{proof}
We conclude this part with two lemmas that study the Bruhat order between elements in~$F_\id$ that differ by a multiple of a simple root.
\begin{lem}\label{lem: comparable elements with nonempty intersection}
    Let $z,z'\in F_\id$ such that $\cen(z')-\cen(z)\in \mathbbm{R}_{>0}\alpha_i$ for some $i\in\{1,2\}$ and $\overline{z A_+}\cap\overline{z'A_+}\neq \emptyset$.
    Then $z< z'$ and $\ell(z,z')=1$.
\end{lem}
\begin{proof}
    By hypothesis we have that $z\in z'-\Cone(\alpha_1,\alpha_2)$.
    By Lemma~\ref{lem: local bruhat order}, $z\in \CC_{z'}$.
    By Corollary~\ref{cor: characterization of the Bruhat order}, this implies that $z\leq z'$.
    It remains to prove that $\ell(z,z')=1$.

    Let us suppose that $z'$ is down-oriented.
    In this case, we have that
    $\cen(z')-\cen(z)=\frac{1}{3}\alpha_i$ and $z,z'$ share an edge.
    It is a known fact from Coxeter complexes that if two elements $z,z'$ share a face, then there is $s\in S$ such that $zs=z'$.
    Since $z\leq z'$, we have that $\ell(z,z')=1$.
    
    Let us suppose that $z'$ is up-oriented.
    In this case, $\overline{z A_+}$ and $\overline{z'A_+}$ intersect in a vertex and
    \begin{equation}\label{eq: c(z)=c(z')-2/3}
        \cen(z)=\cen(z')-\frac{2}{3}\alpha_i.
    \end{equation} 
    Now we study all the possibilities for $z$ and $z'$.
    In each case, we verify that $\ell(z,z')=1$ by using the formulas for the length of the elements of~$W$ provided in Section~\ref{subsec: xtheta partition}.
    \begin{itemize}
        \item $z'= \theta^s(m,n)$, for  integers $m,n>0$.
        In this case, we have that $\ell(z')=2m+2n+4$.
        Then Equation \eqref{eq: c(z)=c(z')-2/3} implies that either $z=\theta(m-1,n+1)$ or $z=\theta(m+1,n-1)$.
        In either case we have that $\ell(z)=2m+2n+3$, so $\ell(z,z')=1$.
        \item $z'= \mathrm{\mathbf{x}}_{2k}$ or $\sigma(\mathrm{\mathbf{x}}_{2k})$ for $k\geq 2$.
        By symmetry, we can assume $z'=\mathrm{\mathbf{x}}_{2k}$.
        In this cas,e we have $\ell(z')=2k$.
        Since $\cen(z')-\frac{2}{3}\a_2\not\in F_\id$, we have that $i=1$, so
        $z=\theta(k-2,0)$.
        Thus $\ell(z)=2(k-2)+3=2k-1$ and $\ell(z,z')=1$.
        \item $z'=\theta^s(m,0)$ or $z'=\theta^s(0,n)$.
        By symmetry, we consider only the first case.
        In this case, we have $\ell(z')=2m+4$.
        Equation \eqref{eq: c(z)=c(z')-2/3} implies that either $z=\theta(m-1,1)$ or $z=\mathrm{\mathbf{x}}_{2m+3}$.
        In both cases we have $\ell(z)=2m+3$, so $\ell(z,z')=1$.\qedhere
    \end{itemize}
\end{proof}
\begin{lem}\label{lem: if distance is 4 then difference is more than 2 times the simple root}
    Let $x,y\in F_\id$ and $t\in\mathbbm{R}_{\geq 0}$ such that $z=y-t\alpha_i\in F_\id$.
    Then $\ell(z,y)\geq 4$ if and only if $t\geq 2$.
    Similarly, let $z=x+t\alpha_i\in F_\id$.
    Then $\ell(x, z)\geq 4$ if and only if $t\geq 2$.
\end{lem}
\begin{proof}
    We will prove only the first statement; the second is analogous.
    Let us define $C \coloneqq\{u\in W \mid u\in \Sgm(y,z)\}$.
    Lemma~\ref{lem: local bruhat order} implies that $\Sgm(y,z) \subset [z,y]$ and any pair of elements of~$W$ in~$\Sgm(y,z)$ is comparable.
    This implies that $C \subset [z,y]$ and any pair of elements in~$C$ is comparable.

    We will enumerate the elements of~$C$ as follows:
    we define $u_0\coloneqq z$.
    Since $z$ is the end-point of~$\Sgm(y,z)$, we define $u_1$ as the unique element in~$C$ such that $\overline{u_1A_+}\cap \overline{u_0A_+}\neq \emptyset$.
    For $k\geq 1$, if $u_k$ is defined and $u_k\neq y$, then we define $u_{k+1}$ as the unique element in~$C\setminus \{u_{k-1}\}$ such that $\overline{u_{k+1}A_+}\cap \overline{u_kA_+}\neq \emptyset$.
    We define $n\in\mathbbm{N}$ such that $u_n=y$.
    
    By Lemma~\ref{lem: comparable elements with nonempty intersection}, we have that $C$ is a maximal chain in~$[z,y]$.
    In other words, we have $n=\ell(z,y)$ and
    \begin{equation*}
        z=u_0\lessdot u_1 \lessdot \hdots \lessdot u_n=y.
    \end{equation*}
    For $1\leq k\leq n$, if $u_{k-1}$ is up-oriented then $u_k$ is down-oriented and $u_{k-1}=u_k-\frac{1}{3}\alpha_i$.
    If $u_{k-1}$ is down-oriented then $u_{k}$ is up-oriented and $u_{k-1}=u_k-\frac{2}{3}\alpha_i$.
    This implies the following facts.
    \begin{itemize}
        \item If $u_{k-1}=u_{k}-\frac{1}{3}\alpha_i$ then $u_k=u_{k+1}-\frac{2}{3}\alpha_i$.
        \item If $u_{k-1}=u_{k}-\frac{2}{3}\alpha_i$ then $u_k=u_{k+1}-\frac{1}{3}\alpha_i$.
    \end{itemize}
    From the facts above we have that $u_k-u_l\in \mathbbm{R}_{\geq 0}\alpha_i$ if and only if $l\leq k$.

    Note that $t\geq 2$ if and only if $u_n-u_4 \in \mathbbm{R}_{\geq 0}\alpha_i$ (note that $u_4=z+2\alpha_i$) if and only if $n=\ell(z,y)\geq 4$.
\end{proof}
\begin{rem}
    We note the following simple generalization, which we do not prove here as it is not required for this paper.
    If $z,z'\in F_\id$ are such that $\cen(z')-\cen(z)=r\alpha_i$ for some $r>0$ and $i\in\{1,2\}$,
    then $\ell(z,z')=\lfloor r \rfloor+\lceil r\rceil$.
\end{rem}
\subsection{The geometry of dihedral subintervals}\label{section: maximal dihedral}
Let $x,y\in D$ be dominant elements.
Consider the sets $\mathcal{D}^u(y)\coloneqq\mathcal{D}(y)=\{w\in W\mid [w,y]\text{ is dihedral}\}$ and $\mathcal{D}^l(x)\coloneqq\{w\in W\mid [x,w]\text{ is dihedral}\}$\footnote{Here the superscript $u$ and $l$ stands for ``upper'' and ``lower'', respectively.}.
We define $\DC^u_{\mathrm{sgm}}(y):=\{z\in W\mid z\in \operatorname{Sgm}(y,s_1y)\cup \operatorname{Sgm}(y,s_2y)\}.$ 
\begin{defn} \label{def: Dsgmx Drestx}
    Let $x\in D$ and $\{u_i\}_{1\leq i\leq 6}= \mathrm{EVS}_x$ be the exterior vertices of~$\mathcal{St}(x)$.
    We define 
    \begin{align*}
        \DD_{\mathrm{sgm}}^l(x)&\coloneqq\{z \geq x\mid z\in\Sgm(u_1,u_3)\cup \Sgm(u_2,u_6)\}.
    \end{align*}
\end{defn}
The main goal of the section is to prove the following proposition:
\begin{prop}\label{prop: Dih-resumen}
    Let $x,y\in D$ and $z\in W$.
    Then 
    \begin{enumerate}
        \item\label{item-prop: Dih-rsm for y} If $z\in \DC^u(y)$ and $\ell(z,y)\geq 4$, then $z\in \DC^{u}_{\mathrm{sgm}}(y)$.
        Furthermore, if $z\in \DC^{u}_{\mathrm{sgm}}(y)$ then $z\in \DC^u(y)$.
        \item\label{item-prop: Dih-rsm for x} If $z\in\DC^l(x)$ and $\ell(x,z)\geq 4$, then $z\in \DC^{l}_{\mathrm{sgm}}(x)$.
        Furthermore, if  $z\in \DC^{l}_{\mathrm{sgm}}(x)$ then $z\in \DC^l(x)$.
    \end{enumerate}
\end{prop}
Before the proof, we need some definitions and a better description of the sets $\DC^u(y)$ and $\DC^l(x)$.
\begin{nota}
    We call an element $y\in W$ \emph{left-dominant} if $y\in D$ and $y\in \mathbbm{R}\rho$ or $y$ lives in the left half-space determined by the line $\mathbbm{R}\rho$.
    We define similarly \emph{right-dominant.}
\end{nota}
For the next definition, we suppose that $y$ is left-dominant.
In other words, $y=\theta(m,n)$ or $y=\theta^s(m,n)$ with $m\geq n\geq 0$.
Let $W_{\geq 2,y}\coloneqq\{s_1s_2y,s_2s_1y,s_1s_2s_1y\}$, we have
\begingroup
\allowdisplaybreaks
\begin{align*}
    \DC^u_{\mathrm{rest}}(\theta(m,0)) &= 
        W_{\geq 2,y} \uplus \{\mathbf{x}_{2m}, \mathbf{x}_{2m+1}, \theta(m{-}1,0)\}, \\[0.5ex]
    \DC^u_{\mathrm{rest}}(\theta(m,n)) &= 
        W_{\geq 2,y} \uplus \{ \theta^s(m{-}1,n{-}1), \theta(m,n{-}1), \theta(m{-}1,n), \\
        &\quad s_1\theta(m,n{-}1), s_2\theta(m{-}1,n) \}, \\[0.5ex]
    \DC^u_{\mathrm{rest}}(\theta^s(m,0)) &= 
        W_{\geq 2,y} \uplus \{ \mathbf{x}_{2m+2}, s_1\mathbf{x}_{2m+3}, \theta(m,0), \\
        &\quad \theta^s(m{-}1,n), s_2\theta(m{-}1,n) \}, \\[0.5ex]
    \DC^u_{\mathrm{rest}}(\theta^s(m,n)) &= 
        W_{\geq 2,y} \uplus \{ \theta(m,n), \theta^s(m,n{-}1), \theta^s(m{-}1,n), \\
        &\quad s_1\theta(m,n), s_2\theta(m,n), s_2\theta(m{-}1,n{+}1), s_1\theta(m{+}1,n{-}1) \}.
\end{align*}
\endgroup
The following lemma holds by simple case-by-case verification.
\begin{lem}\label{lem: dih menores que 3}
    We have $\ell(z,y)\leq 3$, for all  $z\in \DC^u_{\mathrm{rest}}(y)$.
\end{lem}
\begin{lem}\label{lem: descripcion dih y} 
    If $y$ is left-dominant then $\mathcal{D}(y)=\DC^u_{\mathrm{sgm}}(y)\uplus \DC^u_{\mathrm{rest}}(y)$.
    If $y$ is right-dominant then $\mathcal{D}(y)=\DC^u_{\mathrm{sgm}}(y)\uplus \sigma(\DC^u_{\mathrm{rest}}(y))$.
\end{lem}
\begin{proof}
    We prove the lemma for $y$ left-dominant; the argument for $y$ right-dominant is symmetric.
    Let $z \in \mathcal{D}(y)$ be such that $\ell(z,y)>1$.
    By Proposition~\ref{prop: Dyer dihedral}, the interval $[z, y]$  contains exactly two coatoms.
    If $c$ is a coatom of~$[z, y]$, then $c\in \operatorname{L}(y)$.
    Let us label the elements of~$\operatorname{L}(y)$ by $c_1, c_2, \ldots, c_k$.
    By Proposition~\ref{prop: Dyer dihedral}, and Corollary~\ref{cor: characterization of the Bruhat order} 
    there are two different indexes $i,j$ in~$\{1,\ldots,k\}$, such that $\cen(z) \in \CC_{c_i} \cap \CC_{c_j}$ and $\cen(z) \notin \CC_{c_t}$, for any $t\notin \{i,j\}$.
    So, if we define 
    \begin{equation*}
        A_{i, j}\coloneqq (\CC_{c_i}\cap \CC_{c_j})\setminus\left( \bigcup_{y\notin \{i,j\}} \CC_{c_t}\right),
    \end{equation*}
    we have that 
    \begin{equation}\label{eq: dihedral}
        \mathcal{D}(y)=\operatorname{L}(y)\cup \left(\bigcup_{i,j}A_{i,j}\right).
    \end{equation}
    Before proceeding with the proof, we establish the following claim.
    \begin{claim}
        Let $y$ be left-dominant.
        If $c\in \operatorname{L}(y)$, then we have
        \begin{equation*}
            \CC_{c}=
            \begin{cases}
                \operatorname{Conv}(W_f\cdot c)&\mbox{if $c\in D$,}\\
                \operatorname{Conv}(\{c,y-\alpha_1,s_2(y-\alpha_1),s_1s_2y,s_2s_1y,s_1s_2s_1y\})&\mbox{if $c=s_1y$,}\\
                \operatorname{Conv}(\{c,y-\alpha_2,s_1(y-\alpha_2),s_1s_2y,s_2s_1y,s_1s_2s_1y\})&\mbox{if $c=s_2y$ and $c\not\in X$,}\\
                \operatorname{Conv}(\{c,s_1s_2y,s_1s_2s_1y,s_2s_1y\})&\mbox{if $c=s_2y$ and $c\in X$.}
            \end{cases}
        \end{equation*} 
    \end{claim}
    \begin{proof}
        By Lemma~\ref{lem: set of lower covers}, we have either $c\in D$ or $c\in  \{s_1y,s_2y\}$.
        If $c\in D$, the result follows directly from Corollary~\ref{cor: characterization of the Bruhat order}.
        It remains the case when $c\in \{s_1y,s_2y\}.$ 
        
        First, assume that $c=s_1y$.
        By Corollary~\ref{cor: characterization of the Bruhat order}, we have that 
        \begin{align} 
            \CC_c&=\operatorname{Conv}(\{c,s_2c,s_0s_2c,s_1s_2c,s_2s_0s_2c,s_2s_1s_2c\}).
        \end{align}
        The result follows from the identities $s_0=s_{\rho,-1}=s_1s_2s_1-\rho$ and $s_2\alpha_1=\rho$ (for example, we observe that 
        $s_0s_2c=s_1s_2s_1s_2c-\rho=s_2s_1c-\rho=s_2y-\rho=s_3(y-\alpha_1)$).
    
        Now, assume that  $c=s_2y$.
        Since $y$ is left dominant, we have either $c\in \delta s_0D$ or $c\in\{\mathrm{\mathbf{x}}_{2m+2},\mathrm{\mathbf{x}}_{2m+3}\}$.
        The first case is analogous to the case $c=s_1y$.
        Therefore, it remains to consider that case $c\in\{\mathrm{\mathbf{x}}_{2m+2},\mathrm{\mathbf{x}}_{2m+3}\}$, which follows directly from Corollary~\ref{cor: characterization of the Bruhat order}.
    \end{proof}  
    Returning to the proof of the lemma, we first consider the case $y=\theta(5,2)$---we will later show that this case is generic for $\theta(m,n)$.
    In this case $\vert \operatorname{L}(y)\vert =k=4$.
    \begin{figure}[ht!]
        \centering
        \begin{subfigure}[b]{0.45\textwidth}
            \centering
            \includegraphics[width=0.9\textwidth]{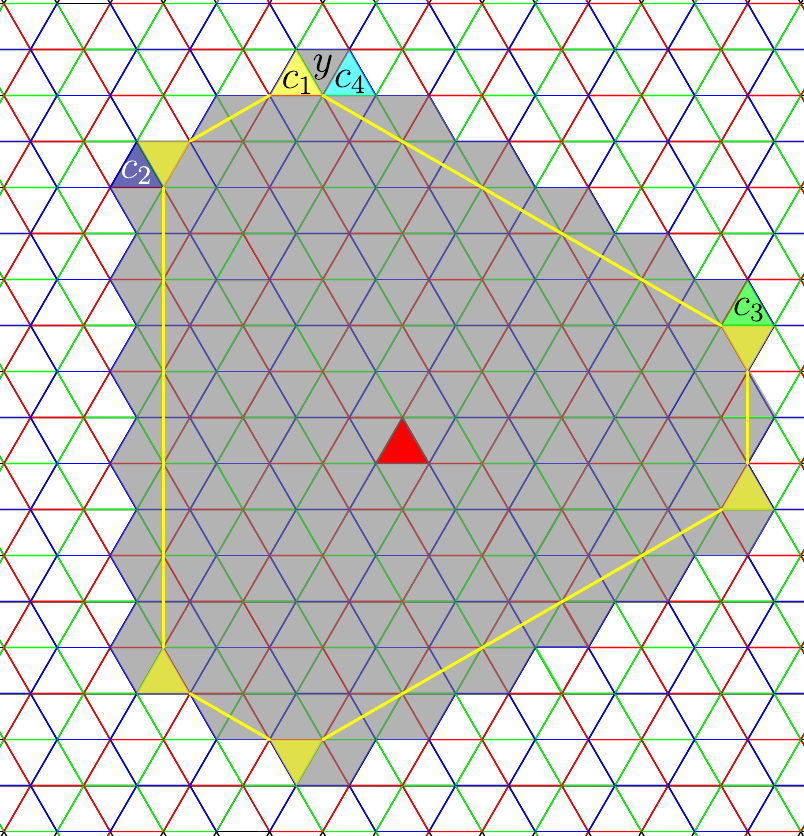}
            \caption{The hexagon $\CC_{c_1}$.\newline} 
            \label{fig: proofdiha}
        \end{subfigure}
        \hfill
        \begin{subfigure}[b]{0.45\textwidth}
            \centering
            \includegraphics[width=0.9\textwidth]{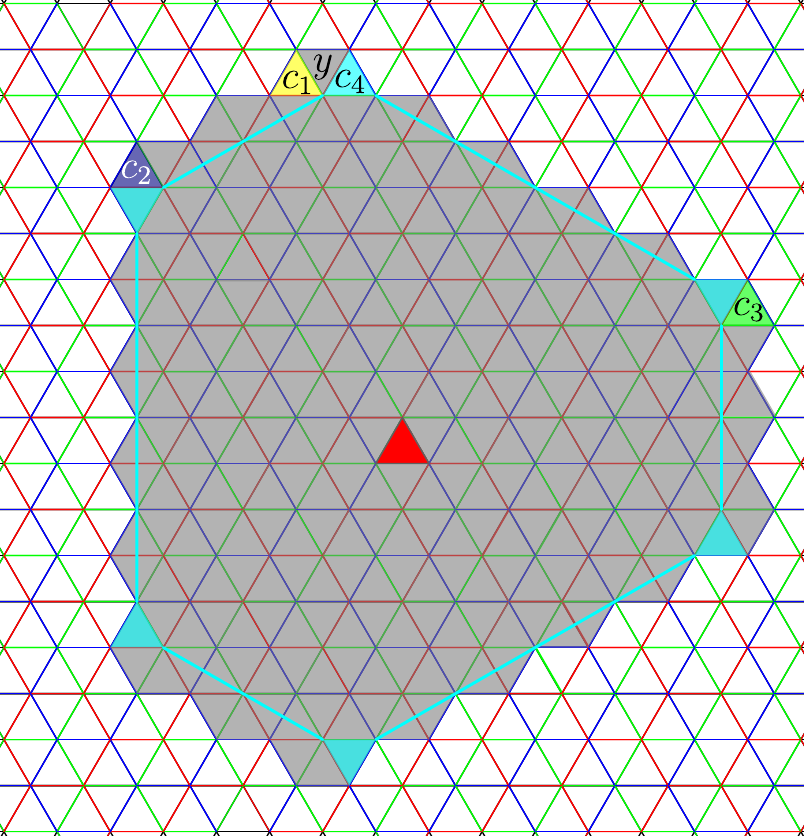}
            \caption{The hexagon $\CC_{c_4}$.\newline} 
            \label{fig: proofdihb}
        \end{subfigure}
        \vfill
        \begin{subfigure}[b]{0.45\textwidth}
            \centering
            \includegraphics[width=0.9\textwidth]{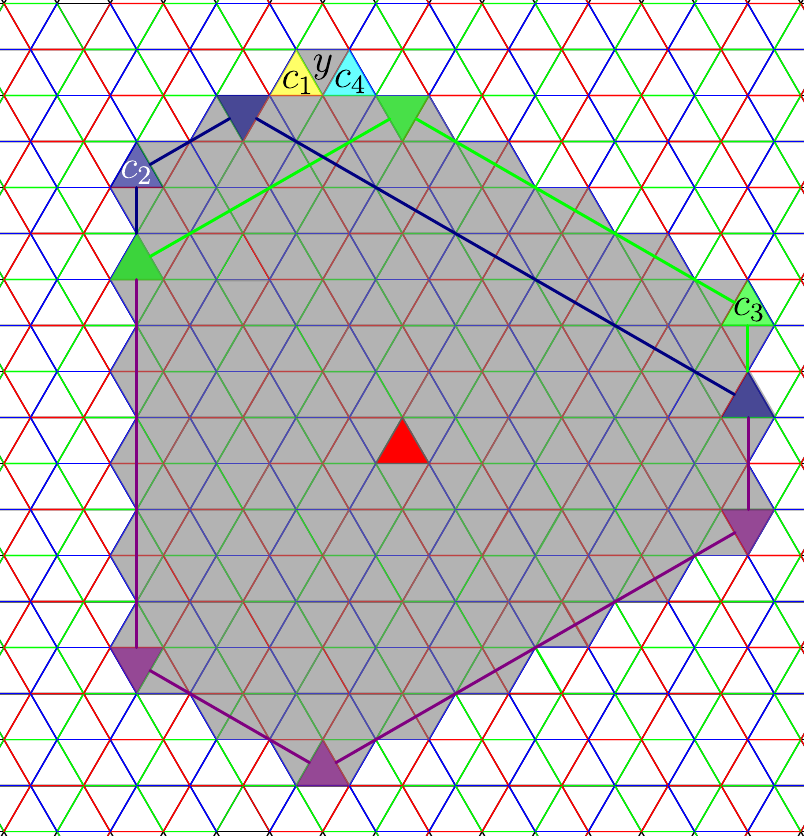}
            \caption{The hexagons $\CC_{c_2}$, $\CC_{c_3}$, and their intersection.}
            \label{fig: proofdihc}
        \end{subfigure} 
        \hfill
        \begin{subfigure}[b]{0.45\textwidth}
            \centering
            \includegraphics[width=0.9\textwidth]{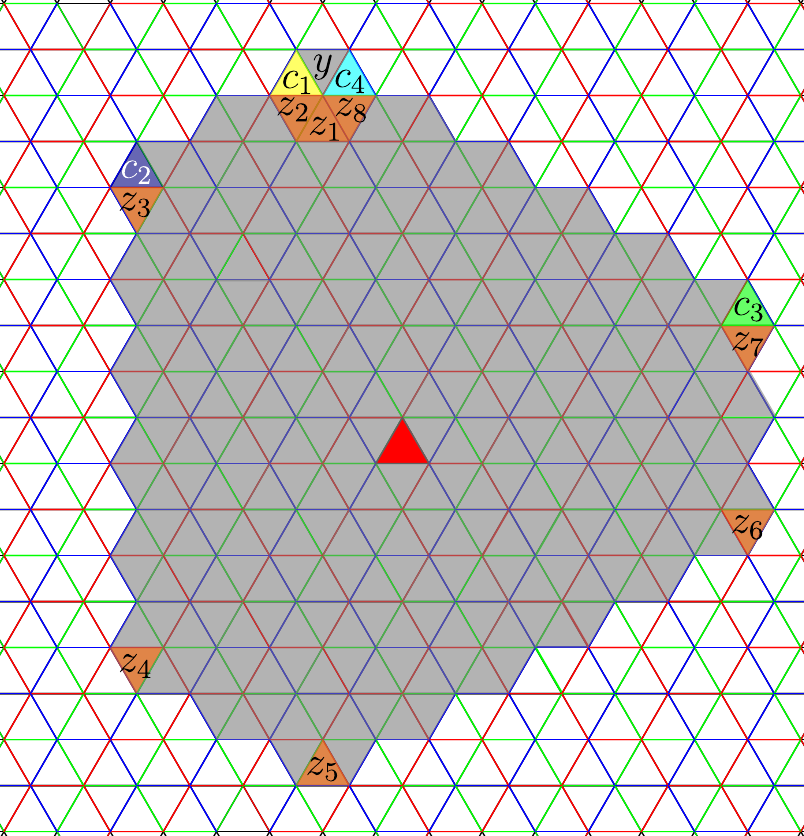}
            \caption{In orange, the eight elements of $\mathcal{D}_{\mathrm{rest}}(y)$.}
            \label{fig: proofdihd}
        \end{subfigure} 
        \caption{Polygons $\CC_z$ for $z\in \operatorname{L}(\theta(5,2))$.}
        \label{fig: Proof Dihedral}
    \end{figure}
    In Figures~\ref{fig: proofdiha},~\ref{fig: proofdihb}, and~\ref{fig: proofdihc}, we have drawn the hexagons $\CC_c$ for all $c\in \operatorname{L}(y)=\{c_1,c_2,c_3,c_4\}$.
    Also, in Figure~\ref{fig: proofdihc}, the purple hexagon corresponds to the intersection $\CC_{c_2}\cap \CC_{c_3}$.
    In Figure~\ref{fig: proofdihd}, the orange alcoves represent the elements of~$\DC^u_{\mathrm{rest}}(y)$.
    A quick inspection shows that: 
    \begin{enumerate}[label=\roman*]
        \item\label{item: 1 proof} $\cen(z)\in A_{1,2}$ if and only if $z \in \operatorname{Sgm}(y,s_2y)\setminus \{c_1,c_2\}$.
        \item\label{item: 2 proof} $\cen(z)\in A_{3,4}$ if and only if $z \in \operatorname{Sgm}(y,s_1y)\setminus \{c_3,c_4\}$.
        \item\label{item: 3 proof} In the remaining cases, we have:
        \begin{equation*}
             A_{1,4}=\{z_1,z_2,z_8\},\quad A_{1,3}=\{z_7\},\quad A_{2,4}=\{z_3\},\quad A_{2,3}=\{z_4,z_5,z_6\}.
        \end{equation*} 
    \end{enumerate}
    Therefore, using Equation \eqref{eq: dihedral} we have proved in this case that $z\in\DC^u(y)$ if and only if $z\in \DC^u_{\mathrm{sgm}}(y)\uplus \DC^u_{\mathrm{rest}}(y)$.

    We claim that the argument for a general $\theta(m,n)$ is essentially the same.
    Lemma~\ref{lem: set of lower covers} tell us that $\mathrm{L}(\theta(m,n))= \{s_1 y,s_2y,\theta^s(m-1,n),\theta^s(m,n-1)\}$.
    This means that $c_1:=\theta^s(m-1,n)$ and $c_4:=\theta^s(m,n-1)$ are the two alcoves sharing an edge with $y$ different from $\theta^s(m,n)$, $c_2:=s_2y$ and $c_3:=s_1y$.
    We will still have that $z_2$ (resp.\@ $z_8$) is the alcove below $c_1$ (resp.\@ $c_4$).
    We have that $z_1$ is the only alcove sharing an edge with both $z_2$ and $z_8$.
    The alcove $z_7$ is the one below $c_3$, and so on.
    All of this implies that the calculations of~$A_{i,j}$ give the exact same results as in \ref{item: 1 proof}, \ref{item: 2 proof}, and \ref{item: 3 proof} above.
    
    The proof for $\theta^s(m,n)$ is similar to the previous case, so we omit the details.
\end{proof}

We already have a description of~$\DD^u(y)$, so our next task is to describe $\DD^l(x)$.
To do this, we need to consider the sets 
\begin{equation*}
    \DD^l_{\mathrm{rest}, \leq 2}(x)\coloneqq \{z\in W\setminus \DD_{\mathrm{sgm}}^l(x)\mid x\leq z\mbox{ and }  \ell(x,z)\leq2\},
\end{equation*}
and
\begin{align*}
    \DC^l_{\mathrm{rest}}(x)\coloneqq& 
    \begin{cases}
        \DD^l_{\mathrm{rest}, \leq 2}(x)&\mbox{ if  $x=\theta(m,n)$,}\\
        \DD^l_{\mathrm{rest}, \leq 2}(x)\uplus \{\theta(m+1,n+1)\}&\mbox{ if  $x=\theta(m,n)s$.}
    \end{cases}
\end{align*}

The following is the analog of Lemma~\ref{lem: descripcion dih y} for $\DD^l(x)$.
\begin{lem}\label{lem: descripcion dih x} 
    Let $x\in D$ be a dominant element.
    Then $\DD^l(x)=\DD_{\mathrm{sgm}}^l(x)\uplus\DD_{\mathrm{rest}}^l(x)$.
\end{lem}
\begin{proof} 
    Assume that $x\in\{\theta(m,n),\theta^s(m,n)\mid m,n\geq1\}$.
    Let $z \in \mathcal{D}^l(x)$ be such that $\ell(x,z)>1$.
    By Proposition~\ref{prop: Dyer dihedral}, the interval $[x, z]$ contains exactly two atoms.
    By Lemma~\ref{lem: cardinality of upper covers}, we have that 
    \begin{equation}
        \operatorname{U}(x)=\left\{
        \begin{array}{@{}l@{\thinspace}l}
            \{\theta^s(m,n),\theta^s(m+1,n-1),\theta^s(m-1,n+1),& \\
            s_0\theta(m,n),  \delta(s_0\theta(m-1,n+1)),\delta^2(s_0\theta(m+1,n-1))\} &\mbox{ if $x=\theta(m,n)$},\medskip\\
            \{\theta(m+1,n),\theta(m,n+1), s_0\theta^s(m,n),&\\
            \d(s_0\theta^s(m-1,n+1)),
            \delta^2(s_0\theta^s(m+1,n-1))\}&\mbox{ if $x=\theta^s(m,n)$}.\medskip
        \end{array}
        \right.
    \end{equation}
    It is evident that if $a$ is an atom of~$[x,z]$, then $a\in \operatorname{U}(x)$.
    We will now proceed by outlining the main ideas of the proof:
   
    \noindent\textbf{The inclusion $\subset$.}
    \begin{enumerate}
        \item Note that $\operatorname{U}(x)\subset \CC_{x+\rho}$.
        This observation implies that $[x,x+\rho]$ is not dihedral.
        \item Define the sets
        \begin{align*}
            A&=\DD_{\mathrm{sgm}}^l(x)\uplus \DD_{\mathrm{rest}}^l(x)\quad \mathrm{and}\\
            D^w(x)&=\big(\mathcal{St}^\circ(x+\rho)\setminus (\mathcal{St}^\circ(x)\cup A)\big) \cap F_w.
        \end{align*}
        \item  For $b\in D^\id(x)$, the description of~$\operatorname{U}(x)$ from above implies that  $|\operatorname{U}(x)\cap \CC_{b}|\geq 3$.
        \item  Let $w\in W_f\setminus\{\id\}$ and let $b\in D^w(x)$.
        By Lemma~\ref{lem: Fact2}, we know that $b(1)\in V(\CC_b)$.
        Since $b(1)\in D^\id(x)$, it follows that  $[x,b(1)]$ is not dihedral.
        Consequently, $[x,b]$ is also not dihedral.
        This proves that 
        if $b\in \cup_{w\in W_f}D^w(x)$, then $b\not\in \DD^l(x)$.
        Equivalently, as $\cup_{w\in W_f}D^w(x)$ is the complement of~$A$ in~$\mathcal{St}^\circ(x+\rho)$ we have that $\DD^l(x)\subset\DD_{\mathrm{sgm}}^l(x)\uplus\DD_{\mathrm{rest}}^l(x)$.
    \end{enumerate}

    \noindent\textbf{The inclusion $\supset$.}
    We consider two cases
    \begin{enumerate}
        \item Suppose that $z\in  \mathcal{D}^{l}_{\mathrm{rest}}(x)$.
        \begin{itemize}
            \item If $\ell(x,z)=1$, then $[x,z]$ is dihedral.
            \item If $\ell(x,z)=2$ then $[x,z]$ is also dihedral because any interval of length~$2$ contains exactly two atoms.
            \item In the case where $z=\theta(m+1,n+1)$.
            The atoms of~$[x,z]$ are $\theta(m+1,n)$ and $\theta(m,n+1)$.
        \end{itemize}
        \item Suppose that $z\in\DD_{\mathrm{sgm}}^l(x)$.
        \begin{itemize}
            \item If $z\in \Sgm(u_1,u_3)$.
            The of atoms of~$[x,z]$ are 
            \begin{equation}
                \begin{cases}
                    \{\theta^s(m+1,n-1),\delta^2(s_0\theta(m+1,n-1))\}&\mbox{ if  $x=\theta(m,n)$,}\\
                    
                    \{\theta(m+1,n),\delta^2(s_0\theta^s(m+1,n-1))\}&\mbox{ if  $x=\theta^s(m,n)$.}
                \end{cases}
            \end{equation}
            \item If $z\in \Sgm(u_2,u_6)$.
            The atoms of~$[x,z]$ are 
            \begin{equation}
                \begin{cases}
                    \{\theta^s(m-1,n+1),\delta(s_0\theta(m-1,n+1))\}&\mbox{ if  $x=\theta(m,n)$,}\\
                    
                    \{\theta(m,n+1),\delta(s_0\theta^s(m-1,n+1))\}&\mbox{ if  $x=\theta^s(m,n)$.}
                \end{cases}
            \end{equation}
        \end{itemize}
        In both cases $[x,z]$ has two atoms, which implies that $z\in \DD^{l}(x)$.
    \end{enumerate}
    The case $x\in\{\theta(m,0),\theta^s(m,0),\theta(0,m),\theta^s(0,m)\mid m>0\}$ is similar so we omit it.
\end{proof}
\begin{proof}[Proof of Proposition~\ref{prop: Dih-resumen}]
    \begin{enumerate}
        \item The first part follows from Lemmas~\ref{lem: dih menores que 3} and~\ref{lem: descripcion dih y}.
        The second part follows from the inclusion~$\supset$ in Lemma~\ref{lem: descripcion dih y}.
        \item The first part follows from Lemma~\ref{lem: descripcion dih x}, the definition of~$\DC^l_{\mathrm{rest}}$, and the fact that $\ell(\theta^s(m,n), \theta(m+1,n+1))=3$.
        The second part follows from the inclusion~$\supset$ in Lemma~\ref{lem: descripcion dih x}.\qedhere
    \end{enumerate}
\end{proof}

\subsection{Types of intervals and the polygon \texorpdfstring{$\operatorname{Pgn}_{x,y}$}{Pgn(x,y)}}\label{subsec: polygon}
We now define the polygon $\Pgn_{x,y}$, which plays a central role in our geometric description of intervals.
Recall that at the beginning of Section~\ref{section: geometry of intervals}, we assume that $x$ and $y$ are dominant elements.
As $[x,y]=[\mathrm{id},y]\cap (\geq x)$, Corollary~\ref{cor: characterization of the Bruhat order} and Proposition~\ref{prop: no mayores geometrico} give the following geometric description of the interval 
\begin{equation}
    \begin{aligned}\label{eq: interval as centers of hexagon without the star}
        [x,y]&=\{ z\in W \mid \operatorname{cen}(z) \in \mathcal{C}_y\setminus \mathcal{St}^\circ(x)\}\\
        &=\cen^{-1}(\mathcal{C}_y\setminus \mathcal{St}^\circ(x)).
    \end{aligned}
\end{equation}
This description was illustrated in Figure~\ref{fig: geometric realization of a Bruhat interval} of the introduction.
Figure~\ref{fig: PGN} illustrates the next two definitions.
\begin{defn}\label{def: paralelogramo}
    We define the polygon $\operatorname{Par}_{x,y}^w$ as the parallelogram with opposite vertices $\{\operatorname{cen}(x_w), w\operatorname{cen}(y)\}$ and edges parallel to $w\alpha_1$ and $w\alpha_2$.
    In the degenerate case, i.e., when  $\operatorname{cen}(x_w)$ and $ w\operatorname{cen}(y)$ differ by a multiple of~$w\alpha_1$ or $w\alpha_2$, we  denote $\operatorname{Par}_{x,y}^w:= \Sgm(\operatorname{cen}(x_w),w\operatorname{cen}(y))$ and still call it a parallelogram.
    We denote $\operatorname{Par}_{x,y}:=\operatorname{Par}_{x,y}^{\mathrm{id}}$.
\end{defn}
\begin{defn}\label{def: polygon}
    For each $w\in W_f$, we define the following convex bounded polygon 
    \begin{equation}
        \operatorname{Pgn}_{x,y}^w= \operatorname{Par}_{x,y}^w\cap\, F_{w}.
    \end{equation}
    As before, we denote $\operatorname{Pgn}_{x,y}:=\operatorname{Pgn}_{x,y}^{\mathrm{id}}$.
\end{defn}
\begin{rem}\label{rem: geometry-of-vertex-of-Pgn}
    \begin{enumerate}
        \item  It is not hard to see that each element of~$\mathrm{V}(\operatorname{Par}_{x,y}^w)$ is either the center or a vertex of an alcove.
        \item Vertices of~$\Pgn_{x,y}$ adjacent to $\cen(x)$ can be either a vertex or the center of an alcove.
        \item Vertices of~$\Pgn_{x,y}$ adjacent to $\cen(y)$ can be the midpoint of an edge, the center, or a vertex of an alcove.
        \item\label{rem: vertex of Pgn} In some cases, $w\cen(y)$ is not contained in~$F_w$, so it is not a vertex of~$\Pgn_{x,y}^{w}$.
    \end{enumerate}
\end{rem}
\begin{lem}
    $\cen^{-1}(\operatorname{Pgn}_{x,y}^w)=\cen^{-1}(F_w\cap \mathcal{C}_y\setminus \mathcal{St}^\circ(x))$.
\end{lem}
\begin{proof}
    The lemma follows by proving $ F_w\cap \big(\CC_y\setminus \mathcal{St}^\circ(x)\big)\cap \cen(W)=\Pgn_{x,y}^w\cap \cen(W)$.
    Note that $F_w\setminus \mathcal{St}^\circ(x)=F_w\cap \big(x_w+\Cone^w(\alpha_1,\alpha_2)\big)$.
    Thus we obtain
    \begingroup
    \allowdisplaybreaks
    \begin{align*}
        F_w\cap\big( \CC_y\setminus \mathcal{St}^\circ(x)\big)\cap \cen(W)&=(F_w\setminus\mathcal{St}^\circ(x))\cap \CC_y\cap \cen(W)\\
        &=F_w\cap \big(x_w+\Cone^w(\alpha_1,\alpha_2)\big)\cap \CC_y\cap \cen(W)\\
        &=\big(F_w\cap\CC_y\cap \cen(W)\big)\cap\big(x_w+\Cone^w(\alpha_1,\alpha_2)\big)\\
        &=\big(F_w\cap \big(wy-\Cone^w(\alpha_1,\alpha_2)\big)\cap\cen(W)\big)\cap\\
        & \hspace{14pt}\big(x_w+\Cone^w(\alpha_1,\alpha_2)\big)\\
        &=F_w\cap \Par_{x,y}^w\cap \cen(W)\\
        &=\Pgn_{x,y}^w\cap\cen(W).
    \end{align*}
    \endgroup
    The fourth equality is given in Lemma~\ref{lem: gral local bruhat order}.
    The fifth and sixth equalities follow by definition of~$\Par_{x,y}^w$ and $\Pgn_{x,y}^w$, respectively.
\end{proof}
Therefore, we have 
\begin{equation}\label{eq: interval as union of polygons} 
    [x,y]=\biguplus_{w\in W_f}\operatorname{cen}^{-1}(\operatorname{Pgn}_{x,y}^w).
\end{equation}
In other words, $[x,y]$ is completely determined by the six polygons $\operatorname{Pgn}_{x,y}^w$.
We can now state an important definition, although we will not use it until Section~\ref{section: translations}.
\begin{figure}[!ht]
    \centering
    \begin{subfigure}[b]{0.48\textwidth}
        \centering
        \includegraphics[width=0.9\textwidth]{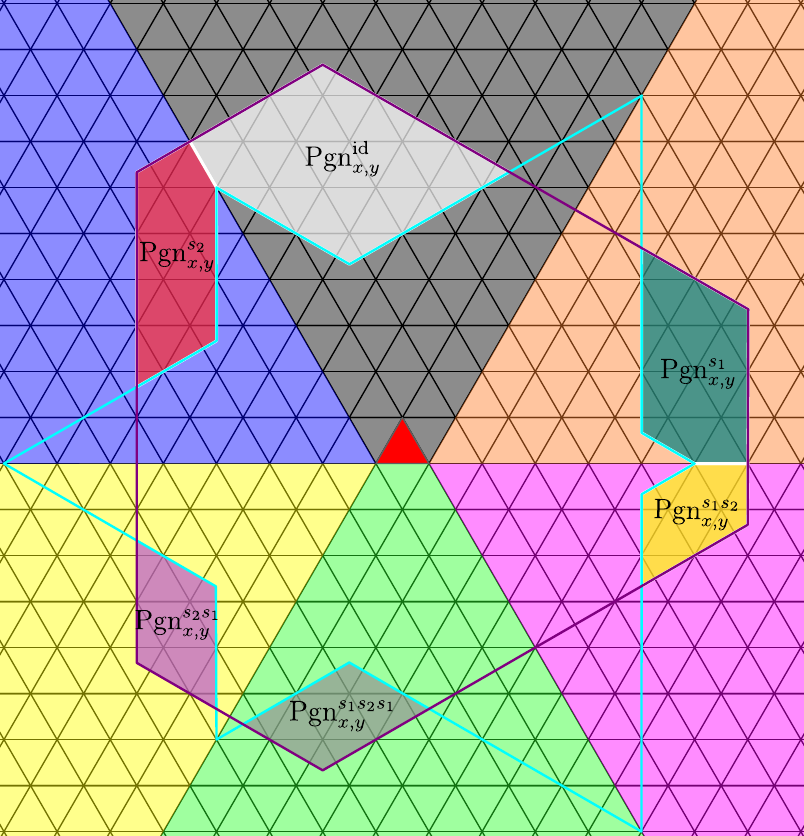}
        \caption{The polygons $\Pgn_{x,y}^w$.\newline}
        \label{fig:Pgn}
    \end{subfigure}
    \hfill
    \begin{subfigure}[b]{0.48\textwidth}
        \centering
        \includegraphics[width=0.9\textwidth]{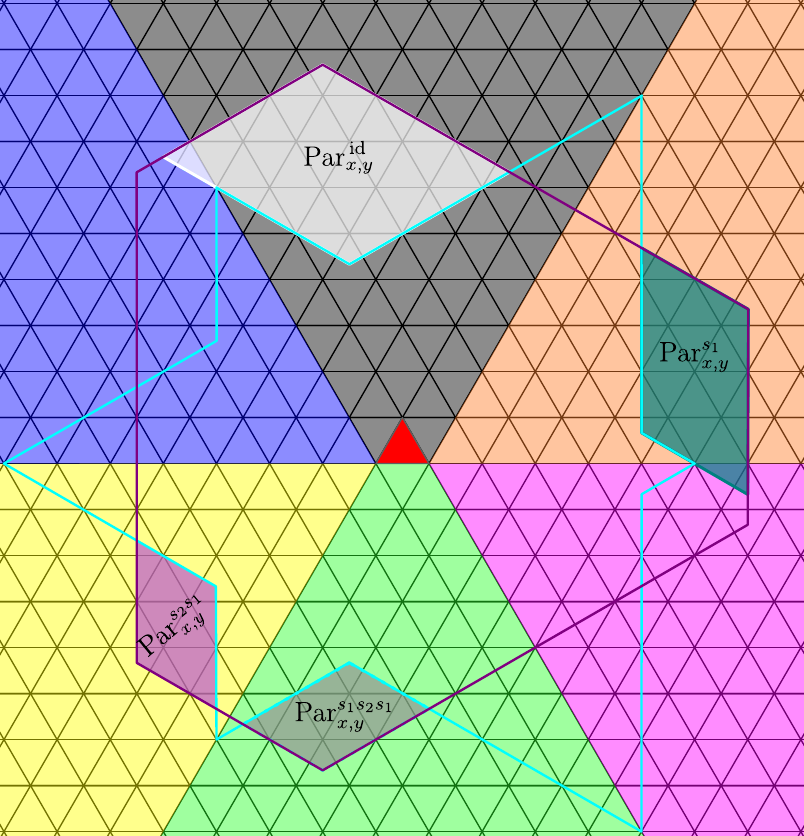}
        \caption{The parallelograms $\Par_{x,y}^\id,\Par_{x,y}^{s_1}\Par_{x,y}^{s_2s_1}$ and $\Par_{x,y}^{s_1s_2s_1}$.}
        \label{fig: three sin x}
    \end{subfigure}
    \caption{Polygons and parallelograms associated with the interval $[\theta^s(2,0),\theta(5,2)]$.} 
    \label{fig: PGN}
\end{figure}
\begin{defn}\label{def: types of intervals}
    We call $[x,y]$ a \emph{parallelogram interval} (resp.\@ \emph{pentagon interval}, \emph{hexagon interval}) if $\operatorname{Pgn}_{x,y}^{}$ is a parallelogram (resp.\@ pentagon, hexagon).
    If the area of~$\operatorname{Pgn}_{x,y}^{}$ is zero, i.e., $\operatorname{Pgn}_{x,y}^{}$ is a line segment, we also say that $[x,y]$ is a \emph{parallelogram interval}.
\end{defn}
We will now describe $\operatorname{Pgn}_{x,y}^w$ in terms of~$\operatorname{Pgn}_{x,y}$.
The polygons $\operatorname{Pgn}_{x,y}^w$ have equal or smaller area than $\operatorname{Pgn}_{x,y}$.
We will shortly see that one can obtain them by subtracting a particular set $S^{x,y}_{w}$ from $\operatorname{Pgn}_{x,y}$ before applying $w$.
Before proving this, we need some definitions.

Consider the  strips in the plane: 
\begin{align*}
    S^{x,y}_{\alpha_1}&:=\{\operatorname{cen}(x)+t\alpha_1+t'\alpha_2\in \operatorname{Par}_{x,y}^{}\mid t\in[0,1), t'\in \mathbbm{R}\},\\
    S^{x,y}_{\alpha_2}&:=\{\operatorname{cen}(x)+t\alpha_2+t'\alpha_1\in \operatorname{Par}_{x,y}^{}\mid t\in[0,1), t'\in \mathbbm{R}\}.
\end{align*}
Furthermore, for $w\in W_f$, we define 
\begin{equation*}
    S^{x,y}_{w}\coloneqq 
    \begin{cases}
        \emptyset & \mbox{if  $w=\id$},\\
        S^{x,y}_{\alpha_1}&\mbox{if  $w=s_1$},\\
        S^{x,y}_{\alpha_2}&\mbox{if  $w=s_2$},\\
        S^{x,y}_{\alpha_1}\cup S^{x,y}_{\alpha_2}&\mbox{otherwise},
    \end{cases}\quad\mbox{and}\quad
    v_w\coloneqq\begin{cases}
        0 &\mbox{ if $w=\id$,}\\
        \alpha_1&\mbox{ if $w=s_1$,}\\
        \alpha_2&\mbox{ if $w=s_2$,}\\
        \rho&\mbox{ otherwise.}
    \end{cases}
\end{equation*}
We have the following equalities.
\begin{equation}\label{eq: identity xw}
    \cen(x_{w})=w\cen(x)+wv_w=w\operatorname{cen}(x+v_w)\in F_w.
\end{equation}
The first equality follows from the definitions of~$x_w$ and $v_w$ (it is indeed the reason to define $v_w$ as so).
The second equality is a tautology, using Notation~\ref{nota: centroide mas vector}.
\begin{prop}\label{prop: pol otras regiones}
    For $w\in W_f$, we have
    \begin{equation}
        \operatorname{Pgn}_{x,y}^w = w(\operatorname{Par}_{x,y}^{}\setminus \, S^{x,y}_{w})\cap F_w = w(\operatorname{Pgn}_{x,y}^{}\setminus \, S^{x,y}_{w})\cap F_w.
    \end{equation}
\end{prop}
\begin{proof}
    Since $\Pgn_{x,y}^w = \Par_{x,y}^w \cap \, F_w$, to prove the first equality, it is sufficient to prove that 
    \begin{equation}\label{eq: wPgn1}
        \Par_{x,y}^w = w(\Par_{x,y} \setminus \, S^{x,y}_{w}).
    \end{equation}
    By definition, $\Par_{x,y}^w$ is completely determined by its opposite vertices $\operatorname{cen}(x_w)$ and $w\operatorname{cen}(y)$.
    It is easy to check that the set $\Par_{x,y} \setminus \, S^{x,y}_{w}$ is a parallelogram with opposite vertices $\operatorname{cen}(x) + v_w$ and $\operatorname{cen}(y)$, and edges parallel to $\alpha_1$ and $\alpha_2$.
    Then $w(\Par_{x,y} \setminus \, S^{x,y}_{w})$ is a parallelogram with opposite vertices $w(\operatorname{cen}(x) + v_w)$ and $w\operatorname{cen}(y)$.
    By Equation~\eqref{eq: identity xw}, we have
    $w(\operatorname{cen}(x) + v_w)=\cen(x_w)$, so
    $w(\Par_{x,y} \setminus \, S^{x,y}_{w})$ and $\Par_{x,y}^w$ have the same opposite vertices and edges parallel to $w\alpha_1$ and $w\alpha_2$ and thus we have proved Equation~\eqref{eq: wPgn1}.
    
    For the second equality, we have
    \begin{align*}
        w(\Pgn_{x,y}\setminus \, S^{x,y}_{w})\cap F_w&=w((\Par_{x,y}\cap \,F_\id)\setminus \, S^{x,y}_{w})\cap F_w\\
        &=w((\Par_{x,y}\setminus \, S^{x,y}_{w})\cap F_\id)\cap F_w\\
        &=w(\Par_{x,y}\setminus \, S^{x,y}_{w})\cap w(F_\id)\cap F_w\\
        &=\Par_{x,y}^w\cap \, F_w.
    \end{align*}
    The last equality follows from the fact that $F_w\subset w(F_\id)$.
\end{proof}

\subsection{The geometry of general intervals}\label{subsec: geometry of general intervals}
We conclude the section with a geometric description of every interval $[x,y]$ in~$W$.
First, observe that 
\begin{equation*}
    [x,y]=[\id,y]\cap (\geq x).
\end{equation*}
Second, Corollary~\ref{cor: characterization of the Bruhat order} provides a geometric description of~$[\id,y]$ for any $y\in W.$
Thus, the remaining issue concerns the geometric description of a set $S$ such that
\begin{equation*}
    \geq x =\{z\in W\mid \cen(z)\not\in S\}.
\end{equation*}

As a consequence of Remark~\ref{rem: remark tontos}, we have that for each $x\in W\setminus\{\id\}$ there is an element $g\in \langle\sigma,\delta\rangle$ such that 
$gx\in D\cup s_0D\cup \{\delta \mathrm{\mathbf{x}}_k\}_{k\geq 1}$.
We begin by describing $S$ for elements in~$D\cup s_0D\cup \{\delta \mathrm{\mathbf{x}}_k\}_{k\geq 1}$.
By Proposition~\ref{prop: no mayores geometrico}, we have $S=St^\circ(x)$ for $x\in D$.
It remains to describe $S$ for $x\in s_0D\cup \{\delta\mathrm{\mathbf{x}}_k\}_{k\geq 1}$.
Fortunately, the description of~$S$ can be given uniformly for all of these cases, which behave differently from the case $x\in D$.
\begin{defn} 
    Let $x\in s_0D\cup \{\delta\mathrm{\mathbf{x}}_k\}_{k\geq 1}$.
    We define the \emph{star} of~$x$, denoted $\mathcal{St}(x)$, as the union of the two equilateral triangles $\triangleleft_a$ and $\triangleright_a$ centered at the origin (the top vertex of~$A_+$), satisfying the following conditions:
    \begin{itemize}
        \item $\triangleleft_a$ is a left-facing triangle whose boundary passes through $\cen(x)$.
        \item $\triangleright_a$ is a right-facing triangle whose boundary passes through $\cen(x)$.
    \end{itemize}
\end{defn}
\begin{rem}
    The subscript $a$ in~$\triangleleft_a$, and $\triangleright_a$ emphasizes that the element $x$ belongs to $F_{s_1s_2s_1}$ (that is, $x$ is \textbf{a}nti-dominant or of the form $\delta x_k$).
    In the introduction, we described a similar construction for dominant elements, where no subscript was used.
\end{rem}
Since $g(\geq x) = \ \geq \hspace{-0.1cm}(gx)$ for any $g \in \langle \sigma, \delta \rangle$, as $g$ is an automorphism of the Bruhat order, we define $\mathcal{St}(gx) = g(\mathcal{St}(x))$.
We restate Theorem~\ref{thm: A} as follows:
\begin{thm}\label{thm: Visualizacion del intervalo}
    For any $x,y\in W$, we have:
    \begin{equation}
        [x,y] = \cen^{-1}\left(\CC_y\setminus \mathcal{St}^\circ(x)\right).
    \end{equation}
\end{thm}
\begin{proof}
    The description of~$\geq x = \{z \in W \mid \cen(z) \notin \mathcal{St}^\circ(x)\}$ for non-dominant~$x$ 
    follows from arguments similar to those in Proposition~\ref{prop: no mayores geometrico}.
    For brevity, we omit the details.
    
    The interval $[x,y]$ is determined by:
    \begin{itemize}
        \item The decomposition $[x,y] = [\mathrm{id}, y] \cap (\geq x)$.
        \item Corollary~\ref{cor: characterization of the Bruhat order}: $[\id,y]=\{x\in W\mid \operatorname{cen}(x)\in \CC_y\}$.
        \item The description $\geq x = \{z \in W \mid \cen(z) \notin \mathcal{St}^\circ(x)\}$.\qedhere
    \end{itemize}
\end{proof}

%% file: Sections/Translation.tex
\section{Piecewise translations}\label{section: translations}
In this section, we assume $x, y \in W$ and $\lambda \in \Lambda$ are dominant unless otherwise stated, and that $x < y$.
We study important maps 
\begin{equation*}
    \tau_\lambda\colon [x,y] \to [x+\lambda, y+\lambda]
\end{equation*}
given by piecewise translations in the Euclidean plane.
These maps play a major role in the main theorem of this paper.
The type of the interval (Definition~\ref{def: types of intervals}) tells us exactly which of these piecewise translations are poset isomorphisms (Proposition~\ref{prop: trasl-por-dominat}).

\subsection{Basic facts about translations and piecewise translations}\label{subsec-translations} 
Let $i\in\{1,2\}$, we define the \emph{piecewise translation by the $i$-th fundamental weight} $\tau_i\colon W\to W$ as follows: 
if $z\in W$ there is a unique $w\in W_f$ such that $z\in F_w$, so we set $\tau_i(z)\coloneqq z+w\varpi_i$ (see Notation~\ref{nota: centroide mas vector}) where $\varpi_i$ is the $i$-th fundamental weight.
For $\lambda=a\varpi_1+b\varpi_2$ dominant weight, i.e., $a,b\geq 0$, we define the \emph{piecewise translation by $\lambda$}  as $\tau_{\lambda}:=\tau^a_1\circ \tau^b_2$.
\begin{lem}\label{lem: traslacion preserva region}
    For $w\in W_f$ and $i\in\{1,2\}$, we have $\tau_i(F_w)\subset F_w$.
    Furthermore, 
    $\tau_\l(F_w)\subset F_w$.
\end{lem}
\begin{proof}
    For any $w\in W_f$, the set $F_w$ is a convex and unbounded polygon with two or three edges.
    The two unbounded edges of~$F_w$ are rays in the direction of~$w\varpi_1$ and $w\varpi_2$.
    The first part follows.
    The second part is a direct consequence of the first.
\end{proof}
\begin{lem}\label{lem: tras es inyectivo}
    For $i\in \{1,2\}$, the map $\tau_i$ is injective.
\end{lem}
\begin{proof}
    Suppose $\tau_i(u)=\tau_i(v)\in F_w$ for some $w\in W_f$ and $u,v\in W$.
    By Lemma~\ref{lem: traslacion preserva region}, we have $u,v\in F_w$, so the previous equality becomes $u+w\varpi_i=v+w\varpi_i$.
    This implies $\cen(u)=\cen(v)$, so $u=v$.
\end{proof}
\begin{lem}\label{lem: traslacion xw} 
    We have $(x+\lambda)_w=x_w+w\lambda$ for all $w\in W_f$, where $x_w$ is as in Definition~\ref{def: estrella}.
\end{lem}
\begin{proof}
    By Equation \eqref{eq: identity xw}, we have $\cen((x+\lambda)_w)=\cen(x_w)+w\lambda=\cen(x_w+w\lambda)$ so the result follows.
\end{proof}
\begin{lem}\label{lem: identities translation}
    For all $w\in W_f$, we have:
    \begin{enumerate}[label=(\roman*)]
        \item $\Par_{x,y}^w + w\lambda= \Par_{x+\lambda,y+\lambda}^w$.\label{lem: identities translation part 1}
        \item $\Pgn_{x,y}^w + w\lambda \subset \Pgn_{x+\lambda,y+\lambda}^w$.\label{lem: identities translation part 2}
    \end{enumerate}
\end{lem}
\begin{proof}
    We prove both claims:
    \begin{enumerate}[label=(\roman*)]
        \item By definition of~$\Par_{x,y}^w$ and Lemma~\ref{lem: traslacion xw}, $\Par_{x,y}^w + w\lambda$ is the parallelogram with opposite vertices $w\cen(y)+w\lambda=w\cen(y+\lambda)$ and $\cen(x_w)+w\lambda=(x+\lambda)_w$ and sides parallel to $w\alpha_1$ and $w\alpha_2$.
        By definition of~$\Par_{x+\lambda,+\lambda y}^w$, we conclude  $\Par_{x,y}^w+w\lambda=\Par_{x+\lambda,y+\lambda}^w$.
        \item $\Pgn_{x,y}^w + w\lambda=(\Par_{x,y}^w\cap\, F_w) + w\lambda \subset (\Par_{x,y}^w+w\lambda)\cap F_w = \Pgn_{x+\lambda,y+\lambda}^w$.
        This is because $F_w+w\lambda\subset F_w$ (Lemma~\ref{lem: traslacion preserva region}.)
    \end{enumerate}
\end{proof}
\begin{lem}\label{lem: tras intervalo es subconjunto}
    For $i\in \{1,2\}$, we have $\tau_i([x,y])\subset [\tau_i(x),\tau_i(y)]$.
\end{lem}
\begin{proof}
    Let $x\leq z\leq y$ and $w\in W_f$ be such that $z\in F_w$.
    We need to prove that $x+\varpi_i\leq z+w\varpi_i\leq y+\varpi_i$.
    By Equation \eqref{eq: interval as union of polygons}, it is enough to show that $\operatorname{cen}(z+w\varpi_i)\in \Pgn^w_{x+\varpi_i,y+\varpi_i}$.
    Note that $\cen(z)\in \Pgn_{x,y}^w$.
    By Lemma~\ref{lem: identities translation}\ref{lem: identities translation part 2},
    \begin{equation*}
    \cen(z+w\varpi_i)=\cen(z)+w\varpi_i\in \Pgn_{x,y}^w+w\varpi_i\subset \Pgn^w_{x+\varpi_i,y+\varpi_i}.\qedhere
    \end{equation*}
\end{proof}
\begin{prop}\label{prop: traslaciones que preservan pentagonos}
    Let $\mu\in \Lambda$ be any weight.
    If $\Pgn_{x,y}$ is a pentagon with two adjacent vertices on~$\mathbbm{R}_{\geq 0}\varpi_{i}+\alpha_{i}$ for some $i\in\{1,2\}$, then $\Pgn_{x,y}+\mu=\Pgn_{x+\mu,y+\mu}$ only if $\mu\in\mathbbm{R}\varpi_{i}$.
\end{prop}
\begin{proof}
    Let $x',y'\in D$ be dominant elements, such that $\Pgn_{x',y'}$ is a pentagon with vertices lying in~$\mathbbm{R}\varpi_2+\alpha_2$.
    It is easy to prove, by comparing the interior angles of both polygons, that if $\Pgn_{x,y}$ has two vertices lying on  $\mathbbm{R}\varpi_{1}+\alpha_{1}$, then $\Pgn_{x',y'}$ cannot be a translation of~$\Pgn_{x,y}$.
    Since this remains true for any choice of dominant elements $x',y'$, the result follows.
\end{proof}
\begin{prop}\label{prop: traslaciones que preservan Pgn}
    The following statements hold: 
    \begin{enumerate}
        \item\label{item-prop: trasl presev par} If $\Pgn_{x,y}$ is a parallelogram, then $\Pgn_{x,y}+\lambda= \Pgn_{x+\lambda,y+\lambda}$.
        \item\label{item-prop: trasl presev pent} If $\Pgn_{x,y}$ is a pentagon with two adjacent vertices on~$\mathbbm{R}_{\geq 0}\varpi_{i}+\alpha_{i}$ for some $i\in\{1,2\}$, then $\Pgn_{x,y}+\lambda=\Pgn_{x+\lambda,y+\lambda}$ if and only if $\lambda\in\mathbbm{R}_{\geq 0}\varpi_{i}$.
        \item\label{item-prop: trasl presev hex} If $[x,y]$ is a hexagon interval, then $\Pgn_{x,y}+\lambda= \Pgn_{x+\lambda,y+\lambda}$ only if $\lambda=0$.
    \end{enumerate}
\end{prop}
\begin{proof} 
    \begin{enumerate}
        \item It is not hard to prove (using that $\cen(x),\cen(y)\in \operatorname{V}(\Pgn_{x,y})$) that $\Pgn_{x,y}$ is a parallelogram if and only if $\Par_{x,y}=\Pgn_{x,y}$.
        By Lemma~\ref{lem: traslacion preserva region}, we have $\Pgn_{x,y}+\lambda\subset F_\id$.
        By Lemma~\ref{lem: identities translation}\ref{lem: identities translation part 1}, we have $\Pgn_{x,y}+\lambda=\Par_{x+\lambda,y+\lambda}$, so
        \begin{equation*}
            \Pgn_{x+\lambda,y+\lambda}=\Par_{x+\lambda,y+\lambda}\cap\, F_\id=(\Pgn_{x,y}+\lambda)\cap F_\id = \Pgn_{x,y}+\lambda.
        \end{equation*}
        \item  The implication $\implies$ is given by Proposition~\ref{prop: traslaciones que preservan pentagonos}.
            
        Let us prove the implication $\impliedby$.
        Without loss of generality, we can assume that $i=1$.
        Consider the following general result of planar geometry: 
    
        Let $H$ be a closed half-plane of the plane $E$.
        For any subset $X\subset E$, we use the notation $X_H:=X\cap H$.
        If $v$ is a vector such that $\partial H+v=\partial H$, then for any subset $S\subset E$ we have $ S_H+v=(S+v)_H$.
        
        We apply this result to the special case where $H=\{u\in E\mid (\alpha_2, u)\geq -1\}$, $v=\lambda\in\mathbbm{R}_{\geq 0}\varpi_1$, and $S=\Par_{x,y}$ (note that $\partial H = \mathbbm{R}\varpi_1+\alpha_1$ and $\lambda+\partial H = \partial H$.)
        Thus, we have $(\Par_{x,y})_H+\lambda = (\Par_{x,y}+\lambda)_H$.
        We have $(\Par_{x,y})_H\subset F_\id$, so $(\Par_{x,y})_H=\Pgn_{x,y}$.
        By combining these identities, we get $\Pgn_{x,y}+\lambda = (\Par_{x,y}+\lambda)_H$.
        By Lemma~\ref{lem: identities translation}\ref{lem: identities translation part 1} and the fact that $F_\id\subset H$, we have
        \begin{align*}
            \Pgn_{x+\lambda,y+\lambda} 
            &=(\Par_{x+\lambda,y+\lambda})_H\cap F_\id\\
            &=(\Par_{x,y}+\lambda)_H\cap F_\id\\
            &=(\Pgn_{x,y}+\lambda)\cap F_\id\\
            &= \Pgn_{x,y}+\lambda.
        \end{align*}
        The last equality follows from $\Pgn_{x,y}\subset F_\id$ and Lemma~\ref{lem: traslacion preserva region}.
        \item Recall that every element of~$\mathrm{V}(\Pgn_{x,y})$ different from $\cen(x)$ and $\cen(y)$ belongs either to $\mathbbm{R}\varpi_1+\alpha_1$ or $\mathbbm{R}\varpi_{2}+\alpha_{2}$.
        The analogous statement holds for $\mathrm{V}(\Pgn_{x+\mu,y+\mu})$.
        Let $i\in \{1,2\}$, and let $a_i\in \operatorname{V}(\Pgn_{x,y})$ adjacent to $\cen(x)$ such that 
        $a_i\in \mathbbm{R}\varpi_i+\alpha_i$.
        Since $a_i+\lambda\in \mathrm{V}(\Pgn_{x+\lambda,y+\lambda})\setminus\{\cen(x+\lambda),\cen(y+\lambda)\}$, and $\l$ is a dominant weight, we have
        \begin{equation}
            a_i+\l\in \mathbbm{R}\varpi_i+\alpha_i.
        \end{equation}
        The condition $a_i+\l\in \mathbbm{R}\varpi_i+\a_i$ holds only if $\l\in \mathbbm{R}_{\geq 0}\varpi_i$, so $\lambda=0$.\qedhere
    \end{enumerate}
\end{proof}

\subsection{Piecewise translations preserving cardinality}\label{subsec: translations preserving cardinality}

To determine whether the map $\tau_\lambda$ defines a poset isomorphism, we begin by comparing the cardinalities of the intervals $[x, y]$ and $[x+\lambda, y+\lambda]$.
In this section, we prove Proposition~\ref{prop: cardinality under translation}, which establishes conditions under which these two intervals have the same cardinality.

\begin{nota}
    If $A \subset E$, we denote $cA := \operatorname{cen}(W) \cap A$.
\end{nota}

Let $A \subset E$.
For any weight $\mu \in \Lambda$, recall from Section~\ref{subsec: alcoves} that $\operatorname{cen}(W) = \operatorname{cen}(W) + \mu$.
This implies that $\operatorname{cen}(W) \cap (A + \mu) = (\operatorname{cen}(W) \cap A) + \mu$, or in compressed notation,
\begin{equation*}
    c(A + \mu) = cA + \mu.
\end{equation*}
Similarly, for $z \in W$, we have $z(\operatorname{cen}(W)) = \operatorname{cen}(W)$.
Hence, for any subset $A \subset E$, we obtain $\operatorname{cen}(W) \cap zA = z(\operatorname{cen}(W) \cap A)$, or equivalently, 
\begin{equation*}
    czA = z(cA).
\end{equation*}
We typically prefer the notation $czA$ over $zcA$ to emphasize that the result lies in $\operatorname{cen}(W)$.

Let $z \in W$.
As discussed in Section~\ref{subsec: alcoves}, the map $z\colon E \to E$ is a bijection.
In particular, for any subset $A \subset E$, the restricted map $z\vert_{cA} \colon cA \to czA$ is a bijection, so $|cA| = |czA|$.
Similarly, for any $\mu \in \Lambda$ and any $A \subset E$, we have $|cA| = |cA + \mu|$.

Summarizing, for any $z \in W$, $\mu \in \Lambda$, and $A \subset E$, we have:
\begin{align}
    c(A + \mu) &= cA + \mu, \label{eq:centro de translation} \\
    |cA| &= |cA + \mu|.\label{eq:cardinality translation}
\end{align}
Finally, by Equation~\eqref{eq: interval as union of polygons}, we have:
\begin{equation}\label{eq: cardinal intervalo con centros}
    |[x, y]| = \sum_{w \in W_f} |\cPgn_{x, y}^w|.
\end{equation}

\begin{lem}\label{lem: cantidad de centroides de Pgn w}
    Let $w\in W_f$, then $|\cPgn^w_{x,y}|\leq |\cPgn^w_{x+\lambda,y+\lambda}|$ for all $w\in W_f$.
    So $|[x,y]|\leq |[x+\lambda,y+\lambda]|$.
\end{lem}
\begin{proof}
    By Equation \eqref{eq:centro de translation} and Lemma~\ref{lem: identities translation}\ref{lem: identities translation part 2} we have
    \begin{equation*}
        c\Pgn^w_{x,y}+w\lambda\subset c\Pgn^w_{x+\lambda, y+\lambda}.
    \end{equation*}
    By Equation \eqref{eq:cardinality translation} we have $\vert c\Pgn^w_{x,y}+w\lambda \vert= \vert c\Pgn^w_{x,y}\vert$, which allows us to conclude $|\cPgn^w_{x,y}|\leq |\cPgn^w_{x+\lambda,y+\lambda}|$.
    By Equation \eqref{eq: cardinal intervalo con centros}, this implies $|[x,y]|\leq |[x+\lambda,y+\lambda]|$.
\end{proof}

\begin{prop}\label{prop: igualdad en F_id implica igualdad Pgn^w+l}
    If $\Pgn_{x,y}+\lambda=\Pgn_{x+\lambda,y+\lambda}$, then $\Pgn_{x,y}^w+w\lambda=\Pgn_{x+\lambda,y+\lambda}^w$ for all $w\in W_f.$
\end{prop}
\begin{proof}
    \textbf{Case A.} If $\Pgn_{x,y}$ is a hexagon, by Proposition~\ref{prop: traslaciones que preservan Pgn} we have $\lambda=w\lambda=0$ and the proposition follows.
        
    To prove the cases where $\Pgn_{x,y}$ is a pentagon or a parallelogram, we need much more work.
    Let $w\in W_f$.
    By Proposition~\ref{prop: pol otras regiones}, we have 
    \begin{align}
        \Pgn_{x,y}^w+w\lambda &=\big(w(\operatorname{Pgn}_{x,y}^{}\setminus \, S^{x,y}_{w})\cap F_w\big)+w\lambda\\
        & =\big(w\big(\operatorname{Pgn}_{x,y}^{}\setminus \, S^{x,y}_{w}\big) +w\lambda\big)\cap \big(F_w +w\lambda\big)\\
        & \subset w  \big( (\operatorname{Pgn}_{x,y}^{}\setminus \, S^{x,y}_{w})+\lambda\big) \cap F_w \label{eq: inclusion lema igul implica idnt}\\
        &=w\big( \operatorname{Pgn}_{x+\lambda,y+\lambda}^{}\setminus \, S^{x+\lambda,y+\lambda}_{w}\big)\cap F_w\\
        & = \Pgn_{x+\lambda,y+\lambda}^w.
    \end{align}
    The inclusion in the third line of \eqref{eq: inclusion lema igul implica idnt} follows from Lemma~\ref{lem: traslacion preserva region}.
    The fourth equality is a consequence of the hypothesis and $S^{x+\lambda,y+\lambda}_{w}=S^{x,y}_{w}+\lambda$ (this last equality easily follows from the definition of~$S^{x,y}_{w}$ and Lemma~\ref{lem: identities translation}\ref{lem: identities translation part 1}).
    To finish the proof of the proposition, we need to prove the inverse inclusion  $\Pgn_{x+\lambda,y+\lambda}^w\subset \Pgn_{x,y}^w+w\lambda$, in other words, we need to prove that
    \begin{equation*}
        A\coloneqq \big( w   (\operatorname{Pgn}_{x,y}^{}\setminus \, S^{x,y}_{w})+w\lambda\big) \cap F_w\subset \big(w(\operatorname{Pgn}_{x,y}^{}\setminus \, S^{x,y}_{w})\cap F_w\big) +w\lambda\eqqcolon B.
    \end{equation*}
    Let $p\in A$, then $p=wp_1+w\lambda\in F_w$ for some $p_1\in \Pgn_{x,y}\setminus\, S_{w}^{x,y}$.
    Then $p\in B$ if and only if $wp_1\in F_w$.
    So we need to prove that $wp_1\in F_w$.
    \begin{claim}\label{claim: Pentagon minus strip in F_w}
            Let $\Pgn_{x,y}$ be a pentagon with two adjacent vertices on~$\mathbbm{R}\varpi_{1}+\alpha_{1}$.
            \begin{enumerate}[label=(\roman*)]
                \item $s_1(\Pgn_{x,y}\setminus\, S_{s_1}^{x,y})\subset  F_{s_1}$.\label{claim: Pentagon minus strip in F_w part 1}
                \item Let $w\in W_f\setminus \{\id,s_1\}$ and $a\in \mathbbm{Z}_{\geq 0}$.
                If $v\in\Pgn_{x,y}\setminus\, S_{w}^{x,y}$ and $w(v+a\varpi_1)\in F_{w}$, then $wv\in F_{w}$.\label{claim: Pentagon minus strip in F_w part 2}  
            \end{enumerate}                   
        \end{claim}
        \begin{proof}
            \begin{enumerate}[label=(\roman*)]
                \item From Figure~\ref{subfig: star and Fw zones intro} we get 
                \begin{equation}\label{eq: Fs1 como imagen de Fid}
                    F_{s_1}=s_1\{v\in F_{\id}\mid (\alpha_1, v)\geq 1\}.
                \end{equation}
                Let $b$ be the only vertex of~$\Pgn_{x,y}$ adjacent to both $\cen(x)$ and $\cen(y)$.
                In particular, $(\alpha_1, b+\alpha_1)\geq 1$.
                From the definitions of~$S^{x,y}_{s_1}$ and $\Pgn_{x,y}$, every element of~$\Pgn_{x,y}\setminus\, S^{x,y}_{s_1}$ can be written as $b+\alpha_1+c\alpha_1-d\alpha_2$, for some $c,d\geq 0$.
                Hence $\Pgn_{x,y}\setminus\, S^{x,y}_{s_1}\subset \{v\in F_\id\mid (\alpha_1,v)\geq (\alpha_1,b+\alpha_1)\}$.
                The claim follows from Equation \eqref{eq: Fs1 como imagen de Fid}.
                \item We only prove the case $w=s_2$; the other cases are similar, so we omit them.
                By hypothesis $v\in\Pgn_{x,y}\setminus\, S_{s_2}^{x,y}$.
                Let us prove the claim by contrapositive, so let us suppose that $s_2v\not\in F_{s_2}$.
                Thus $s_2v\in s_2(F_{\id})\setminus F_{s_2}$.
                From Figure~\ref{subfig: star and Fw zones intro}, we note that $s_2(F_{\id})\setminus F_{s_2}\subset\{u\in E\mid -1< (\alpha_2,u)\leq 1\}$ and $F_{s_2}\subset \{u\in E\mid (\alpha_2,u)\leq -1\}$.
                In particular, $-1 < (\alpha_2, s_2v+a\varpi_1)$, so $s_2(v+a\varpi_1)\notin F_{s_2}$ (recall that $s_2\varpi_1=\varpi_1)$.
                This proves the claim.\qedhere
            \end{enumerate}
        \end{proof}
    
    \noindent\textbf{Case B.} If $\Pgn_{x,y}$ is a pentagon, without loss of generality, assume that $\Pgn_{x,y}$ has two adjacent vertices in~$\mathbbm{R}\varpi_1+\alpha_1$.
    By Proposition~\ref{prop: traslaciones que preservan Pgn}, we have $\lambda=a\varpi_1$ for some $a\in \mathbbm{Z}_{\geq 0}$.
    We have two cases:
    \begin{itemize}
        \item $w=s_1$.
        By Claim~\ref{claim: Pentagon minus strip in F_w}\ref{claim: Pentagon minus strip in F_w part 1} we have $s_1p_1\in F_{s_1}$ and the proposition follows.
        \item $w\in W_f\setminus\{\id,s_1\}$.
        Note that $w(p_1+a\varpi_1)=p\in F_{w}$, so by Claim~\ref{claim: Pentagon minus strip in F_w}\ref{claim: Pentagon minus strip in F_w part 2} we have $wp_1\in F_{w}$ and the proposition follows.
    \end{itemize}
    \begin{claim}\label{claim: Parallelogram minus strip in F_w}
        Let $\Pgn_{x,y}$ be a parallelogram.
        We have $w(\Pgn_{x,y}\setminus\, S_{w}^{x,y})\subset F_w$.
    \end{claim}
    \begin{proof}
        The proof is similar to the first part of Claim~\ref{claim: Pentagon minus strip in F_w}, so we will omit it.
    \end{proof}
    \textbf{Case C.} If $\Pgn_{x,y}$ is a parallelogram, Claim~\ref{claim: Parallelogram minus strip in F_w} gives $wp_1\in F_w$ and the proposition follows.
    This finishes the proof of the proposition.
\end{proof}
\begin{lem}\label{lem: isometria poligono en identidad implica en otras regiones}
    If $\Pgn_{x,y}+\lambda=\Pgn_{x+\lambda,y+\lambda}$, then $|[x,y]|=|[x+\lambda,y+\lambda]|$.
\end{lem}
\begin{proof}
    By Equation \eqref{eq: cardinal intervalo con centros}, Proposition~\ref{prop: igualdad en F_id implica igualdad Pgn^w+l}, Equation \eqref{eq:centro de translation}, and again Equation \eqref{eq: cardinal intervalo con centros} we obtain the following sequence of equalities.
    \begingroup
    \allowdisplaybreaks
    \begin{align}
        |[x+\lambda,y+\lambda]|&=\sum_{w\in W_f} |c\Pgn_{x+\lambda,y+\lambda}^w|\\ &=\sum_{w\in W_f}|c(\Pgn_{x,y}^w+w\lambda)|\\
        &=\sum_{w\in W_f}|c\Pgn_{x,y}^w|\\
        &=|[x,y]|.\qedhere
    \end{align}
    \endgroup
\end{proof}
\begin{defn}\label{def: label of vertices}
    Let $x<y$ and $1\leq i \leq 2$ and let $v^{x,y}_i(y)\in \operatorname{V}(\Pgn_{x,y})$ be the vertex adjacent to $\cen(y)$ such that $\cen(y)-v^{x,y}_i(y)\in \mathbbm{R}_{>0}\alpha_i$.
    Similarly, let $v^{x,y}_i(x)\in \operatorname{V}(\Pgn_{x,y})$ be the vertex adjacent to $\cen(x)$ such that $v^{x,y}_i(x)-\cen(x)\in \mathbbm{R}_{>0}\alpha_i$.
    If the interval $[x,y]$ is clear from the context, the superscript $(-)^{x,y}$ will be omitted.
\end{defn} 
\begin{lem}\label{lem: vertice de 60 grados}
    Let $\{i,j\}=\{1,2\}$.
    We have that $v^{x,y}_{i}(x)=v^{x,y}_j(y)$ implies the equality 
    \begin{equation*}
        v^{x,y}_{i}(x)=\cen(x)+(\varpi_i, \cen(y)-\cen(x))\alpha_i.
    \end{equation*}
    Moreover, if we denote $a_i\coloneqq(\varpi_i, \cen(y)-\cen(x))$ for $1\leq i\leq 2$, we have
    \begin{equation*}
        \operatorname{V}(\Par_{x,y})=\{\cen(x),\cen(y), \cen(x)+a_1\alpha_1, \cen(x)+a_2\alpha_2\}.
    \end{equation*}
\end{lem}
\begin{proof}
    Note that $v=v^{x,y}_{i}(x)=v^{x,y}_j(y)$ is in the intersection between the affine lines $\cen(x)+\mathbbm{R}\alpha_i$ and $\cen(y)+\mathbbm{R}\alpha_j$.
    Thus, there is $a_i\in \mathbbm{R}$ such that $v=\cen(x)+a_i\alpha_i \in \cen(y)+\mathbbm{R}\alpha_j$.
    This implies that $(\varpi_i,\cen(x)+a_i\alpha_i)=(\varpi_i, \cen(y))$ so we obtain $a_i=(\varpi_i, \cen(y)-\cen(x))$.
    Similarly, the second statement holds.
\end{proof}
\begin{lem}\label{lem: traslacion de vertice de 60 grados}
    Let $\{i,j\}=\{1,2\}$.
    If $v=v^{x,y}_{i}(x)=v^{x,y}_j(y)$, then 
    \begin{equation*}
        v+\lambda =v^{x+\lambda,y+\lambda}_{i}(x+\lambda)=v^{x+\lambda,y+\lambda}_j(y+\lambda).
    \end{equation*}
\end{lem}
\begin{proof}
    The equality $v^{x,y}_{i}(x)=v^{x,y}_j(y)$ is equivalent to  the inclusion
    \begin{equation*}
        \emptyset \neq (x+\mathbbm{R}\alpha_i) \cap (y+\mathbbm{R}\alpha_j)\subset F_{\id}.
    \end{equation*}
    This and Lemma~\ref{lem: traslacion preserva region} imply that
    \begin{equation*}
        \lambda+\big((x+\mathbbm{R}\alpha_i) \cap (y+\mathbbm{R}\alpha_j)\big)=(x+\lambda+\mathbbm{R}\alpha_i) \cap (y+\lambda+\mathbbm{R}\alpha_j)\in F_{\id}.
    \end{equation*}
    This in turn implies that $v^{x+\lambda,y+\lambda}_{i}(x+\lambda)=v^{x+\lambda,y+\lambda}_j(y+\lambda)=v+\l$.
\end{proof}
\begin{lem}\label{lem: vertex in the wall and vertex outside the wall}
    Let $\{i,j\}=\{1,2\}$.
    \begin{enumerate}[label=(\roman*)]
        \item \label{lem: vertex in the wall and vertex outside the wall part 1} If $v^{x,y}_i(x)\notin \partial F_{\id}$, then  
        \begin{equation*}
            v^{x,y}_{i}(x)=v^{x,y}_j(y)=\cen(x)+(\varpi_i, \cen(y)-\cen(x))\alpha_i.
        \end{equation*}
        \item \label{lem: vertex in the wall and vertex outside the wall part 2} If $v^{x,y}_i(x)\in \partial F_{\id}$, then  
        \begin{equation*}
            v^{x,y}_i(x)=\cen(x)+(1+(\alpha_j, \cen(x)))\alpha_i.
        \end{equation*}
    \end{enumerate}
\end{lem}
\begin{proof}
    \begin{enumerate}[label=(\roman*)]
        \item If $v^{x,y}_{i}(x)\notin \partial F_{\id}$ then $v^{x,y}_{i}(x)=v^{x,y}_j(y)$.
        The second equality follows from Lemma~\ref{lem: vertice de 60 grados}.
        \item By definition of~$v^{x,y}_{i}(x)$, there is $a_i>0$ such that $\cen(x)+a_i\alpha_i= v^{x,y}_i(x)$.
        Since $v^{x,y}_{i}(x)\in \partial F_{\id}$, we have either $v^{x,y}_{i}(x)$ is in~$H_{\alpha_j,-1}$, $H_{\alpha_i,-1}$, or $H_{\rho,-1}$.
        The second case is impossible since $a_i$, $(\alpha_i, \cen(x))$, and $(\alpha_i, \alpha_i)$ are positive.
        Similarly, the third case is impossible since $a_i$, $(\rho, \cen(x))$, and $(\rho, \alpha_i)$ are positive.
        Therefore $(\alpha_j, \cen(x)+a_i\alpha_i)=-1$, so we obtain $a_i=1+(\alpha_j, \cen(x))$.\qedhere
    \end{enumerate}
\end{proof}
\begin{lem}\label{lem: no isometria poligono en identidad implica mas cardinal}
    Suppose $\Pgn_{x,y}+\lambda\neq\Pgn_{x+\lambda,y+\lambda}$.
    Then $|\cPgn_{x,y}|<|\cPgn_{x+\lambda,y+\lambda}|$ and $|[x,y]| < |[x+\lambda,y+\lambda]|$.
\end{lem}
\begin{proof}
    We start the proof by replacing the condition $\Pgn_{x,y}+\lambda=\Pgn_{x+\lambda,y+\lambda}$ with a simpler statement.
    Let $n_i\in \mathbbm{Z}_{\geq 0}$ be maximal such that $\cen(x)+n_i\alpha_i\in \Pgn_{x,y}$ and $m_i\in \mathbbm{N}$ maximal such that $\cen(x)+\lambda+m_i\alpha_i\in \Pgn_{x+\lambda,y+\lambda}$ for $i\in\{1,2\}$.
    Similarly, let $a_i\in \mathbbm{R}_{\geq 0}$ be maximal such that $\cen(x)+a_i\alpha_i\in \Pgn_{x,y}$ and $b_i\in \mathbbm{R}_{\geq 0}$ be maximal such that $\cen(x)+\lambda+b_i\alpha_i\in \Pgn_{x+\lambda,y+\lambda}$ for $i\in\{1,2\}$.
    Clearly, $n_i=\lfloor a_i \rfloor$ and $m_i=\lfloor b_i \rfloor$.
    By Lemma~\ref{lem: identities translation}\ref{lem: identities translation part 2}, we have $\Pgn_{x,y}+\lambda\subset\Pgn_{x+\lambda,y+\lambda}$ which implies that 
    \begin{equation}\label{eq:bi>ai}
    m_i\geq n_i\ \mathrm{ and } \  b_i\geq a_i.
    \end{equation}
    In particular, $v^{x,y}_i(x)=\cen(x)+a_i\alpha_i$ and $v^{x+\lambda,y+\lambda}_i(x+\lambda)=\cen(x+\lambda)+b_i\alpha_i$.
    \begin{claim}\label{claim: aibi}
        The condition $a_i=b_i$ for $i\in\{1,2\}$ implies $\Pgn_{x,y}+\lambda=\Pgn_{x+\lambda,y+\lambda}$.
    \end{claim}
    \begin{proof}
        Let us assume that $a_i=b_i$ for $i\in\{1,2\}$.
        There exist $c_i\in \mathbbm{R}_{\geq 0}$ such that 
        \begin{equation}
            \operatorname{V}(\Pgn_{x,y})=\{x,y, x+a_1\alpha_1,x+a_2\alpha_2,x+a_1\alpha_1+c_1\varpi_1,x+a_2\alpha_2+c_2\varpi_2 \},
        \end{equation}
        where if one or two of the $c_i$'s are zero, we are in the pentagon or parallelogram case.
        As $x+a_i\alpha_i+c_i\varpi_i\in y+\mathbbm{R}\alpha_j$, by pairing this inclusion with $\varpi_i$ we obtain 
        \begin{equation}\label{eq: ci}
            c_i=\frac{3}{2}((y-x,\varpi_i)-a_i).
        \end{equation}
        The set $\operatorname{V}(\Pgn_{x+\lambda,y+\lambda})$ is given by
        \begin{equation}
            \{ x+\lambda,y+\lambda,x+\lambda+a_1\alpha_1,x+\lambda+a_2\alpha_2,x+\lambda+a_1\alpha_1+c'_1\varpi_1,x+\lambda+a_2\alpha_2+c'_2\varpi_2\}.
        \end{equation}
        As the right-hand side of Equation \eqref{eq: ci} only depends on~$y-x=(y+\lambda)-(x+\lambda)$ and $a_i=b_i$, we obtain that $c_i=c_i'$ and we conclude.
    \end{proof}
    \begin{claim}\label{claim: ai dist b_i implica posicion de vertice}
        Let $1\leq i\leq 2$ and suppose that $n_i=m_i$.
        If $a_i\neq b_i$, then $v_i^{x,y}(x)\in \partial F_\id$ and $ v_i^{x+\lambda,y+\lambda}(x+\lambda)\not\in \partial F_\id.$
    \end{claim}
    \begin{proof}
        Let $j$ be such that $\{i,j\}=\{1,2\}$.
        Recall that $v_i^{x,y}(x)\in \mathrm{V}(\Pgn_{x,y})\subset F_{\id}$.
        \begin{itemize}
            \item Suppose that $v^{x,y}_i(x)\notin\partial F_{\id}$.
            By Lemma~\ref{lem: vertex in the wall and vertex outside the wall}\ref{lem: vertex in the wall and vertex outside the wall part 1} we have that $v^{x,y}_{i}(x)=v^{x,y}_j(y)$.
            By Lemma~\ref{lem: traslacion de vertice de 60 grados} we have that $v^{x,y}_i(x)+\lambda =v^{x+\lambda,y+\lambda}_{i}(x+\lambda)$, so $a_i=b_i$, a contradiction.
            \item Suppose that $v^{x+\lambda,y+\lambda}_i(x+\lambda)\in\partial F_{\id}$.
            As $v^{x,y}_i(x)\in\partial F_{\id}$, by Lemma~\ref{lem: vertex in the wall and vertex outside the wall}\ref{lem: vertex in the wall and vertex outside the wall part 2}, we have that $b_i-a_i=(\alpha_j,\lambda)\in\mathbbm{Z}$ is the difference between two numbers having the same floor, so $a_i=b_i$, a contradiction.\qedhere
        \end{itemize}
    \end{proof}
    \begin{claim}\label{claim: ai dist bi implica elemento}
        Let $1\leq i\leq 2$ and suppose that $n_i=m_i$.
        If $a_i\neq b_i$ then $v_i^{x+\lambda,y+\lambda}(x+\lambda)=\cen(z)$ for some $z\in W$.
        Furthermore $\cen(z)-\l \in \mathrm{V}(\Par_{x,y})$ and $\cen(z)-\l\not\in \Pgn_{x,y}$.
    \end{claim}
    \begin{proof} 
        Let $j$ be such that $\{i,j\}=\{1,2\}$.
        By Claim~\ref{claim: ai dist b_i implica posicion de vertice}, we have that $v_i^{x,y}(x)\in \partial F_\id$ and  $ v^{x+\lambda,y+\lambda}_i(x+\lambda)\notin\partial F_{\id}$.
        From Lemma~\ref{lem: vertex in the wall and vertex outside the wall}\ref{lem: vertex in the wall and vertex outside the wall part 1}, we obtain that
        \begin{equation*}
            v\coloneqq v^{x+\lambda,y+\lambda}_i(x+\lambda)=v^{x+\lambda,y+\lambda}_j(y+\lambda).
        \end{equation*}
        This implies that $v$ is a vertex of both $\Par_{x+\lambda,y+\lambda}$ and $\Pgn_{x+\lambda,y+\lambda}$, because it is the intersection of the affine lines $x+\lambda+\mathbbm{R}\alpha_i$ and $y+\lambda+\mathbbm{R}\alpha_j$.
        Hence $v-\l\in \operatorname{V}(\Par_{x,y})$ because it is the intersection of the affine lines $x+\mathbbm{R}\alpha_i$ and $y+\mathbbm{R}\alpha_j$.

        Now, recall that the floor of~$a_i$ is $n_i$,  and $v_i^{x,y}(x)=\cen(x)+a_i\a_i.$
        By Remark~\ref{rem: geometry-of-vertex-of-Pgn}, we have that $a_i$ is either $n_i$, $n_i+\frac{1}{3}$ or $n_i+\frac{2}{3}$.
        Similarly $b_i$ is either $m_i$, $m_i+\frac{1}{3}$ or $m_i+\frac{2}{3}$.
        
        If $a_i=n_i$, since $\cen(W)=\cen(W)+n_i\alpha_i$ we have that $\cen(x)+a_i\alpha_i\in \cen(W)\cap \partial F_\id=\emptyset$, a contradiction.
        So we conclude that $a_i\neq n_i$.
        From hypothesis and Equation~\eqref{eq:bi>ai}, it follows that $a_i<b_i$.
        
        By hypothesis $n_i=m_i$, thus, $a_i<b_i$ implies that $a_i=n_i+\frac{1}{3}$ and $b_i=n_i+\frac{2}{3}$.
        This implies that clearly, 
        \begin{equation}\label{eq: 3c3}
            v-\lambda=\cen(x)+\left(n_i+\frac{2}{3}\right)\alpha_i\in \operatorname{V}(\Par_{x,y}).
        \end{equation}
       
        The point $v^{x,y}_i(x)=\cen(x)+a_i\alpha_i\in \partial F_{\id}$ is  a vertex of an alcove because $v^{x,y}_i(x)\in \partial F_{\id}$.
        Since
        $a_i\in \mathbbm{N}+\frac{1}{3}$, we have that $x$ is down-oriented.
        By Equation \eqref{eq: 3c3} we have that $v-\lambda=\cen(z_v)$ for some up-oriented element $z_v\in W$.
        The claim follows by taking \( z = z_v + \lambda \), since
        \begin{equation*}
            \cen(z) = \cen(z_v) + \lambda = v \in \Pgn_{x+\lambda, y+\lambda}.
        \end{equation*}
        Moreover,  by Equation \eqref{eq: 3c3}
        \begin{equation*}
            \cen(z) - \lambda  = v_i^{x,y}(x) + \frac{1}{3} \alpha_i \not\in F_\id,
        \end{equation*}
        which implies that \( \cen(z) - \lambda \not\in \Pgn_{x,y} \).
    \end{proof}
    \begin{claim}\label{claim: same ni and mi implies same bi}
        If $n_i=m_i$ for all $i\in\{1,2\}$, then either $a_i=b_i$ for all $i\in\{1,2\}$ or $|\cPgn_{x+\lambda,y+\lambda}|-|\cPgn_{x,y}|\geq 1$.
    \end{claim}
    \begin{proof}
        Without loss of generality, we can suppose $a_1\neq b_1$.
        By Claim~\ref{claim: ai dist bi implica elemento}, we have $v_1^{x+\lambda,y+\lambda}(x+\lambda)=\cen(z)\in\Pgn_{x+\lambda,y+\lambda}$ for some $z\in W$.
        Furthermore we have that $\cen(z)-\l\notin \Pgn_{x,y}$.
        Lemma~\ref{lem: identities translation}\ref{lem: identities translation part 2} states that $\Pgn_{x,y} + \lambda \subset \Pgn_{x+\lambda,y+\lambda}$, and this inclusion is strict because $\cen(z)\in \Pgn_{x+\lambda,y+\lambda}\setminus (\Pgn_{x,y} + \lambda)$.
    \end{proof} 
    Now we can finish the proof of Lemma~\ref{lem: no isometria poligono en identidad implica mas cardinal}.
    By the previous claim, if $n_i=m_i$ for all $i\in\{1,2\}$, we have two possibilities.
    The first one is that  $a_i=b_i$ for all $i\in\{1,2\}$.
    In this case by Claim~\ref{claim: aibi} we have that $\Pgn_{x,y}+\lambda=\Pgn_{x+\lambda,y+\lambda}$ and this contradicts the hypothesis.
    The second possibility, by Claim~\ref{claim: same ni and mi implies same bi}, we have $|\cPgn_{x,y}|<|\cPgn_{x+\lambda,y+\lambda}|$.
    By Lemma~\ref{lem: cantidad de centroides de Pgn w} we know that $|\cPgn^w_{x,y}|\leq |\cPgn^w_{x+\lambda,y+\lambda}|$ for all $w\in W_f$.
    Equation~\eqref{eq: cardinal intervalo con centros} implies that $|[x,y]| < |[x+\lambda,y+\lambda]|$, which proves the lemma in this case.
    
    It remains the case when there is $1\leq i_0\leq 2$ such that $n_{i_0}<m_{i_0}$.
    Then $\cen(x)+m_{i_0}\alpha_{i_0}\notin \Pgn_{x,y}$ and the injective map $\cPgn_{x,y}\to \cPgn_{x+\lambda, y+\lambda}$ given by $z\mapsto z+\lambda$ is not surjective because $\cen(x)+\lambda+m_{i_0}\alpha_{i_0}$ has no preimage.
    This proves that $|\cPgn_{x,y}|<|\cPgn_{x+\lambda, y+\lambda}|$.
    By the same reasoning as before, we have $|[x,y]|<|[x+\lambda,y+\lambda]|$.
\end{proof} 
Lemmas~\ref{lem: isometria poligono en identidad implica en otras regiones} and~\ref{lem: no isometria poligono en identidad implica mas cardinal} can be summarized as follows.
\begin{lem}\label{lem: same cardinal iff same poligon}
    $|[x,y]|=|[x+\lambda,y+\lambda]|$ if and only if $\Pgn_{x,y}+\lambda=\Pgn_{x+\lambda,y+\lambda}$.
\end{lem}
The following Proposition is a direct consequence of Proposition~\ref{prop: traslaciones que preservan Pgn}, Lemma~\ref{lem: cantidad de centroides de Pgn w}, and Lemma~\ref{lem: same cardinal iff same poligon}.
\begin{prop}
    \label{prop: cardinality under translation}
    For $x,y$ dominant elements such that $x\leq y$, the following statements hold:
    \begin{enumerate}
        \item If $[x,y]$ is a parallelogram interval, then $|[x,y]|=|[x+\lambda,y+\lambda]|$, for $i\in\{1,2\}$.
        \item If $[x,y]$ is a pentagon interval, then there is $\{i_0, j_0\}\in\{1,2\}$ such that for every $a\in \mathbbm{Z}_{>0}$, the following holds
        \begin{equation*}
            |[x,y]|=|[x+a\varpi_{i_0},y+a\varpi_{i_0}]| \text{ and } |[x,y]|<|[x+a\varpi_{j_0},y+a\varpi_{j_0}]|.
        \end{equation*}
        Moreover, $\Pgn_{x,y}$ is a pentagon with two adjacent vertices lying on~$\mathbbm{R}_{\geq 0}\varpi_{i_0}+\alpha_{i_0}$.
        \item If $[x,y]$ is a hexagon interval, then $|[x,y]|<|[\tau_i(x),\tau_i(y)]|$, for $i\in\{1,2\}$.
    \end{enumerate}
\end{prop}
\subsection{Piecewise translations as order isomorphisms}\label{subsec:Trasl-Preserv-Bruhat-order}
We conclude the section by establishing sufficient conditions for the map $\tau_\lambda\colon[x,y]\to[\tau_\lambda(x),\tau_\lambda(y)]$ to be a poset isomorphism (Proposition~\ref{prop: trasl-por-dominat}).

Let $z,z'\in W$.
For $i\in\{1,2\}$, consider the following property
\begin{align}
    P_i(z',z): z' \leq z \iff \tau_i(z') \leq \tau_i(z).
\end{align}
We say that the property $P(z',z)$ holds if $P_1(z',z)$ and $P_2(z',z)$ hold.
\begin{lem}\label{lem: eqv menor y trsl D}
    The property $P(z',z)$ holds for $z'\in W$ and $z\in D\cup s_0D$.
\end{lem}
\begin{proof}
    Let $z\in D$.
    Corollary~\ref{cor: characterization of the Bruhat order}, says that  $\mathrm{V}(\CC_{\tau_i(z)})=\{w\tau_i(z)\}_{w\in W_f}$.
    We have:    
    \begin{align*}
        z'\leq z &\iff z'\in \CC_z\cap F_w\cap  \cen(W)&\mbox{(Corollary~\ref{cor: characterization of the Bruhat order})}\\
        &\iff z'\in \big(wz-\Cone^w(\alpha_1,\alpha_2)\big)\cap F_w\cap\cen(W)&\mbox{(Lemma~\ref{lem: gral local bruhat order})}\\
        &\iff \tau_i(z')\in \big(w\tau_i(z)-\Cone^w(\alpha_1,\alpha_2)\big)\cap F_w\cap \cen(W)\\
        &\iff \tau_i(z')\in \CC_{\tau_i(z)}\cap F_w\cap\cen(W)&\mbox{(Lemma~\ref{lem: gral local bruhat order})}\\
        &\iff \tau_i(z')\leq \tau_i(z)&\mbox{(Corollary~\ref{cor: characterization of the Bruhat order})}
    \end{align*}  
    The case \( z \in s_0D \) is proved in the same way, with the only difference being that it requires an analog of Lemma~\ref{lem: gral local bruhat order} for these elements, where the same equation holds but with \( wz \) replaced by \( ww_0z \).
    Since the proof of this analog of Lemma~\ref{lem: gral local bruhat order} follows the same reasoning, we omit it.
\end{proof}
\begin{lem}\label{lem: eqv menor y trsl X}
    Suppose that $\tau_i\colon[x,y]\to[\tau_i(x),\tau_i(y)]$ is a bijection.
    The property $P_i(z',z)$ holds for $z'\in [x,y]$ and $z\in [x,y]\cap \{\mathrm{\mathbf{x}}_n, \sigma(\mathrm{\mathbf{x}}_n)\}_{n\in \mathbbm{N}}$.
\end{lem}
\begin{proof}
    We prove the result only for \( z \in [x,y] \cap \{\mathrm{\mathbf{x}}_n\}_{n\in \mathbbm{N}} \), as the case \( z \in [x,y] \cap \{\sigma(\mathrm{\mathbf{x}}_n)\}_{n\in \mathbbm{N}} \) follows analogously.
    To continue, we examine each piecewise translation individually.
    \newline
     
    \noindent \textbf{Case $i=1$.}  
    In this case, we observe that \(\tau_1(z) \in \{\mathrm{\mathbf{x}}_n \mid n \geq 3\}\).
    Suppose that $z'\in F_w$, for $w\notin\{s_2,s_1s_2\}$.
    We have:
    \begingroup
    \allowdisplaybreaks
    \begin{align*}
        z'\leq z &\iff z'\in \CC_z\cap F_w\cap  \cen(W)&\mbox{(Corollary~\ref{cor: characterization of the Bruhat order})}\\ 
        &\iff z'\in \big(wz-\Cone^w(\a_1,\rho)\big)\cap F_w\cap\cen(W)&\mbox{(Lemma~\ref{lem: gral local bruhat order for x})}\\
        &\iff \tau_1(z')\in \big(w\tau_1z-\Cone^w(\a_1,\rho)\big)\cap F_w\cap\cen(W)\\
        &\iff \tau_1(z')\in \CC_{\tau_1z}\cap F_w\cap\cen(W)&\mbox{(Lemma~\ref{lem: gral local bruhat order for x})}\\
        &\iff \tau_1(z')\leq \tau_1(z)&\mbox{(Corollary~\ref{cor: characterization of the Bruhat order})}
    \end{align*}
    \endgroup
    Suppose that $z'\in F_w$, for $w\in\{s_2,s_1s_2\}$.
    In the third equivalence below, we use that $s_2\varpi_1=\varpi_1$:
    \begingroup
    \allowdisplaybreaks
    \begin{align*}
        z'\leq z &\iff z'\in \CC_z\cap F_w\cap  \cen(W)&\mbox{(Corollary~\ref{cor: characterization of the Bruhat order})}\\ 
        &\iff z'\in \big(ws_2z-\Cone^{ws_2}(\a_1,\rho)\big)\cap F_w\cap\cen(W)&\mbox{(Lemma~\ref{lem: gral local bruhat order for x})}\\
        &\iff \tau_1(z')\in \big(ws_2\tau_1z-\Cone^{ws_2}(\a_1,\rho)\big)\cap F_w\cap\cen(W)\\
        &\iff \tau_1(z')\in \CC_{\tau_1z}\cap F_w\cap\cen(W)&\mbox{(Lemma~\ref{lem: gral local bruhat order for x})}\\
        &\iff \tau_1(z')\leq \tau_1(z)&\mbox{(Corollary~\ref{cor: characterization of the Bruhat order})}
    \end{align*}
    \endgroup
 
    \noindent \textbf{Case \(i = 2\).}  In this case, we note that $\tau_2(z)\in D$.
    By Lemma~\ref{lem: tras intervalo es subconjunto}, and since $\tau_2$ is a bijection, we have that \( |[x,y]| = |[\tau_2(x),\tau_2(y)]| \).
    By Proposition~\ref{prop: cardinality under translation}, the polygon \(\Pgn_{x,y}\) has no edge supported on \(\mathbbm{R}_{\geq 0}\varpi_1 + \alpha_1\).
    It is not hard to see that if $v_1^{x,y}(x)\notin \partial F_\id$, then $v_1^{x,y}(x)=\mathrm{\mathbf{x}}_m$ for some $m\in \mathbbm{N}$.
    Similarly, it is not hard to see that 
    if $v_1^{x,y}(x)\in \partial F_\id$, then $v_1^{x,y}(x)=\mathrm{\mathbf{x}}_m-\frac{1}{3}\alpha_2$ for some $m\in \mathbbm{N}.$

    \begin{claim}\label{claim: intsec star and Cx}      
        $\mathcal{St}^\circ(x)\cap \mathrm{V}(\CC_z)   =\{s_2s_1z, s_1s_2s_1z\}$.
    \end{claim}
    \begin{proof}
        The proof of this claim is straightforward and will become clear to the reader once we illustrate a single case (see Figure~\ref{fig: dem biyeccion 1}).
        Consider the case $v_1^{x,y}\in\partial F_\id$.
        The green, yellow, and purple alcoves are the only possible choices for $z$ that live in~$\Pgn_{x,y}\cap \{\mathrm{\mathbf{x}}_n\}_{n\in \mathbbm{N}}$.
        In this example, we chose $z$ to be the green alcove.
        The aqua quadrilateral is $\CC_z$.
        In this case, $s_2s_1z$ is the blue dot, and $ s_1s_2s_1z$ is the red dot.
        \begin{figure}[!ht]
            \centering
            \includegraphics[width=0.48\textwidth]{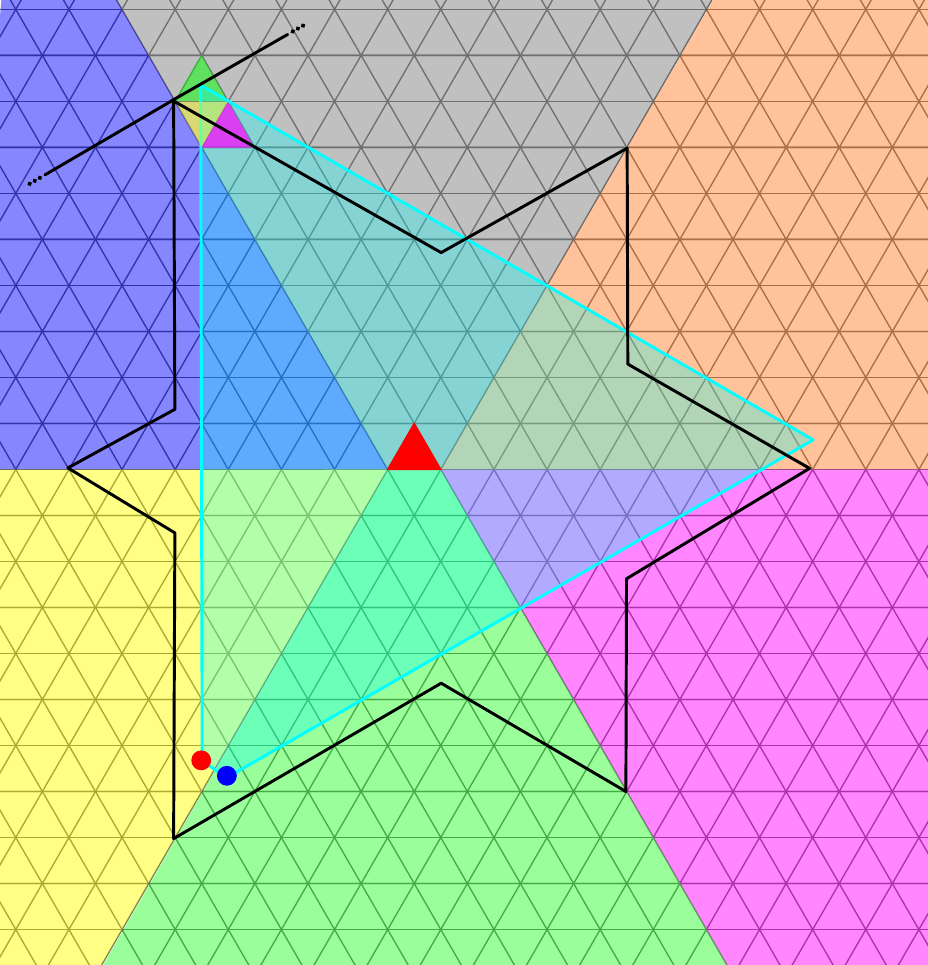}
            \caption{$\CC_{z}$ and $\mathcal{St}(x)$.}
            \label{fig: dem biyeccion 1}
        \end{figure}
    \end{proof}
    In a similar manner, it can be shown that
    \begin{equation}\label{eq: inter St and Cx trasl}
        \mathcal{St}(\tau_2(x)) \cap \mathrm{V}(\CC_{\tau_2(z)}) = \mathrm{V}(\CC_{\tau_2(z)}) \setminus \{\tau_2(z), \tau_2(s_1z)\}.
    \end{equation}
    Since the proof is longer and tedious, it will be omitted.
    
    Returning to the proof of the lemma,      
    by Claim~\ref{claim: intsec star and Cx}, we know that $\{z,s_1z\}\not\in \mathcal{St}^\circ(x)$.
    Let \(P_z\) denote the parallelogram with opposite vertices \(\{x, z\}\)  and edges parallel to \(\alpha_1\) and \(\alpha_2\).
    Similarly, let \(P_{s_1z}\) denote the parallelogram with opposite vertices \(\{x_{s_1}, s_1z\}\) and edges parallel to \(\alpha_1\) and \(\rho\).
    We have two cases:
    \begin{itemize}
        \item $z\in\mathcal{St}(x)$ (for example, the purple or the yellow alcove in Figure~\ref{fig: dem biyeccion 1}).
        The two parallelograms are degenerate and can be expressed by the formulas $
            P_{z}=\Sgm(z,x)$ and $P_{s_1z}=\Sgm(x_{s_1},s_1z).$
            By Corollary~\ref{cor: characterization of the Bruhat order} we have that $[x,z]=\{u\in[x,y]\mid u\in P_z\uplus P_{s_1z}\}$.
        Furthermore, since  
        \begin{align}
            \tau_2(\Sgm(z,x)\cup(\Sgm(x_{s_1},s_1z))&=\Sgm(z,x)\uplus(\Sgm(x_{s_1},s_1z)+\varpi_2.\\
            &=\Sgm(\tau_2(z),\tau_2(x))\uplus\Sgm(\tau_2(x_{s_1}),\tau_2(s_1(z))).
        \end{align}
        From this, it follows that $[\tau_2x, \tau_2z]=\{\tau_2u\mid u\in P_z\uplus P_{s_1z}\}$ so the claim is proved.
        
        \item If $z\in\partial\CC_z\setminus\mathcal{St}(x)$ (for example, see the green alcove in Figure~\ref{fig: dem biyeccion 1}.)  Since $z = v_1^{x,y}(x) + \frac{1}{3}\alpha_2$ and $s_1z=u_3+\frac{1}{3}\rho$, it follows that 
        \begin{align}
            \mathrm{V}(P_z) &= \{x, z, v_1^{x,y}(x), x + \frac{1}{3}\alpha_2\}, \mbox{ and } \\
            \mathrm{V}(P_{s_1z})&=\{x_{s_1}, s_1z, s_1 v_1^{x,y}(x), x_{s_1} + \frac{1}{3}\rho\}.
        \end{align} 
        Additionally, we have 
        \begin{equation}
            \begin{split}\label{eq: cenP_z}
                \mathrm{c}P_z&=\{\cen(u)\mid u\in \Sgm(v_1^{x,y}(x),x)\,\uplus\,\Sgm(y,x+\frac{1}{3}\alpha_2)\},\\
                \mathrm{c}P_{s_1z}&=\{\cen(u)\mid u\in \Sgm(x_{s_1},s_1 v_1^{x,y}(x))\uplus \Sgm(x_{s_1}+\frac{1}{3}\rho,s_1y)\}.
            \end{split}
        \end{equation}
        By Corollary~\ref{cor: characterization of the Bruhat order}, we note that $[x,z]=\{u\mid u\in P_z\uplus P_{s_1z}\}$.
        It is an easy computation to check that in fact $[\tau_2x, \tau_2z]=\{\tau_2u\mid u\in P_z\uplus P_{s_1z}\}$ so the claim is proved.\qedhere
    \end{itemize}
\end{proof}
The following lemma generalizes Lemmas~\ref{lem: eqv menor y trsl D} and~\ref{lem: eqv menor y trsl X}.
Its proof follows ideas similar to those in these lemmas, but we omit it for the sake of brevity.
\begin{lem}\label{lema basura}
    Let $i\in \{0,1,2\}$.
    The property $P(z',z)$ holds for $z'\in W$ and $z\in \d^iD\cup \d^is_0D$.
    Furthermore,  if $z',z\in [x,y]$ and $\tau_i\colon[x,y]\to[\tau_i(x),\tau_i(y)]$ the property $P_i(z',z)$ holds for $z'\in W$ and $z\in X$.
\end{lem}
\begin{lem}\label{lem: bijec-tras-implica-iso} 
    Suppose that $\tau_i\colon[x,y]\to[\tau_i(x),\tau_i(y)]$ is a bijection.
    Then $\tau_i$ defines an isomorphism of posets between $[x,y]$ and $[\tau_i(x),\tau_i(y)]$.
\end{lem}
\begin{proof}
    Recall that for $z,z'\in W$, we defined the following property:
    \begin{equation*}
        P_i(z',z): z' \leq z \iff \tau_i(z') \leq \tau_i(z).
    \end{equation*}
    The proposition is equivalent to stating that $P_i(z',z)$ holds for all $z, z' \in [x,y]$.
    By Lemma~\ref{lem: eqv menor y trsl D}, $P(z',z)$ holds when \(z \in D \cup s_0D\) and $z'\in W$.
    By Lemma~\ref{lem: eqv menor y trsl X}, $P(z',z)$ holds for \(z \in \{\mathrm{\mathbf{x}}_n,\sigma \mathrm{\mathbf{x}}_n\}_{n\in \mathbbm{N}}\) and $z'\in W$.
    The remaining cases follow from Lemma~\ref{lema basura}.
\end{proof}
\begin{lem}\label{lem: condition for isomorphism t_i}
    Let $[x,y]$ be an interval such that $x,y$ are dominant.
    \begin{enumerate}
        \item If $[x,y]$ is a parallelogram interval, then $\tau_i\colon [x,y]\to [\tau_i(x),\tau_i(y)]$ is an isomorphism of posets, for $i\in\{1,2\}$.
        \item If $[x,y]$ is a pentagon interval, then for some $i\neq j\in\{1,2\}$, we have $\tau_i\colon [x,y]\to [\tau_i(x),\tau_i(y)]$ is an isomorphism of posets and $[x,y]\not\simeq[\tau_j(x),\tau_j(y)]$.
        \item If $[x,y]$ is a hexagon interval, then $[x,y]\not\simeq [\tau_i(x),\tau_i(y)]$, for $i\in\{1,2\}$.
    \end{enumerate}
\end{lem}
\begin{proof} 
    The result follows directly from Lemmas~\ref{lem: tras es inyectivo},~\ref{lem: tras intervalo es subconjunto},~\ref{lem: bijec-tras-implica-iso}, and Proposition~\ref{prop: cardinality under translation}.
\end{proof}
\begin{figure}
    \centering 
    \includegraphics[scale=0.7]{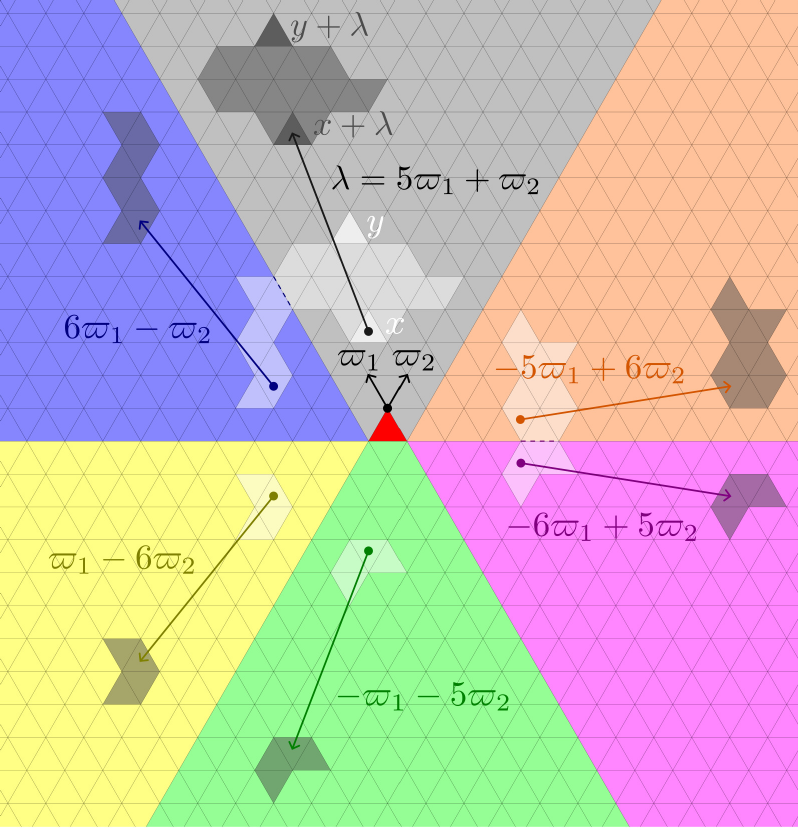}
    \caption{Visualization of the piecewise translation poset isomorphism $\tau_{\lambda} \colon [x, y] \to [x + \lambda, y + \lambda]$, where $x = \theta^s(1,0)$, $y = \theta^s(3,1)$, and $\lambda = 5\varpi_1 + \varpi_2$.}
    \label{fig: bijective translation}
\end{figure}
The following result is of fundamental importance for this paper.
An illustration of the first part is given in Figure~\ref{fig: bijective translation}.
\begin{prop}\label{prop: trasl-por-dominat}
    Let $[x,y]$ be an interval such that $x,y$ are dominant.
    We have:
    \begin{enumerate}
        \item If $[x,y]$ is a parallelogram interval, then $\tau_\lambda\colon[x,y]\to [\tau_\lambda(x),\tau_\lambda(y)]$ is an isomorphism of posets for any dominant weight $\lambda$.
        \item If $[x,y]$ is a pentagon interval (so $\operatorname{Pgn}_{x,y}$ has two adjacent vertices on~$\mathbbm{R}\varpi_{i_0}+\alpha_{i_0}$ for some $i_0\in\{1,2\}$), then $\tau_{\lambda}\colon [x,y]\to[\tau_{\lambda}(x),\tau_{\lambda}(y)]$ is an isomorphism of posets for $\lambda=a\varpi_{i_0}$ and $a\geq 0$.
        For any other $\lambda$, we have $[x,y]\not\simeq[\tau_\lambda(x),\tau_\lambda(y)]$.
        \item If $[x,y]$ is a hexagon interval, then $[x,y]\not\simeq[\tau_\lambda(x),\tau_\lambda(y)]$ for all $\lambda\neq 0$.
    \end{enumerate}
\end{prop}

%% file: Sections/Sides.tex
\section{The sides of \texorpdfstring{$\Pgn_{x,y}$}{Pgn x,y} and thick intervals}\label{section: sides}
In this section, $x,y$ are dominant elements such that $[x,y]$ is a full interval.

To avoid the heavy use of square roots, we sometimes use a normalized version of the Euclidean distance $d(-,-)$.
Given $a,b\in E$, we define
\begin{equation*}
    \operatorname{dist}(a,b)\coloneqq \frac{\sqrt{2}}{3} d(a,b).
\end{equation*}
In this normalization, the height of an alcove is $3/2$, its sides have length $\sqrt{3}$, and the distance from any vertex of an alcove to its center is $1$.
\subsection{The edges of \texorpdfstring{$\Pgn_{x,y}^{}$}{Pgn xy}}\label{subsec: edges of pgn}
The main goal of the section is to obtain a formula for the lengths of the edges of~$\Pgn_{x,y}$ in terms of the interval $[x,y]$ (Proposition~\ref{prop: length of all sides}).

Let $v\in \mathrm{V}(\Pgn_{x,y})$ be adjacent to $\cen(y)$.
There is $z\in W$ such that $z\in \Sgm(y,v)$ and
\begin{equation}\label{eq: index free definition for z of y}
    d(y, z)=\max\{d(y,w)\mid w\in W, w\in \Sgm(y,v)\}.
\end{equation}
Such a $z$ is unique.
By Equation \eqref{eq: interval as union of polygons}, we have that $z\in [x,y]$.
Equivalently, $z$ can be characterized as the unique element in~$W$ such that $z\in \Sgm(y,v)$ and $v\in z$ (we recall that by $v\in z$ we mean $v\in\overline{zA_+}$.)
Similarly, if $v\in \mathrm{V}(\Pgn_{x,y})$ is adjacent to $\cen(x)$ there is a unique $z\in W$ such that $z\in \Sgm(v,x)$ and
\begin{equation}\label{eq: index free definition for z of x}
    d(z,x)=\max\{d(w,x)\mid w\in W, w\in \Sgm(v,x)\},
\end{equation}
or equivalently, such that $z\in \Sgm(v,x)$ and $v\in z$.
\begin{lem} \label{lem: Lados Sgm y}
    Let $v\in \mathrm{V}(\Pgn_{x,y})$ be adjacent to $\cen(y)$.
    Let $z\in [x,y]$ be the unique element such that $z\in \Sgm(y,v)$ and $v\in z$ as above.
    Suppose that $z\neq y$, then we have 
    \begin{equation}
        \operatorname{dist}(y,v)=
        \begin{cases}
            \frac{3}{2}(\ell(z,y)-1)+\epsilon(y)+1 &\mbox{ if $v=\cen(z)$ and $\ell(z)$ is odd},\\ 
            \frac{3}{2}(\ell(z,y)-1)+\epsilon(y)+\frac{1}{2}&\mbox{ if $v=\cen(z)$ and $\ell(z)$ is even},\\ 
            \frac{3}{2}(\ell(z,y)-1)+\epsilon(y)+\frac{3}{2}&\mbox{ otherwise}.
        \end{cases}
    \end{equation}
    Where $\epsilon(y)=1/2$ if $y\in \Theta$ and $\epsilon(y)=1$ if $y\in\Theta^s$.
\end{lem}
\begin{proof}
    Let $p_y,p_z\in \Sgm(y,v)$ be such that $p_y\in \partial(yA_+)$, $p_z\in \partial(zA_+)$, and
    \begin{align}
        d(y,p_y)&=\max\{d(y,p)\mid p\in \Sgm(y,v)\cap \overline{yA_+}\}\\
        d(y,p_z)&=\min\{d(y,p)\mid p\in \Sgm(y,v)\cap \overline{zA_+}\}.
    \end{align}
    We will break the problem into three computations using the following formula.
    \begin{equation*}
        \operatorname{dist}(y,v)=\operatorname{dist}(y,p_y)+\operatorname{dist}(p_y,p_z)+\operatorname{dist}(p_z,v).
    \end{equation*}
    Since $\Sgm(y,v)$ is parallel to a simple root, we have that either $p_y$ (resp.\@ $p_z$) is a vertex or the midpoint of an edge of~$y$ (resp.\@ $z$).
    If $y\in \Theta$, $y$ is down-oriented, so $p_y$ is the midpoint of an edge of~$y$.
    If $y\in \Theta^s$, $y$ is up-oriented,  $p_y$ is one of the bottom vertices of~$y$.
    In both cases, $\operatorname{dist}(y,p_y)=\epsilon(y)$.
    
    Similarly, when $z$ is down-oriented, $p_z$ is a vertex of~$z$, and when $z$ is up-oriented, we have that $p_z$ is a midpoint of one of the edges of~$z$.
    By Lemma~\ref{lem: length determines orientation}, we have
    \begin{equation}
        \operatorname{dist}(p_z,z)=\begin{cases}
        \frac{1}{2}&\mbox{if $\ell(z)$ is even,}\\
        1 & \mbox{if $\ell(z)$ is odd.}
        \end{cases}
    \end{equation}
    If $v = \cen(z)$, then obviously $\operatorname{dist}(p_z,z)=\operatorname{dist}(p_z,v)$.
    If $v\neq \cen(z)$, then \(\Sgm(p_z, v)\) is a height of \(zA_+\), so \(\operatorname{dist}(p_z, v) = \frac{3}{2}\).

    By Corollary~\ref{cor: characterization of the Bruhat order} and Lemma~\ref{lem: local bruhat order}, an element $w$ whose center is in~$\Sgm(y,v)$ is also in~$[z,y]$.
    As in the proof of Lemma~\ref{lem: if distance is 4 then difference is more than 2 times the simple root}, we have that $\{u\in W \mid u\in \Sgm(y,v)\}$ is a maximal chain in~$[z,y]$.
    As 
    \begin{equation*}
        \vert \{u\in W \mid u\in \Sgm(y,v)\} \vert =\ell(z,y)+1,
    \end{equation*}
    one has 
    \begin{equation*}
        \vert \{u\in W \mid u\in \Sgm(p_y,p_z)\} \vert =\ell(z,y)-1.
    \end{equation*}
    Since for any $u\in W$ with $\cen(u)\in \Sgm(p_y,p_z)$, we have that $uA_+$ intersects $\Sgm(p_y,p_z)$ along one of its heights, the equality $\operatorname{dist}(p_y,p_z)=\frac{3}{2}(\ell(z,y)-1)$ follows.
\end{proof}
The next lemma is proved in an analogous way to the previous one, so we omit the proof.
\begin{lem} \label{lem: Lados Sgm x}
    Let $v\in \mathrm{V}(\Pgn_{x,y})$ be adjacent to $\cen(x)$.
    Let $z\in [x,y]$ be the unique element such that $z\in \Sgm(v,x)$ and $v\in z$.
    If $z\neq x$, we have 
    \begin{equation}
        \operatorname{dist}(v,x)=
        \begin{cases}
            \frac{3}{2}(\ell(x,z)-1)+\eta(x)+\frac{1}{2}&\mbox{ if $v=\cen(z)$ and $\ell(z)$ is odd},\\
            \frac{3}{2}(\ell(x,z)-1)+\eta(x)+1&\mbox{ if $v=\cen(z)$ and $\ell(z)$ is even},\\ 
            \frac{3}{2}(\ell(x,z)-1)+\eta(x)+\frac{3}{2} &\mbox{ otherwise}.
        \end{cases}
    \end{equation}
    Where $\eta(x)=1$ if $x\in \Theta$ and $\eta(x)=1/2$ if $x\in\Theta^s$.
\end{lem}
\begin{lem}\label{lem: length of the middle segment}
    Let $v_0,v_1\in \mathrm{V}(\Pgn_{x,y})$ be vertices that are either identical or adjacent, such that $v_0$ is adjacent to $\cen(x)$ and $v_1$ is adjacent to $\cen(y)$.
    Let $z_0,z_1\in [x,y]$ be the elements such that $z_0\in \Sgm(x,v_0)$, $v_0\in z_0$, $z_1\in \Sgm(v_1,y)$, and $v_1\in z_1$.
    We have
    \begin{equation}
        \operatorname{dist}(v_0,v_1)=    
        \begin{cases}
            \frac{\sqrt{3}}{2}(\ell(z_0,z_1)-1) &\mbox{ if $v_0\neq \cen(z_0)$},\\
            0 &\mbox{  if $v_0=\cen(z_0)$}.
        \end{cases}
    \end{equation}
\end{lem}
\begin{proof}
    If $v_0=\cen(z_0)$, then $z_0=z_1$ so $v_0=v_1$ and $d(v_0,v_1)=0$.
    
    Let us suppose that $v_0\neq \cen(z_0)$.
    By Remark~\ref{rem: geometry-of-vertex-of-Pgn} and since $z_0\in \Sgm(x,v_0)$, it follows that $v_0$ is one of the top vertices of the down-oriented $z_0$.
    By Lemma~\ref{lem: length determines orientation},
    this can only happen if $\ell(z_0)$ is odd,
    that is, $z_0=\theta(m,n)$ for some $m,n\geq 0$ or $z_0\in X^{\mathrm{odd}}$.
    In the first case, we have $z_1=\theta^s(m,n)$ which leads to $v_0=v_1$ and $d(v_0,v_1)=0=\ell(z_0,z_1)-1$.
    
    For the remainder of the proof, we address the remaining case by assuming that $z_0\in X^{\mathrm{odd}}$.
    Note that $v_0\in \partial F_\id$.
    Since $v_0$ and $v_1$ are adjacent, we have $\Sgm(v_0,v_1)\subset\partial F_\id$.
    This implies that $z_0A_+\cup z_1A_+$ is contained in a strip $\{v\in E : -1<(x,\alpha_i)<0\}$ for some $i\in\{1,2\}$.
    Without loss of generality, we can assume $i=2$.
    So $z_0=\mathrm{\mathbf{x}}_{2k-1}$ for some $k\geq 3$.
    There is $j\geq 3$ such that either $z_1=\mathrm{\mathbf{x}}_{2j}$ or $z_1=\mathrm{\mathbf{x}}_{2j-1}$.
    \begin{itemize}
        \item $z_1=\mathrm{\mathbf{x}}_{2j-1}$.
        From Lemma~\ref{lem: difference of lengths for alcoves in the wall}, we have $z_1-z_0=
        (j-k)\varpi_1$.
        Let $c_1$ be the top vertex of~$z_1$ that belongs to $\partial F_\id$, then $c_1-v_0=(j-k)\varpi_1$.
        Since $z_0$ is down-oriented, $v_1\in \partial F_\id$ is the midpoint of the left edge of~$z_1$.
        More precisely, we have that $v_1=c_1-(1/2)\varpi_1$, hence
        \begin{equation*}
            v_1-v_0=(j-k-1/2)\varpi_1=\frac{\ell(z_0,z_1)-1}{2}\varpi_1.
        \end{equation*}
        Since $\operatorname{dist}(0,\varpi_1)=\sqrt{3}$, the lemma follows.
        \item $z_1=\mathrm{\mathbf{x}}_{2j}$.
        In this case, $z_1$ is up-oriented, so $v_1\in \partial F_\id$ is the bottom-left vertex of~$z_1$.
        Note that $v_1$ is also the top-left vertex of~$\mathrm{\mathbf{x}}_{2j-1}$.
        From the same computation as in the previous case, we have that
        \begin{equation*}
            v_1-v_0=(j-k)\varpi_1=\frac{\ell(z_0,z_1)-1}{2}\varpi_1,
        \end{equation*}
        and the lemma follows.\qedhere
    \end{itemize}
\end{proof}
\begin{defn}\label{def: vertices centres and midpoints}
    Let $z\in W$.
    We define $C(z)$ as the singleton containing the element $\cen(zA_+), \mathrm{V}(z)$ as the three-element vertex set of~$zA_+$, and $\mathrm{ME}(z)$ as the three-element set of midpoints of the edges of~$zA_+.$ 
\end{defn}
\begin{prop}\label{prop: length de adjacencia implica posicion}
    Let $z_0, z_1 \in [x,y]$ and $v_0, v_1 \in \mathrm{V}(\Pgn_{x,y})\setminus \{\cen(x), \cen(y)\}$ be vertices that are either the same or adjacent, such that $v_0$ is adjacent to $x$ and $v_1$ is adjacent to $y$.
    If $v_i \in z_i$ for $i\in\{0,1\}$, $z_0\in \Sgm(v_0,x)$, and $z_1\in \Sgm(y,v_1)$, then one of the following statement holds
    \begin{itemize}
        \item If $\ell(z_0,z_1)=0$, then $v_0\in C(z_0)$ and $v_1\in C(z_1).$ 
        \item If $\ell(z_0,z_1)=1$, then 
        $v_0\in \mathrm{V}(z_0), $  $v_1\in \mathrm{V}(z_1)$ and 
        $v_0 = v_1$ is a vertex of the segment $\overline{z_0A_+} \cap \overline{z_1A_+}$.
        \item If $\ell(z_0, z_1) > 1$, then  $v_0$ and $v_1$ are different points of~$\partial F_\id$.
        Furthermore:
        \begin{itemize}
            \item If $\ell(z_0, z_1)$ is even, $v_0\in \mathrm{V}(z_0)$ and $v_1\in \mathrm{V}(z_1)$.
            \item If $\ell(z_0, z_1)$ is odd, $v_0\in \mathrm{V}(z_0)$ and $v_1\in \mathrm{ME}(z_1)$.
        \end{itemize}
    \end{itemize}
\end{prop}
\begin{proof}
    By Lemma~\ref{lem: local bruhat order}, it follows that $z_0\leq z_1$.
    Now, recall that $z_0$ (resp.\@ $z_1$) is the unique element in~$[x,y]$, such that $z_0\in \Sgm(v_0,x)$ and $v_0\in z_0$ (resp.\@ $z_1\in \Sgm(y,v_1)$ and $v_1\in z_1$).
    \begin{itemize}
        \item $\ell(z_0,z_1)=0$.
        This is equivalent to the equality $z_0=z_1$.
        By Remark~\ref{rem: geometry-of-vertex-of-Pgn}, $v_0\in \{C(z_0),\mathrm{V}(z_0)\}$, but if $v_0\in \mathrm{V}(z_0)$ then $z_1\neq z_0,$ which is a contradiction, so $v_0\in C(z_0).$ This
        implies that $v_0=\cen(z_0)=v_1=\cen(z_1)$, thus proving this bullet.
        \item $\ell(z_0,z_1)=1$.
        Once again, by Remark~\ref{rem: geometry-of-vertex-of-Pgn},  we have $v_0\in \{C(z_0),\mathrm{V}(z_0)\}$.
        In either case, applying the formulas from Lemma~\ref{lem: length of the middle segment}, we find $d(v_0,v_1)=0$, which implies $v_0=v_1$.
        If $v_0=\cen(z_0)$, then $z_0=z_1$, which is a contradiction, so $v_0\in \mathrm{V}(z_0)$.
        A vertex of an alcove cannot be the center or the midpoint of an edge of another alcove, which implies that $v_1\in \mathrm{V}(z_1).$
        \item $\ell(z_0,z_1)>1$.
        As in the last bullet, if $v_0\in C(z_0)$ then $z_0=z_1$, which is a contradiction.
        Then, by Lemma~\ref{lem: length of the middle segment}, we have $d(v_0,v_1)>0$.
        By definition of~$\Pgn_{x,y}$, this can only happen if both $v_0,v_1$ lie on~$\partial F_\id$.
        Once again, by Remark~\ref{rem: geometry-of-vertex-of-Pgn}, we have that $v_0\in \mathrm{V}(z_0)$, with  $z_0$ down-oriented.
        
        Similarly, as $v_1\in\overline{z_1A_+}\cap \partial F_\id,$ we observe that
        \begin{equation}\label{eq vertex or edge}
            \overline{z_1A_+}\cap \partial F_\id=
            \begin{cases}
                \mbox{a vertex of~$z_1$}& \mbox{if $z_1$ is up-oriented;}\\
                \mbox{an edge of~$z_1$}& \mbox{if $z_1$ is down-oriented.}
            \end{cases}
        \end{equation}
        As $z_0$ is down-oriented, by Lemma~\ref{lem: length determines orientation}, we have that $\ell(z_0)$ is odd.
        Additionally, by applying Lemma~\ref{lem: length determines orientation} again, we have that:
        \begin{itemize}
            \item  If $\ell(z_0,z_1)$ is odd, then $z_1$ is up-oriented.
            By Equation \eqref{eq vertex or edge}, $v_1\in \mathrm{V}(z_1)$.
            \item If $\ell(z_0,z_1)$ is even, then   $z_1$ is down-oriented.
            Here $v_1$ is the midpoint of the edge $\overline{z_1A_+}\cap \partial F_\id$, or rephrasing, $v_1\in \mathrm{ME}(z_1).$\qedhere
        \end{itemize}
    \end{itemize}
\end{proof}
We will rewrite the results of the section using a more convenient notation, which is compatible with Definition~\ref{def: label of vertices}.
\begin{nota}\label{nota: z_i's}
    Let $z_i^{x,y}(y)\in W$ be the element as in Equation \eqref{eq: index free definition for z of y} such that $z_i^{x,y}(y)\in \Sgm(y, v_i^{x,y}(y))$ and
    \begin{equation}\label{eq: zeta(y) definition}
        d(y, z_i^{x,y}(y))=\max\{d(y,w)\mid w\in W, w\in \Sgm(y, v_i^{x,y}(y))\}.
    \end{equation}
    Similarly, let $z_i^{x,y}(x)\in W$ be the element as in Equation \eqref{eq: index free definition for z of x} such that $z_i^{x,y}(x)\in \Sgm(v_i^{x,y}(x),x)$ and
    \begin{equation*}
        d(z_i^{x,y}(x),x)=\max\{d(w,x)\mid w\in W, w\in \Sgm(v_i^{x,y}(x),x)\}.
    \end{equation*}
    Furthermore, if the interval $[x,y]$ is clear from the context, we remove the superscript $(-)^{x,y}$ from the notation, so for $1\leq i \leq 2$, we write $z_i(x)$, $z_i(y)$, $v_i(x)$, and $v_i(y)$ instead.
\end{nota}
If $\{i,j\}=\{1,2\}$, note that it is possible that $z_j(y)=z_i(x)$ and/or $v_j(y)=v_j(x)$.
In the latter case, the edge $\{v_j(y), v_i(x)\}$ is degenerate.
We have the set
\begin{equation*}
    \mathrm{V}(\Pgn_{x,y})\setminus \{\cen(x),\cen(y)\}=\{v_1(x), v_1(y), v_2(x), v_2(y)\},
\end{equation*}
consists of~$2,3,$ or $4$ elements; these are precisely the cases where $\Pgn_{x,y}$ is a parallelogram, a pentagon, or a hexagon.
\begin{rem}\label{rem: distancias}
    Let $\{i,j\}=\{1,2\}$.
    We have: 
    \begin{align*}
        \cen(y) &=v_j(y)+d(y,v_j(y))\alpha_j/||\alpha_j||\\
        v_j(y)&=v_i(x)+d(v_j(y),v_i(x))\varpi_i/||\varpi_i||\\
        v_i(x)&=\cen(x)+d(v_i(x),x)\alpha_i/||\alpha_i||
    \end{align*}
\end{rem}
The next two results are a summary of the Lemmas~\ref{lem: Lados Sgm y},~\ref{lem: Lados Sgm x},~\ref{lem: length of the middle segment}, and Proposition~\ref{prop: length de adjacencia implica posicion} in terms of Notation~\ref{nota: z_i's}.
\begin{cor}\label{cor: z1(x) es menor a z2(y)}
    Let $\{i,j\}=\{1,2\}$, then $z_i(x)\leq z_{j}(y)$.
\end{cor}
\begin{prop}\label{prop: length of all sides}
    Let $\{i,j\}=\{1,2\}$, then we have:
    \begin{align*}
        \operatorname{dist}(y,v_i(y))&=
        \begin{cases}
            \frac{3}{2}(\ell(z_i(y),y)-1)+\epsilon(y)+\gamma(z_i(y)) &\mbox{ if $v_i(y)=\cen(z_i(y))$},\\
            \frac{3}{2}(\ell(z_i(y),y)-1)+\epsilon(y)+\frac{3}{2}&\mbox{ otherwise}.
        \end{cases}\\
        \operatorname{dist}(v_{j}(y),v_{i}(x))&=    
        \begin{cases}
            \frac{\sqrt{3}}{2}(\ell(z_{i}(x),z_{j}(x))-1) &\mbox{ if $v_{i}(x)\neq \cen(z_{i}(x))$},\\
            0 &\mbox{ otherwise}.
        \end{cases}\\
        \operatorname{dist}(v_i(x),x)&=
        \begin{cases}
            \frac{3}{2}(\ell(x,z_i(x))-1)+\eta(x)+\frac{3}{2}-\gamma(z_i(x))&\mbox{ if $v_i(x)=\cen(z_i(x))$},\\ 
            \frac{3}{2}(\ell(x,z_i(x))-1)+\eta(x)+\frac{3}{2} &\mbox{ otherwise}.
        \end{cases}
    \end{align*}
    Where    
    \begin{equation}
        \begin{aligned}
            \epsilon(y) &= 
            \begin{cases} 
                \frac{1}{2} & \text{if } y\in \Theta \\
                1 & \text{if } y\in\Theta^s 
            \end{cases}; 
            &\eta(x)&= 
            \begin{cases} 
                \frac{1}{2} & \text{if } x \in\Theta^s\\
                1 & \text{if } x \in \Theta
            \end{cases}; 
            &\gamma(z)&= 
            \begin{cases} 
                \frac{1}{2} & \mbox{if $\ell(z)$ is even}  \\
                1 & \mbox{if $\ell(z)$ is odd.}
            \end{cases} 
        \end{aligned}
    \end{equation}
\end{prop}
\begin{prop}\label{Prop: vert are det by length}
    $\{i,j\}=\{1,2\}$.
    \begin{itemize}
        \item If $\ell(z_i(x),z_j(y))=0$, then $v_i(x)\in C(z_i(x))$ and $v_j(y)\in C(z_j(y))$.
        \item If $\ell(z_i(x),z_j(y))=1$, then 
        $v_i(x)\in \mathrm{V}(z_i(x))$, $v_j(y)\in \mathrm{V}(z_j(y))$, and $v_i(x) = v_j(y)$ is a vertex of the segment $\overline{z_i(x)A_+} \cap \overline{z_j(y)A_+}$.
        \item If $\ell(z_i(x), z_j(y)) > 1$, then  $v_i(x)$ and $v_j(y)$ are different points of~$\partial F_\id$.
        Furthermore:
        \begin{itemize}
            \item If $\ell(z_i(x), z_j(y))$ is even, $v_i(x)\in \mathrm{V}(z_i(x))$ and $v_j(y)\in \mathrm{V}(z_j(y)).$
            \item If $\ell(z_i(x), z_j(y))$ is odd, $v_i(x)\in \mathrm{V}(z_i(x))$ and $v_j(y)\in \mathrm{ME}(z_j(y)).$ 
        \end{itemize}
    \end{itemize}
\end{prop}

\subsection{Thick intervals}\label{subsec: thicks}
Our aim now is to prove Proposition~\ref{prop: seg max fixed under iso}, which imposes constraints on the possible poset isomorphisms between Bruhat intervals under certain mild assumptions.
Recall that $x,y\in D$ and $[x,y]$ is a full interval.
\begin{defn}
     The set of maximal length elements in~$\DD^l(x)\cap [x,y]$ is denoted by $\DC_{[x,y]}^{l,\max}$.
     The set of minimal length elements in~$\DD^u(y)\cap [x,y]$ is denoted by $\DC_{[x,y]}^{u,\min}$.
\end{defn}
\begin{defn}
    For \( z \in \mathcal{D}_{[x,y]}^{l,\max}\), the subset of maximal length elements of \( (\mathcal{D}^l(x)\cap [x,y]) \setminus [x,z] \) is denoted  \( A_{x,z} \).
    For \( z \in \mathcal{D}_{[x,y]}^{u,\min}\), the subset of minimal length elements of \( (\mathcal{D}^u(y)\cap [x,y]) \setminus [z,y]\) is denoted \( B_{z,y} \).
\end{defn}
\begin{rem}
    Since $[x,y]$ is a full interval, the sets \( A_{x,z} \) and \( B_{z,y} \) are nonempty.
    Indeed, since \( [x,z] \) has at most two upper covers of \( x \), it follows from Lemma~\ref{lem: cardinality of upper covers} that \( \mathcal{D}^l(x) \setminus [x,z] \) is nonempty.
    An analogous argument applies to \( B_{z,y} \), using Lemma~\ref{lem: cardinality of lower covers}.
\end{rem}
\begin{lem}\label{lem: max dih x unicamente det}
    Let $z_1 \in \mathcal{D}_{[x,y]}^{l,\max}$ and $z_2\in A_{x,z_1}$.
    If $\ell(x, z_i) \geq 4$ for $i\in\{1,2\}$, then $\{z_1(x), z_2(x)\}=\{z_1,z_2\}$.
    In particular,
    \begin{equation*}
        \mathcal{D}_{[x,y]}^{l,\max}\subset \{z_1(x), z_2(x)\}.
    \end{equation*}
\end{lem}
\begin{proof}
    Let $u_1,u_2,u_3,u_6,$ and $x_2$ be vertices of~$\mathcal{St}(x)$ as in Definition~\ref{def: estrella}.
    By Proposition~\ref{prop: Dih-resumen}, $z_1, z_2 \in \DD_{\mathrm{sgm}}^l(x)$.
    Thus, 
    \begin{equation}
        \label{eq: u2 u6}
        z_1, z_2\in \Sgm(u_1,u_3)\cup \Sgm(u_2,u_6).
    \end{equation}
    Without loss of generality, assume that $z_1\in \Sgm(u_1,u_3)$.
    In this case, we have that $z_1\in F_\id\cup F_{s_1}$.
    Note that
    \begin{equation*}
        \Sgm(u_1,u_3)\setminus \mathcal{St}^\circ(x) = \Sgm(u_1,x)\uplus \Sgm(x_2,u_3).
    \end{equation*}
    Since $z_1\in [x,y]$,
    by Equation \eqref{eq: interval as union of polygons}, we have that $z_1\in \Pgn_{x,y}\cup\Pgn_{x,y}^{s_1}$.
    From Proposition~\ref{prop: pol otras regiones}, it is not hard to see that
    \begin{align*}
        \cPgn_{x,y}\cap \Sgm(u_1,u_3)&=\cSgm(z_1(x),x) \\
        \cPgn^{s_1}_{x,y}\cap \Sgm(u_1,u_3)&=\cSgm(x_2,s_1z_1(x)).
    \end{align*}
    This implies that 
    $z_1\in \Sgm(z_1(x),x)\uplus \Sgm(x_2,s_1z_1(x))$.
    \newline

    \noindent\textbf{Case A.} $z_1\in\Sgm(x_2,s_1z_1(x))$.
    Note that 
    \begin{equation*}
        s_1z_1(x)\in F_{s_1}=\delta^2s_0D\cup \delta \{\mathrm{\mathbf{x}}_n\mid n\geq 2\}.
    \end{equation*}
    By Corollary~\ref{cor: characterization of the Bruhat order}, 
    we have that $s_2s_0s_2\cdot s_1z_1(x)$ is a vertex of~$\CC_{s_1z_1(x)}$.
    Since $s_2H_{\rho,-1} =H_{\alpha_1,-1}$, it follows that $s_2s_0s_2=s_{\alpha_1,-1}=r_1
    $.
    By definition of~$\mathcal{St}(x)$ we have that $r_1(\Sgm(x_2,u_3))\subset \Sgm(x,u_1)$, so $r_1s_1z_1(x)\in \Sgm(x,u_1)$.
    As
    \begin{equation*}
        r_1s_1z_1(x),s_1z_1(x),\in \mathrm{V}(\CC_{s_1z_1(x)}),
    \end{equation*}
    by convexity of~$\CC_{s_1z_1(x)}$, we have that $\Sgm(x_2,s_1z_1(x))\subset \CC_{s_1z_1(x)}$, so $z_1\leq s_1z_1(x)<z_1(x)$, the strict inequality holds because $z_1(x)\in F_{\id}.$
    
    Since $z_1(x)\in \DD_{\mathrm{sgm}}^{l}(x)$, Proposition~\ref{prop: Dih-resumen} implies that $z_1(x)\in \DD^{l}(x)$.
    This contradicts the maximality of~$\ell(z_1)$ in~$\mathcal{D}^l(x)$.
    \newline
    
    \noindent\textbf{Case B.} $z_1\in\Sgm(z_1(x),x)$.
    Suppose by contradiction that $z_1\neq z_1(x)$.
    Clearly $z_1$ and $z_1(x)$ differ by a multiple of~$\alpha_1$.
    By Lemma~\ref{lem: comparable elements with nonempty intersection}, there exists a chain $z_1=t_0\lessdot t_1 \lessdot \hdots \lessdot t_n=z_1(x)$ of elements in~$\Sgm(z_1(x),x)\subset F_\id$, similar to the one described in the proof of Lemma~\ref{lem: if distance is 4 then difference is more than 2 times the simple root}.
    In particular $z_1< z_1(x)$.
    Since $z_1(x)\in \DD_{\mathrm{sgm}}^{l}(x)$, Proposition~\ref{prop: Dih-resumen} implies $z_1(x)\in \DD^{l}(x)$.
    This contradicts the maximality of~$z_1$ in~$\mathcal{D}^l(x)$, so $z_1=z_1(x)$.

    It remains to prove that $z_2=z_2(x)$.
    Recall that by Equation \eqref{eq: u2 u6} we have that $z_2\in \Sgm(u_1,u_3)\cup \Sgm(u_2,u_6)$.
    Let us show that $z_2\in \Sgm(u_2,u_6)$.
    Suppose, by contradiction, that $z_2 \in\Sgm(u_1, u_3)$.
    As in \textbf{Case A}, if $z_2\in \Sgm(x_2,s_1z_1(x))$ we have that $z_2\leq s_1z_1(x)<z_1(x)$ which yields a contradiction.
    Otherwise, $z_2\in\Sgm(z_1(x),x)$ and by the same reasoning as before,
    there exists a chain $z_2=t_0\lessdot t_1 \lessdot \hdots \lessdot t_n=z_1(x)$ of elements in~$\Sgm(z_1(x),x)\subset F_\id$, so  $z_2\leq z_1(x)$, and
    since $z_1=z_1(x)$, this contradicts the assumption that $z_2 \in A_{x, z_1}$.
    Therefore, we conclude that $z_2 \in \Sgm(u_2, u_6)$.
    
    The proof of the fact that $z_2=z_2(x)$ is analogous to that of~$z_1=z_1(x)$, relying on the maximality of \( \ell(z_2) \), so we will omit the details.
\end{proof}
The next lemma is the analog of Lemma~\ref{lem: max dih x unicamente det} for the elements $z_1(y)$ and $z_2(y)$.
Since the proof follows similarly to that of Lemma~\ref{lem: max dih x unicamente det}, we will omit certain details that are not essential.
\begin{lem}\label{lem: max dih y unicamente det}
    Let $z_1 \in \mathcal{D}_{[x,y]}^{u,\min}$ and $z_2\in B_{x,z_2}$.
    If $\ell(x, z_i) \geq 4$ for $i\in\{1,2\}$, then $\{z_1(y), z_2(y)\}=\{z_1,z_2\}$.
    In particular, 
    \begin{equation*}
        \mathcal{D}_{[x,y]}^{u,\min}\subset \{z_1(y), z_2(y)\}.
    \end{equation*}
\end{lem}
\begin{proof}
    By Proposition~\ref{prop: Dih-resumen}, $z_1, z_2 \in \DD_{\mathrm{sgm}}^u(y)$.
    Thus, $z_1, z_2\in \Sgm(y,s_1y)\cup\,\Sgm(y, s_2y)$.
    This implies that $z_1,z_2\in F_{\id}\cup F_{s_1}\cup F_{s_2}$.
    Without loss of generality, assume that $z_1\in \Sgm(y,s_1y)$.
    Recall from Section~\ref{subsec: description de mayores que x} that $r_1=s_{\alpha_1,-1}$ and $r_2=s_{\alpha_2,-1}$.
    Since $z_1\in [x,y]$, by Equation~\eqref{eq: interval as union of polygons} and as $z_1\in F_\id\cup F_{s_1}$, we have that $z_1\in \Pgn_{x,y}\cup\Pgn_{x,y}^{s_1}$.
    From Proposition~\ref{prop: pol otras regiones}, it is not hard to see that
    \begin{equation}
        \begin{split}\label{eq: sgm Fid}
            \cPgn_{x,y}\cap \Sgm(y,s_iy) \cap F_{\id}&=\cSgm(y,z_i(y)), \\
            \cPgn^{s_i}_{x,y}\cap \Sgm(y,s_iy) \cap F_{s_i}&=\cSgm(r_iz_i(y),s_iy), \mbox{ for $i\in\{1,2\}$}.
        \end{split}
    \end{equation}
    As in \textbf{Case A} of the proof of Lemma~\ref{lem: max dih x unicamente det}, we conclude that $ z_1 \in F_\id$.
    Furthermore, since $z_1$ and $z_1(y)$ differ by a multiple of~$\alpha_1$, and both are in~$F_\id$, we apply the reasoning as in \textbf{Case B} of the proof of Lemma~\ref{lem: max dih x unicamente det}, to conclude that $z_1=z_1(y)$.
    
    It remains to prove that $z_2=z_2(x)$.
    Before proceeding, we will show that $z_2\in \Sgm(y,s_2y)\cap F_\id$.
    First, suppose that $z_2\not\in F_\id$.
    Then 
    \begin{itemize}
        \item $z_2\in \delta s_0D\cup  (\sigma\delta)\{\mathrm{\mathbf{x}}_n\mid n\geq 2\}$, or
        \item $z_2\in\delta^2s_0 D\cup \delta \{\mathrm{\mathbf{x}}_n\mid n\geq 2\}$.
    \end{itemize}
    In the first case, by Corollary~\ref{cor: characterization of the Bruhat order}, we have that $r_2z_2\in \mathrm{V}(\CC_{z_2})$.
    In the second case, by Corollary~\ref{cor: characterization of the Bruhat order},  $r_1z_2\in \mathrm{V}(\CC_{z_2})$.
    In both cases, since
    \begin{equation*}
        \Sgm(r_iz_i(y),y)~\subset~ r_i\Sgm(r_iz_i(y),s_iy),
    \end{equation*}
    we obtain that $z_i(y)< z_2$.
    Since $z_i(y)\in \mathcal{D}^u_{\mathrm{Sgm}}(y)$, Proposition~\ref{prop: Dih-resumen} implies that $z_i(y)\in \mathcal{D}^u_{[x,y]}$.
    This contradicts the minimality of~$z_2$.
    Hence $z_2\in F_\id.$
    
    Now, suppose by contradiction that $z_2\in\Sgm(y, s_1y)$.
    Clearly, $\ell(z_2)\leq \ell(z_1)$.
    By the same reasoning as in \textbf{Case B} of the proof of Lemma~\ref{lem: max dih x unicamente det}, we can prove that $z_2\leq z_1$.
    This contradicts the assumption that $z_2 \in B_{x, z_1}$.
    Therefore, we conclude that $z_2 \in \Sgm(y, s_2y)\cap F_\id$.
    By Equation \eqref{eq: sgm Fid} we have that $z_2\in\Sgm(y,z_2(y))$.
    Since $z_2$ and $z_2(y)$ differ by $\alpha_2$ and both are in~$F_\id$, we can apply the reasoning as in \textbf{Case B} of the proof of Lemma~\ref{lem: max dih x unicamente det} to conclude that $z_2=z_2(y)$.
\end{proof}
In the following definition, the interval $[x,y]$ is not assumed to be full.
\begin{defn}\label{def: thick interval}
    We say that $[x,y]$ ($x,y$ dominant, $[x,y]$ nonempty) is \emph{thick} if $\ell(x, z_i(x)) \geq 4$ and $\ell(z_i(y),y) \geq 4$ for $i\in\{1,2\}$.
    Equivalently, by Lemma~\ref{lem: if distance is 4 then difference is more than 2 times the simple root} $[x,y]$ is thick if $z_i(x)-x\in \mathbbm{R}_{\geq 2}\alpha_i$ and $y-z_i(y)\in \mathbbm{R}_{\geq 2}\alpha_i$.
\end{defn}
\begin{lem}\label{lem: thick implies full}
    Let $x,y\in D$ be dominant elements, such that $[x,y]$ is a thick interval.
    Then
    \begin{enumerate}
        \item $x,y\in D+\rho$ and $x+\rho,y-\rho\in [x,y]$.
        \item $[x,y]$ is full.
    \end{enumerate}
\end{lem}
\begin{proof}
    \begin{enumerate}
        \item Suppose by contradiction that $x\in D\setminus D+\rho$, so there is $1\leq i\leq 2$ such that $(\varpi_i,\cen(x))<1$.
        Let $t$ be such that $\cen(z_i(x))=\cen(x)+t\alpha_i$.
        Since $z_i(x)\in F_{\id}$, we have that 
        \begin{equation*}
            -1<(\varpi_i,\cen(z_i(x)))=(\varpi_i,\cen(x))-t
        \end{equation*}
        which implies that $t<2$, contradicting thickness.
        The argument for $y$ is analogous, so $x,y\in D+\rho$.
        It remains to prove that $x+\rho,y-\rho\in [x,y]$.
        
        Let $\{i, j\} = \{1, 2\}$.
        We use the notations from Definition~\ref{def: label of vertices}.
        By  Equation~\ref{eq: zeta(y) definition} and the definition of~$v_i(y)$, it follows that $z_i(y)-v_i(y)\in\mathbbm{R}_{\geq0}\a_i$.
        In the following, the fourth (in)equality follows from thickness, and the other (in)equalities are immediate:  
        \begin{align*}  
            (\varpi_i, \cen(x)) &= (\varpi_i, v_j(x)) \\  
            &\leq (\varpi_i, v_i(y)) \\  
            &\leq (\varpi_i, \cen(z_i(y))) \\
            &\leq (\varpi_i, \cen(y) - \alpha_i) \\  
            &=(\varpi_i, \cen(y - \rho))  
        \end{align*}  
        Since $y \in D$, we have $y-\rho \in F_{\id}$.
        By Corollary~\ref{cor: equivalence of being in the hexagon in the gray region} and the above, we conclude that $x \leq y-\rho$.
        Similarly,  
        \begin{equation*}  
            (\varpi_i, \cen(y - \rho)) = (\varpi_i, \cen(y) - \alpha_i) \leq (\varpi_i, \cen(y)).
        \end{equation*}  
        Thus, by Corollary~\ref{cor: equivalence of being in the hexagon in the gray region}, we have $y-\rho \leq y$.
        The inequalities $x \leq x + \rho \leq y$ follow analogously.
        \item By the previous part, it is enough to show that  $U(x)\subset [x,x+\rho]$ and $L(y)\subset [y-\rho, y]$.
        \newline
    
        \noindent\textbf{Inclusion $\operatorname{L}(y)\subset [y-\rho,y]$.}
        We only need to prove that $z\in \operatorname{L}(y)$ implies $y-\rho\leq z$.
        By Lemma~\ref{lem: set of lower covers}, we have that  $\operatorname{L}(y)=\operatorname{L}_D(y)\cup\{s_1y,s_2y\}$ and $\operatorname{L}(y)\cap D=\operatorname{L}_D(y)$.
        
        Suppose that $z\in \operatorname{L}_D(y)$.
        We have two cases:
        \begin{itemize}
            \item $y=\theta(m,n)$.
            We have that $y-\rho=\theta(m-1,n-1)$.
            Clearly,
            \begin{equation*}
                \theta(m-1,n-1)\leq \theta^s(m-1,n-1)\leq \theta^s(m-1,n), \theta^s(m,n-1).
            \end{equation*}
            \item $y=\theta^s(m,n)$.
            We have that $y-\rho=\theta^s(m-1,n-1)$.
            The following three identities are immediate:
            \begin{align*}
                \theta^s(m-1,n-1)&\leq \theta^s(m,n-1)\vee\theta^s(m-1,n)\\
                \theta^s(m,n-1)&\leq \theta(m+1,n-1)\vee \theta(m,n)\\
                \theta^s(m-1,n)&\leq \theta(m-1,n+1)\vee \theta(m,n)
            \end{align*}
        \end{itemize}
        In each case, $y-\rho\leq z$.
        
        Now, suppose that $z=s_1y$.
        Since $y\in D+\rho$, we obtain that $z\in \delta^2 (s_0D)$.
        By Corollary~\ref{cor: characterization of the Bruhat order}, we have that $y-\alpha_1\leq z$ since
        \begin{equation*}
            y-\alpha_1=s_2s_0s_2z\in \mathrm{V}(\CC_{z})\cap F_\id,
        \end{equation*}
        where the ``$\cap F_\id$'' part comes from thickness.
        By Corollary~\ref{cor: equivalence of being in the hexagon in the gray region}, we have that $y-\rho\leq y-\a_1$.
        Therefore, $y-\rho\leq z$.
        
        The case $z=s_2y$ is analogous.
        This proves the inclusion.
        \newline

        \noindent\textbf{Inclusion $\operatorname{U}(x)\subset [x,x+\rho]$.}
        We only need to prove that $z\in \operatorname{U}(x)$ implies $z\leq x+\rho$.
        Since  $x\in D+\rho$, it follows that $x\in\{\theta(m,n),\theta^s(m,n)\mid m,n>0\}$.
        In the proof of Lemma~\ref{lem: cardinality of upper covers}, we obtain that $\operatorname{U}(x)=\operatorname{U}_D(x)\cup \{s_0x,x_{s_1},x_{s_2}\}$, where $x_{s_1},x_{s_2}$ are as in Definition~\ref{def: estrella} and
        \begin{equation}
            \operatorname{U}_D(x)=\left\{
            \begin{array}{@{}l@{\thinspace}l}
                \{\theta^s(m,n),\theta^s(m+1,n-1),\theta^s(m-1,n+1)\} &\mbox{ if $x=\theta(m,n)$},\medskip\\
                \{\theta(m+1,n),\theta(m,n+1)\}&\mbox{ if $x=\theta^s(m,n)$}.\medskip
            \end{array}
            \right.
        \end{equation}
        
        Suppose that $z\in \operatorname{U}_D(x)$.
        Then $z\leq x+\rho$ by an analogous argument as in the first part of the proof of the \textbf{Inclusion $\operatorname{L}(y)\subset [y-\rho,y]$.}
        
        Now, suppose that $z\in\{s_0x,x_{s_1},x_{s_2}\}$.
        By Corollary~\ref{cor: characterization of the Bruhat order}, it is enough to show that $z\in \CC_{x+\rho}$.
        Since $x+\rho \in D$, we have that $\mathrm{V}(\CC_{x+\rho})=W_f\cdot(x+\rho)$.
        If $z=s_0x$, we have that
        \begin{equation*}
            s_0x=s_{\rho,-1}x=w_0(x+\rho)\in \mathrm{V}(\CC_{x+\rho}).
        \end{equation*}
        By symmetry of the proof, without loss of generality, we can assume that $z=x_{s_1}$.
        By Equation \eqref{eq: identity xw}, we have that
        \begin{equation*}
            x_{s_1}+\rho=s_1(x+\rho)\in \mathrm{V}(\CC_{x+\rho}).
        \end{equation*}
        Note that $s_0x_{s_1}=w_0s_1(x+\rho)\in \mathrm{V}(\CC_{x+\rho})$.
        Since $x_{s_1}\in F_{s_1}$ and $s_0x_{s_1}\in F_{s_1s_2}$, we have that $\cen(x_{s_1})-\cen(s_0x_{s_1})\in \mathbbm{Z}_{\geq 0}\rho$.
        Thus, $\cen(x_{s_1})$ is in the segment $\operatorname{Conv}(x_{s_1}+\rho, s_0x_{s_1})\subset \CC_{x+\rho}$.\qedhere
    \end{enumerate}
\end{proof}
\begin{rem}
    By the previous lemma, thick intervals are full.
    For the rest of the paper, we will use this fact without mention.
\end{rem}
\begin{prop}\label{prop: seg max fixed under iso}
    Let $\phi\colon [x,y]\to [x',y']$ be an isomorphism of posets.
    \begin{itemize}
        \item If $\ell(x,z_i(x))\geq 4$ for $i\in\{1,2\}$, then $\phi(\{z_1(x), z_2(x)\})=\{z_1(x'),z_2(x')\}$.
        \item If $\ell(z_i(y),y)\geq 4$ for $i\in\{1,2\}$, then $\phi(\{z_1(y), z_2(y)\})=\{z_1(y'),z_2(y')\}$.
    \end{itemize}
\end{prop}
\begin{proof}
    By Lemma~\ref{lem: iso-preserva-dih}, $\phi$ induces bijections
    \begin{gather*}
        \mathcal{D}^{l}(x)\cap [x,y] \longleftrightarrow\mathcal{D}^{l}(x')\cap [x',y']\\
        \mathcal{D}^{u}(y)\cap [x,y] \longleftrightarrow\mathcal{D}^{u}(y')\cap [x',y'].
    \end{gather*}
    The first bijection, plus Lemma~\ref{lem: betti invariante}, implies that
    \begin{equation}\label{eq: dmax is preserved}
        \phi(\mathcal{D}_{[x,y]}^{l,\mathrm{max}}) =\mathcal{D}_{[x',y']}^{l,\mathrm{max}}.
    \end{equation}
    By Lemma~\ref{lem: max dih x unicamente det}, there are distinct indexes $i,j\in\{1,2\}$ such that $z_i(x) \in \mathcal{D}_{[x,y]}^{l,\max}$ and $z_{j}(x)\in A_{x,z_i(x)}$.
    By Equation \eqref{eq: dmax is preserved}, it follows that 
    \begin{equation*}
        \phi(z_i(x)) \in \mathcal{D}_{[x',y']}^{l,\max}\mbox{ and }\phi(z_{j}(x))\in A_{x',\phi(z_i(x))}.
    \end{equation*}
    By Lemma~\ref{lem: max dih x unicamente det}, we conclude that $\phi(\{z_1(x), z_2(x)\})=\{z_1(x'),z_2(x')\}$.
    
    By a similar reasoning, but using Lemma~\ref{lem: max dih y unicamente det} instead of Lemma~\ref{lem: max dih x unicamente det}, we conclude that $\phi(\{z_1(y), z_2(y)\})=\{z_1(y'),z_2(y')\}$.
\end{proof}
\begin{rem}
    The previous proposition fails if $\ell(x,z_i(x))<4$ or $\ell(z_i(y),y)<4$ for some $i\in\{1,2\}$, which motivates the thickness hypothesis in the main theorem of this paper.
\end{rem}

%% file: Sections/Proof.tex
\section{Congruence of polygons and proof of the main theorem}\label{subsec: congruence and proof of main}
In this section, we consider  $x,x',y,y'\in D$ such that the intervals $[x,y]$ and $[x',y']$ are thick.
\begin{prop}\label{prop: iso implies congruent polygons}
    If $[x,y]$ and $[x',y']$ are isomorphic as posets, then there is a weight $\mu$, such that $\Pgn_{x',y'}+\mu$ is equal to $\Pgn_{x,y}$ or to $\sigma\Pgn_{x,y}$.
    In particular, $\Pgn_{x,y}$ and $\Pgn_{x',y'}$ are congruent.
\end{prop}
\subsection{Congruence of polygons}\label{subsec: congruence}
Before embarking on the proof, we need some technical lemmas.
\begin{lem}\label{lem: x+rho es invariante}
    If $\phi\colon[x,y]\to[x',y']$ is a poset isomorphism, then $\phi(x+\rho)=x'+\rho.$
\end{lem}
\begin{proof}
    Let $Z(x)$ be the set of elements $z\in [x,y]$ such that 
    \begin{enumerate}
        \item $\ell(x,z)=4$.\label{item: condition 1}
        \item $\operatorname{U}(x)\subset [x,z]$.\label{item: condition 2}
        \item $B_2(x,z):=|\{[a,b]\subset [x,z] \mid \mbox{$[a,b]$ is a $2$-crown}\}|=2$.\label{item: condition 3}
    \end{enumerate}
    We define $Z(x')\subset [x',y']$ in a similar manner.
    Let $Z_{12}(x)$ be the set of elements that satisfy conditions (\ref{item: condition 1}) and (\ref{item: condition 2}).
    In the proof of Lemma~\ref{lem: cardinality of upper covers}, the elements of~$\operatorname{U}(x)$ are given explicitly, depending on whether $x$ is an element of the form $\theta(m,n)$ or $\theta^s(m,n)$, where \( m, n \) are positive integers.
    The cases ($m=0,n>0$ and $m=0,n=0$) are not explicitly written there, but they are easy to compute.
    From this small set of elements, one can check which ones have the correct length.
    Thus, we have: 
    \begin{itemize}
        \item If $x=\theta(m,n)$, 
        \begin{equation*}
            Z_{12}(x) = \{\theta(m+1,n+1),\delta(s_0\theta^s(m,n+1)),\delta^2((s_0\theta^s(m+1,n))\}.
        \end{equation*}
        \item If $x=\theta^s(m,n)$,
        \begin{equation*}
            Z_{12}(x) = \{\theta^s(m+1,n+1),\delta(s_0\theta(m+1,n+1)),\delta^2(s_0\theta(m+1,n+1))\}.
        \end{equation*}
    \end{itemize}
    In~\cite{LP23}, explicit formulas are given for the Kazhdan--Lusztig polynomials $P_{w,w'}(q)$, for any pair of elements $w,w'\in W.$ We can apply this to  obtain the following formulas: 
    \begin{align}
        P_{\theta^s(m,n),\theta^s(m+1,n+1)}&=1+2q,\\
        P_{\theta^s(m,n),\delta(s_0\theta(m+1,n+1))}&=1+q,\\
        P_{\theta^s(m,n),\delta^2(s_0\theta(m+1,n+1))}&=1+q,
    \end{align}
    and for the specific cases.
    \begin{align} 
        \mbox{If $m>0,n>0$:}& & \mbox{If $m>0,n=0$:}  \\
        P_{\theta(m,n),\theta(m+1,n+1)}&=1+q,& P_{\theta(m,0),\theta(m+1,1)}&=1+q,\\
        P_{\theta(m,n),\delta(s_0\theta^s(m,n+1))}&=1+2q,& 
        P_{\theta(m,n),\delta(s_0\theta^s(m,1))}&=1+2q,\\
        P_{\theta(m,n),\delta^2(s_0\theta^s(m+1,n))}&=1+2q.&
        P_{\theta(m,0),\delta^2(s_0\theta^s(m+1,0))}&=1+q.
    \end{align}
    Similarly, the remaining cases ($m=0,n=0$ or $m=0, n>0$) can also be addressed, but we will omit the formulas.
    By~\cite[Exercise 5.8]{BB05} we have
    \begin{equation}\label{crown}
        P_{w,w'}(q)=1+\left(|L(w)|+\frac{B_2(w,w')}{2}-4\right)q.
    \end{equation}
    From Equation \eqref{crown}, Lemma~\ref{lem: cardinality of lower covers}, and the explicit Kazhdan--Lusztig polynomials calculated above, we obtain
    \begin{align}\label{eq: conjunto Z0}
        Z(x)&=
        \begin{cases}
            \{\theta(m+1,n+1)\}& \mbox{if $x=\theta(m,n)$},\\
            \{\theta^s(m+1,n+1)\} & \mbox{if $x=\theta^s(m,n)$}.
        \end{cases}\\
        &= \{x+\rho\}.
    \end{align}
    By the same argument, $Z(x')=\{x'+\rho\}$.
    Since the defining properties of~$Z(x)$ are preserved by any poset isomorphism, we have $\phi(Z(x))\subset Z(x')$, so $\phi(x+\rho)=x'+\rho$.
\end{proof}
For the rest of the section, we use the following shortcuts for $1\leq i\leq 2$:
\begin{align*}
    z_i(x)&\coloneqq z^{x,y}_i(x) &v_i(x)&\coloneqq v^{x,y}_i(x)\\
    z_i(y)&\coloneqq z^{x,y}_i(y) &v_i(y)&\coloneqq v^{x,y}_i(y).
\end{align*}
The same notations apply if we replace $x,y$ with $x',y'$.
\begin{lem}\label{lem: z_i(x)<z_i(y) implica hex}
    Suppose that $z_{i_0}(x)\leq z_{i_0}(y) $, for some $i_0\in\{1,2\}$.
    Then $\Pgn_{x,y}$ is a hexagon.
\end{lem}
\begin{proof}
    Without loss of generality, we assume $z_1(x)\leq z_1(y)$.
    We need to prove that $v_2(y)\neq v_1(x)$ and $v_1(y)\neq v_2(x)$.
    
    Suppose that $v_2(y)=v_1(x)$.
    According to Remark~\ref{rem: geometry-of-vertex-of-Pgn},
    $v_1(x)$ is either the center of an alcove or a vertex of an alcove.
    In the first case, we have $v_1(x)=\cen(z_1(x))$, while in the second case, it is straightforward to see that
    $v_1(x)=\frac{1}{3}\alpha_1+\cen(z_1(x))$.
    In any case, $(\varpi_1,v_2(y)-z_1(x))\leq 1/3$.
    As $v_2(y)\in y+\mathbbm{R}\alpha_2$, we have $(\varpi_1, y)=(\varpi_1,v_2(y))$.
    By thickness of~$[x,y]$, there is $t\geq 2$ such that $z_1(y)=y-t\alpha_1$.
    We have,
    \begin{equation*}
        (\varpi_1,z_1(y)-z_1(x))=(\varpi_1,v_2(y)-t\alpha_1-z_1(x))\leq \frac{1}{3}-2<0.
    \end{equation*}
    By Corollary~\ref{cor: equivalence of being in the hexagon in the gray region} we get $z_1(x)\not\leq z_1(y)$, a contradiction.
    
    Suppose \( v_1(y) = v_2(x) \).
    The reasoning here is similar, so we omit the details.
    We have either \( v_2(x) = \cen(z_1(y)) - \frac{1}{3}\alpha_1 \) or \( v_2(x) = \cen(z_1(y)) \).
    In either case, \linebreak \( (\varpi_1, z_1(y) - v_2(x)) \leq \frac{1}{3} \) and \( (\varpi_1, v_2(x)) = (\varpi_1, x) \).
    Again, by thickness of~$[x,y]$,  we deduce \( (\varpi_1, x- z_1(x) ) \leq -2 \).
    Combining these inequalities yields \( (\varpi_1, z_1(y) - z_1(x)) \leq -\frac{5}{3} < 0 \), which, once again, leads to a contradiction.
\end{proof}
In the following three lemmas, we explore the subintervals $[x+\rho, y]\subset [x,y]$ and $[x'+\rho, y']\subset [x',y']$.
For convenience, we replace the superscripts $(-)^{x+\rho,y}$ and $(-)^{x'+\rho,y'}$ by $(-)^{\rho}$.
That is, for $1\leq i\leq 2$:
\begin{align*}
    z^{\rho}_i(x+\rho)&\coloneqq z^{x+\rho,y}_i(x+\rho) &v^{\rho}_i(x+\rho)&\coloneqq v^{x+\rho,y}_i(x+\rho)\\
    z^{\rho}_i(y)&\coloneqq z^{x+\rho,y}_i(y) &v^{\rho}_i(y)&\coloneqq v^{x+\rho,y}_i(y).
\end{align*}
The same notations apply if we replace $x,y$ with $x',y'$.
\begin{lem}\label{lem: z_i(x)<z_i(x+rho)}
    Let $\{i,j\}=\{1,2\}$, then $z_i(x)<z_i^{\rho}(x+\rho)$ and $z_i(x)\not\le z_{j}^{\rho}(x+\rho)$.
\end{lem}
\begin{proof}
    It is not hard to deduce the first part using Lemma~\ref{lem: local bruhat order}.
    
    We now prove the second part.
    By thickness we have that $z_i(x)=x+t\alpha_i$ for $t\geq 2$ leads to the inequality $(\varpi_i, z_i(x)-x)\geq 2$.
    Note that $(\varpi_i, x+\rho)=(\varpi_i, z_j^\rho(x+\rho))$ and $(\varpi_i, x+\rho)=1+(\varpi_i, x)$.
    So 
    \begin{equation*}
        (\varpi_i,z_i(x))\geq (\varpi_i,x)+2= (\varpi_i, x+\rho)+1>(\varpi_i, z_j^\rho(x+\rho)).
    \end{equation*} 
    By Corollary~\ref{cor: equivalence of being in the hexagon in the gray region}, we have $z_i(x)\not\le z_{j}^{\rho}(x+\rho)$.
\end{proof}
\begin{lem}\label{lem: z_i(y) in [x+rho,y]}
    Suppose that $z_{i_0}(x)\leq z_{i_0}(y) $, for some $i_0\in\{1,2\}$.
    Then $z_i^{\rho}(y)=z_i(y)$ for $1\leq i \leq 2$.
\end{lem}
\begin{proof}
    Without loss of generality, we assume $z_1(x)\leq z_1(y)$.
    By Lemma~\ref{lem: z_i(x)<z_i(y) implica hex}, $\Pgn_{x,y}$ is a hexagon.
    Since $\Pgn_{x+\rho,y}\subset \Pgn_{x,y}$, we have $z_i^{\rho}(y)\in \Sgm(y,z_i(y))$.
    By Lemma~\ref{lem: local bruhat order}, we have $z_i^{\rho}(y)\in [z_i(y),y]$ for $1\leq i \leq 2$.
    \newline
    
    \noindent\textbf{Part A.} We prove that $z_1^{\rho}(y)= z_1(y)$.
    Suppose by contradiction that $z_1^{\rho}(y)\neq z_1(y)$.
    Then $z_1(y)<z_1^{\rho}(y)$, so $v_1^{\rho}(y)=v_2^{\rho}(x+\rho)\notin \partial F_\id$.
    By Remark~\ref{rem: geometry-of-vertex-of-Pgn}(\ref{rem: vertex of Pgn}), we have that \( v_2^\rho(x + \rho) \) is either the center of \( z_2^\rho(x+\rho) \) or a vertex of \( z_2^\rho(x+\rho) \).
    \begin{enumerate}
        \item In the first case, we have that $z_1^{\rho}(y) = z_2^{\rho}(x+\rho)$.
        This implies $z_1(x)\leq z_1(y)<  z_2^{\rho}(x+\rho)$, which is not possible by Lemma~\ref{lem: z_i(x)<z_i(x+rho)}.
        \item The remaining case corresponds to the second case of Proposition~\ref{prop: length de adjacencia implica posicion}, so we have: 
        \begin{itemize}
            \item $\ell(z_2^{\rho}(x+\rho), z_1^{\rho}(y))=1$ and $z_2^{\rho}(x+\rho)$ and $z_1^{\rho}(y)$ share an horizontal edge.
            \item $z_2^{\rho}(x+\rho)+\frac{1}{3}\alpha_2=v_1^{\rho}(y)$, so $(\varpi_1, z_2^{\rho}(x+\rho)-v_1^{\rho}(y))=0$.
        \end{itemize}
        We also have $v_1^{\rho}(y)-t\alpha_1=\cen(z_1(y))$, where $t>0$.
        Hence $(\varpi_1, z_1(y))\leq (\varpi_1,v_1^{\rho}(y))=(\varpi_1,z_2^{\rho}(x+\rho))$ by the second bullet above.
        Since $\Pgn_{x,y}$ is a hexagon, there is $k\geq 5$ such that $z_1(x)=\mathrm{\mathbf{x}}_k$.
        By Corollary~\ref{cor: equivalence of being in the hexagon in the gray region}, $0\leq (\varpi_1, z_1(y)-\mathrm{\mathbf{x}}_k)\leq (\varpi_1, z_2^{\rho}(x+\rho)-\mathrm{\mathbf{x}}_k)$.
        It is not hard to check that if $z\in F_{\id}$ and $0\leq (\varpi_1, z-\mathrm{\mathbf{x}}_k)$ then $0\leq (\varpi_2, z-\mathrm{\mathbf{x}}_k)$.
        Thus we have that  $0\leq (\varpi_2, z_2^{\rho}(x+\rho)-\mathrm{\mathbf{x}}_k)$ and by Corollary~\ref{cor: equivalence of being in the hexagon in the gray region} we have that $z_1(x)\in \CC_{z_2^{\rho}(x+\rho)}$.
        By Corollary~\ref{cor: characterization of the Bruhat order}, $z_1(x)\leq z_2^{\rho}(x+\rho)$ contradicting again Lemma~\ref{lem: z_i(x)<z_i(x+rho)}.
    \end{enumerate}
  
    \noindent \textbf{Part B.} We prove that $z_2^{\rho}(y)=z_2(y)$.
    We base the proof on the following geometric observation: 
    Let $u_1(x)$ and $u_1(x+\rho)$ be the exterior vertices of~$\mathcal{St}(x)$ and $\mathcal{St}(x+\rho)$, which are on the left boundary of~$F_\id$.
    It is easy to see that $u_1(x)+3\varpi_1=u_1(x+\rho)$.
    Since $[x,y]$ is a hexagon interval, $u_1(x)=v_1(x)$, so \begin{equation}\label{u1}
        v_1(x)+3\varpi_1=u_1(x+\rho).
    \end{equation}    
    \begin{claim}\label{claim: distance in hexagon}
        $\operatorname{dist}(v_2(y),v_1(x))\geq \frac{5}{2}\sqrt{3}$.
    \end{claim}
    \begin{proof}
        Since $\Pgn_{x,y}$ is a hexagon, there is $t>0$ such that $v_2(y)-v_1(x)=t\varpi_1$.
        From $v_2(y)\in y+\mathbbm{R}\alpha_2$, we get $(\varpi_1, y)=(\varpi_1,v_2(y))$.
        By thickness of~$[x,y]$, we have $(\varpi_1, z_1(y))\leq (\varpi_1, y)-2$.
        At the beginning of the proof, we assumed $z_1(x)\leq z_1(y)$, so by Corollary~\ref{cor: equivalence of being in the hexagon in the gray region}, $(\varpi_1, z_1(y)-z_1(x))\geq 0$.
        Then
        \begin{equation*}
            (\varpi_1, z_1(x))\leq (\varpi_1, v_2(y))-2=(\varpi_1, v_1(x))+t(\varpi_1,\varpi_1)-2.
        \end{equation*}
        Since $v_1(x)=z_1(x)+\frac{1}{3}\a_1$, we have $(\varpi_1, z_1(x)-v_1(x))=-1/3$.
        Then
        \begin{equation*}
            \frac{5}{3}=2-\frac{1}{3}\leq (\varpi_1,\varpi_1)t=\frac{2}{3}t,
        \end{equation*}
        or equivalently, $t\geq 5/2$.
    \end{proof}
    We have two scenarios:
    \begin{itemize}
        \item $\operatorname{dist}(v_2(y),v_1(x))=\frac{5}{2}\sqrt{3}$.
        This implies that $v_1(x)+\frac{5}{2} \varpi_1=v_2(y)$.
        In this case $v_2(y)$ must be the midpoint of the left edge of~$z_2(y)$ which must be down-oriented, and since $v_1(x)+3\varpi_1=u_1(x+\rho)$, we have
        $u_1(x+\rho)=v_2(y)+\frac{1}{2}\varpi_1$, so $u_1(x+\rho)$ is the top-left vertex of~$z_2(y)$.
        This implies that 
        \begin{equation*}
            \Sgm(u_1(x+\rho), x+\rho)\cap \Sgm(y, v_2(y))=\{\cen(z_2(y))\}.
        \end{equation*}
        In particular, $\cen(z_2(y))\in \mathrm{V}(\Pgn_{x+\rho,y})$, so
        $v_1^{\rho}(x+\rho)=v_2^{\rho}(y)=\cen(z_2(y))$.
        The second equality implies that $z_2^{\rho}(y)=z_2(y)$.
        \item $\operatorname{dist}(v_2(y),v_1(x))>\frac{5}{2}\sqrt{3}$.
        As $v_2(y)-v_1(x)\in \frac{1}{2}\mathbbm{Z}\varpi_1,$ we have that 
        $v_1(x)+a\varpi_1=v_2(y)$, where $a\geq 3$.
        Furthermore, since $v_1(x)+3\varpi_1= u_1(x+\rho)$, it follows that $u_1(x+\rho)=v_1^{\rho}(x+\rho)\in \partial F_\id$.
        Therefore $v_2^{\rho}(y)\in \partial F_\id$, so  $v_2(y)=v_2^{\rho}(y)$, and we conclude that $z_2(y)=z_2^{\rho}(y)$.\qedhere       
    \end{itemize}
\end{proof}
\begin{lem}\label{lem: v_i(x+rho) =u_i(x+rho)}
    Assume $z_{i_0}(x) \leq z_{i_0}(y)$ for some $i_0\in\{1,2\}.$
    If $1\leq i\leq 2$, then we have that $v_{i}(x)+3\varpi_{i}=u_{i}(x+\rho)$ 
    and  $u_i(x+\rho)$ is a top vertex of~$z_i(x+\rho)$.
    In particular, we have that  
    $z_{i}(x)+3\varpi_{i}=z_{i}^\rho(x+\rho)$.
\end{lem}
\begin{proof} 
    By Lemma~\ref{lem: z_i(x)<z_i(y) implica hex}, we know that $\Pgn_{x,y}$ is an hexagon.
    Consequently, we have that  $v_i(x)=u_i(x)$.
    The first equation follows from the fact that   
    $u_i(x)+3\varpi_i=u_i(x)+\rho+\a_i=u_i(x+\rho)$.
    We now proceed to prove that $u_i(x+\rho)$ is a top vertex of~$z_i^\rho(x+\rho)$.
    Without loss of generality, suppose that $i_0=1$.
    By Claim~\ref{claim: distance in hexagon}, we have that $\operatorname{dist}(v_2(y),v_1(x))\geq\frac{5}{2}\sqrt{3}.$
    By thickness of~$[x,y]$, we have that  
    \begin{equation}\label{eq: z_i(*)=*+ta_i}
        z_i(x)=x+t\a_i,
    \end{equation}
    for some $t\geq 2$.
    By Equation \eqref{eq: z_i(*)=*+ta_i} and since $z_1(x)\leq z_1(y)$ and $v_1(y)\in v_2(x)+\mathbbm{R}_{\geq0}\varpi_2$, it is not hard to check that $\operatorname{dist}(v_1(y),v_2(x))\geq\frac{5}{2}\sqrt{3}.$ 
    By following the argument in the bullets of Part $B$ in the proof of Lemma~\ref{lem: z_i(y) in [x+rho,y]}, we have that $u_i(x+\rho)$ is a top vertex of~$z_i^\rho(x+\rho)$.
    The equation $z_{i}(x)+3\varpi_{i}=z_{i}^\rho(x+\rho)$, follows from the fact that $u_i(x+\rho)$ is a top vertex of~$z_i(x+\rho)$.
    The reason for this is that  since  $u_i(x)$ is a top vertex of~$z_i(x)$, and  $u_i(x)+3\varpi_i=u_i(x+\rho)$, we deduce that $z_i(x)+3\varpi_i=z_i^\rho(x+\rho)$.
\end{proof}
\begin{lem}\label{lem: [x+rho,y] is thick}
    Suppose that $z_{i_0}(x)\leq z_{i_0}(y) $, for some $i_0\in\{1,2\}$.
    Then $[x+\rho,y]$ is thick.
\end{lem}
\begin{proof} 
    Let $1\leq i\neq j\leq 2$.
    By Lemma~\ref{lem: z_i(y) in [x+rho,y]}, we have that  $z_i^{\rho}(y)=z_i(y)$.
    As $[x,y]$ is thick, it suffices to show that $\ell(x+\rho,z_i^{\rho}(x+\rho))\geq 4$.
    By Proposition~\ref{prop: length of all sides}, as $\gamma(z_i(x+\rho))>0,$ this inequality is implied by the following inequality.
    \begin{equation}
        \operatorname{dist}(v_i^{\rho}(x+\rho),x+\rho)\geq 6+\eta(x+\rho).
    \end{equation}
    Since $[x,y]$ is thick, it follows that $\ell(x,z_i(x))\geq 4.$
    By Lemma~\ref{lem: z_i(x)<z_i(y) implica hex}, we have that $\Pgn_{x,y}$ is a hexagon.
    Therefore,
    $v_i(x)\in H_{\alpha_{j},-1}$, which means that $v_i(x)$ is a top vertex of~$z_i(x)$, in particular $v_i(x)\neq\cen(z_i(x))$.
    We obtain the following inequality by applying Proposition~\ref{prop: length of all sides}.
    \begin{equation}\label{eq: dis(x,v_i) in proof}
        \operatorname{dist}(v_i(x),x)\geq 6+\eta(x).
    \end{equation}
    Recall that  $u_i(\ast)$ be the exterior vertex of~$\mathcal{St}(\ast)$ in~$H_{\alpha_{j},-1}$, where $*\in\{x,x+\rho\}$.
    Since $v_i(x)\in H_{\alpha_{j},-1}$, it follows that $v_i(x)=u_i(x)$.
    Consequently, $u_i(x)+3\varpi_i$ is a top vertex of~$z_i(x)+3\varpi_i$.
    By Lemma~\ref{lem: v_i(x+rho) =u_i(x+rho)},
    we have that $u_i(x)+3\varpi_i=u_i(x+\rho)$ and $z_i(x)+3\varpi_i=z_i^\rho(x+\rho)$.
    Thus, $u_i(x+\rho)$ is a top vertex of~$z_i^\rho(x+\rho)$, which implies that $u_i(x+\rho)=z_i^\rho(x+\rho)+\frac{1}{3}\a_i$.
    Since $\rho+\alpha_{i}=3\varpi_{i}$, it follows that $v_{i}(x)+\rho+\alpha_{i}=u_{i}(x+\rho)$ which implies that $v_{i}(x)+\rho+\frac{2}{3}\a_i=z_i^\rho(x+\rho)$.
    Therefore, we conclude that 
    \begin{equation*}
        \Sgm(v_i(x)+\rho,x+\rho)\subset \Sgm(z_i^\rho(x+\rho),x+\rho)\subset \Sgm(u_i(x+\rho),x+\rho).
    \end{equation*}
    One has that $\operatorname{dist}(z_i^{\rho}(x+\rho), v_i(x)+\rho)= \operatorname{dist}(0,\frac{2}{3}\alpha_i)= 2$.
    By Lemma~\ref{lem: v_i(x+rho) =u_i(x+rho)}, we can deduce that
    $v_i^\rho(x+\rho)$ is either equal to $z_i^\rho(x+\rho)$ or to $u_i(x+\rho)$.
    Therefore, the result follows from the following inequalities: 
    \begin{align}
        \operatorname{dist}(v_i^{\rho}(x+\rho),x+\rho)&=\operatorname{dist}(v_i^{\rho}(x+\rho), v_i(x)+\rho)+ \operatorname{dist}(v_i(x)+\rho,x+\rho)\\
        &\geq 2+\operatorname{dist}(v_i(x),x)\\
        &\geq 2+(6+\eta(x))\\
        &=8+\eta(x+\rho).
    \end{align}
    Where the inequality in the third line follows from Equation \eqref{eq: dis(x,v_i) in proof} and the final equality is a consequence of the fact that both $x$ and $(x+\rho)$ belong to the same $X\Theta$-partition.
\end{proof}
\begin{lem}\label{lem: 3bullets}
    Let $\{i,j\}=\{1,2\}$.
    Let $\phi\colon [x,y]\to [x',y']$ be a poset isomorphism.
    We have:
    \begin{align*}
        &\begin{array}{l}
            \bullet \quad [x,y] \mbox{ is thick} \\
            \bullet \quad \phi(z_i(x))=z_i(x') \\
            \bullet \quad \phi(z_i(y))=z_j(y') \\
            \bullet \quad z_i(x)\leq z_i(y)       
        \end{array}
        \implies
        \begin{array}{l}
            \bullet \quad [x+\rho,y] \mbox{ is thick} \\
            \bullet \quad \phi(z_i^\rho(x+\rho))=z_i^\rho(x'+\rho) \\
            \bullet \quad  \phi(z_i^\rho(y))=z_j^\rho(y') \\
            \bullet \quad z_i^\rho(x+\rho)\leq z_i^\rho(y).
        \end{array}
    \end{align*}
\end{lem}
\begin{proof}
    Without loss of generality, let us suppose that $i=1$ and $j=2$.
    From Lemma~\ref{lem: x+rho es invariante}, we know that $\phi(x+\rho)=x'+\rho$.
    In particular, we have that $\phi([x+\rho,y])=[x'+\rho,y']$.
    \begin{itemize}
        \item By Lemma~\ref{lem: [x+rho,y] is thick} and the fourth bullet, the interval $[x+\rho,y]$ is thick.
        \item By Lemma~\ref{lem: z_i(x)<z_i(x+rho)}, we know that
        \begin{equation}\label{eq: z_i(x)<z_i(x+rho) in bijec dem}
            z_1(x)\leq z_1^\rho(x+\rho)\quad\text{and} \quad z_1(x)\not\leq z_2^\rho(x+\rho)
        \end{equation}
        Applying $\phi$ in the inequalities \eqref{eq: z_i(x)<z_i(x+rho) in bijec dem},
        we obtain 
        \begin{equation}\label{eq: dem phi(z(x'+rho))}
            z_1(x')\leq \phi(z_1^{\rho}(x+\rho))\quad \text{and}\quad z_1(x')\not\leq \phi(z_2^{\rho}(x+\rho))
        \end{equation}
        Since $[x+\rho,y]$ is thick, Proposition~\ref{prop: seg max fixed under iso} implies that 
        \begin{equation*}
            \{\phi(z_1^{\rho}(x+\rho)),\phi(z_2^{\rho}(x+\rho))\}=\{z_1^{\rho}(x'+\rho),z_2^{\rho}(x'+\rho)\}.
        \end{equation*}
        Thus, applying Lemma~\ref{lem: z_i(x)<z_i(x+rho)} to $[x',y']$, we obtain that $z_1(x')\nleq z_2^{\rho}(x'+\rho)$, so one can not have $\phi(z_1^{\rho}(x+\rho))=z_2^{\rho}(x'+\rho)$, because  the first inequality of \eqref{eq: dem phi(z(x'+rho))} would yield a contradiction.
        So we conclude that $\phi(z_1^\rho(x+\rho))=z_1^\rho(x'+\rho).$
        \item By Corollary~\ref{cor: z1(x) es menor a z2(y)} we have that $z_1(x)\leq z_2(y)$, which implies $z_1(x')\leq z_1(y')$.
        By the fourth bullet and Lemma~\ref{lem: z_i(y) in [x+rho,y]}, we have that $z_1^\rho(y')=z_1(y')$.
        By Lemma~\ref{lem: z_i(y) in [x+rho,y]}, we know that $z_1^{\rho}(y)=z_1(y)$ and $z_2^{\rho}(y)=z_2(y)$, so $\phi(z_2^{\rho}(y))=\phi(z_2(y))=z_1(y')=z_1^\rho(y')$.
        \item Suppose now, that $z_1^\rho(x'+\rho)\not\leq z_1(y')$.
        Applying $\phi^{-1}$, we obtain that $z_1^\rho(x+\rho)\not\leq z_2(y)$, which contradicts Corollary~\ref{cor: z1(x) es menor a z2(y)}.\qedhere
    \end{itemize}
\end{proof}
\begin{cor}\label{cor: contradiccion}
    Let $\{i,j\}=\{1,2\}$, and let $\phi \colon [x, y] \to [x', y']$ be a poset isomorphism.
    The four conditions listed on the left side of Lemma~\ref{lem: 3bullets} cannot all hold simultaneously.
\end{cor}
\begin{proof}
    If these conditions were satisfied simultaneously, then by Lemma~\ref{lem: 3bullets} and induction on \( n \), the four conditions would also hold when \( x \) is replaced by \( x + n\rho \), \( x' \) by \( x' + n\rho \), and \( \phi \) by its restriction  
    \[
    \phi\vert_{[x+n\rho, y]}\colon [x+n\rho, y] \to [x'+n\rho, y'].
    \]  
    However, this leads to a contradiction because for sufficiently large \( n \), \( x + n\rho \) will no longer satisfy \( x + n\rho \leq y \), so $[x+n\rho, y]$ will not be thick.
\end{proof}
\begin{lem}\label{lem: z1z1}
    Let $1\leq i\leq 2$.
    Let $\phi\colon [x,y]\to [x',y']$ be a poset isomorphism.
    If $\phi(z_i(x))=z_i(x')$ then $\phi(z_i(y))=z_i(y')$.
\end{lem}
\begin{proof}
    Without loss of generality, let us suppose that $i=1$ and $j=2$.
    Let us suppose that $\phi(z_1(y))=z_2(y')$.
    This implies that $z_1(x)\leq z_1(y)$, because if  $z_1(x)\not\leq z_1(y)$, we have that $z_1(x')\not\leq z_2(y')$ which is a contradiction by Corollary~\ref{cor: z1(x) es menor a z2(y)}.
    By Corollary~\ref{cor: contradiccion}, this is not possible, so by Proposition~\ref{prop: seg max fixed under iso}, we conclude that $\phi(z_1(y))=z_1(y')$.
\end{proof}
\begin{lem}\label{lem: iso preserves adjacency}
    Let \(\phi \colon [x, y] \to [x', y']\) be a poset isomorphism.
    There exists a bijection \(f \colon \mathrm{V}(\Pgn_{x,y}) \to \mathrm{V}(\Pgn_{x',y'})\) such that \(f(x) = x'\), \(f(y) = y'\), and \(f\) preserves adjacency relations.
    Moreover, for \(i \in \{1, 2\}\), we have \(f(v_i(x)) \in \phi(z_i(x))\) and \(f(v_i(y)) \in \phi(z_i(y))\).
\end{lem}
\begin{proof}
    By Proposition~\ref{prop: seg max fixed under iso}, either $\phi(z_1(x))$ is equal to $z_1(x')$ or to $z_2(x')$.
    Without loss of generality, we consider the first case.
    Then $\phi(z_2(x))=z_2(x')$.
    We define $f\colon\mathrm{V}(\Pgn_{x,y})\to \mathrm{V}(\Pgn_{x',y'})$ by the formulas\footnote{If $\phi(z_1(x))=z_2(x')$ this formula reads $f(v_i(\ast))=v_j(\ast')$ with $\{i,j\}=\{1,2\}$.} $f(x)=x', f(y)=y'$, and $f(v_i(\ast))=v_i(\ast')$.
    Where $\ast\in \{x,y\}$ and $i\in\{1,2\}$.
    The fact that $f$ preserves the adjacency relations is obvious from the definition of~$f$.
    That $f(v_i(x))\in \phi(z_i(x))$ is trivial.
    Lemma~\ref{lem: z1z1} implies that $\phi(z_1(y))=z_1(y')$ and this implies by Proposition~\ref{prop: seg max fixed under iso} that $\phi(z_2(y))=z_2(y')$.
\end{proof}
\begin{defn}
    Let \(\phi \colon [x, y] \to [x', y']\) be a poset isomorphism.
    We denote by $f_{\phi}\colon \mathrm{V}(\Pgn_{x,y}) \to \mathrm{V}(\Pgn_{x',y'})$ the bijection from Lemma~\ref{lem: iso preserves adjacency}.
\end{defn}
\begin{lem}\label{lem: relative positions}
    Let $\phi\colon [x,y] \to [x',y']$ be a poset isomorphism.
    If $v_i(\ast) \in A(z_i(\ast))$, then $f_{\phi}(v_i(\ast)) \in A(\phi(z_i(\ast)))$ where $A\in \{C,\mathrm{V}, \mathrm{ME}\}$, $\ast\in\{x,y\}$, and $i\in\{1,2\}$.
\end{lem}
\begin{proof}
    Let $\{i,j\}=\{1,2\}$.
    By Lemma~\ref{lem: iso preserves adjacency}, we know that:
    \begin{enumerate}[label=(\alph*)]
        \item \label{item adj a}$f_{\phi}(v_i(x))\in \phi(z_i(x))$ and $f_{\phi}(v_i(y))\in \phi(z_i(y))$,
        \item \label{item adj b}$f_{\phi}(v_i(x))$ and $f_{\phi}(v_{j}(y))$ are adjacent or equal.
    \end{enumerate} 
    By~\ref{item adj a}, there are $A,B,A',B'\in \{C,\mathrm{V}, \mathrm{ME}\}$ such that \(v_i(x)\in A(z_i(x))\), \(v_j(y)\in B(z_j(y))\), \(f_{\phi}(v_i(x))\in A'(\phi(z_i(x)))\) and \(f_{\phi}(v_j(y))\in B'(\phi(z_j(y)))\).
    By Proposition~\ref{Prop: vert are det by length}, $A$ and $B$ are determined by \(\ell(z_i(x), z_j(y))\).
    By~\ref{item adj b}, $A'$ and $B'$ are determined by $\ell(\phi(z_i(x)), \phi(z_j(y)))$.
    By Lemma~\ref{lem: betti invariante}, it follows that 
    \[\ell(z_i(x), z_j(y)) = \ell(\phi(z_i(x)), \phi(z_j(y)))\]
    so $A=A'$ and $B=B'$.
\end{proof}
\begin{lem}\label{lem: isomorphism preserves all distances}
    Let $\phi\colon [x,y] \to [x',y']$ be a poset isomorphism., we have for $\{i,j\}=\{1,2\}$:
    \begin{align}
        d(y,v_j(y))&= d(y',f_{\phi}(v_j(y))),\label{eq1: lem igual dis}\\
        d(v_j(y), v_{i}(x))&=d(f_{\phi}(v_j(y)),f_{\phi}(v_i(x))),\label{eq2: lem igual dis} \\
        d(v_i(x),x)&=d(f_{\phi}(v_i(x)),x')\label{eq3: lem igual dis}.
    \end{align}
\end{lem}
\begin{proof}
    This is immediate from Corollary~\ref{cor: iso intvls same theta partition}, Proposition~\ref{prop: length of all sides} and Lemma~\ref{lem: relative positions}.
\end{proof}    
We are ready to prove the combinatorial nature of~$\Pgn_{x,y}$.
\begin{proof}[Proof of Proposition~\ref{prop: iso implies congruent polygons}]
    From Proposition~\ref{prop: seg max fixed under iso}, we have two cases:
    \begin{itemize}
        \item $\phi(z_1(x)) = z_1(x')$ and $\phi(z_2(x)) = z_2(x')$.
        By thickness of the interval $[x', y']$, we have $z_1(x') \cap z_2(x') = \emptyset$, so $f_\phi(v_1(x)) = v_1(x')$ and $f_\phi(v_2(x)) = v_2(x')$.
        Since $f_\phi$ is a bijection preserving adjacency relations, and $v_i(x)$ is either adjacent or equal to $v_j(y)$ for $i \neq j$, it follows that $f_\phi(v_1(y)) = v_1(y')$ and $f_\phi(v_2(y)) = v_2(y')$.

        By Corollary~\ref{cor: iso intvls same theta partition}, both $x$ and $x'$ belong either to $\Theta$ or to $\Theta^s$.
        This implies that there exists a weight~$\mu$ such that $x' = x - \mu$.
        Note that $\mathrm{V}(\Pgn_{x,y} - \mu) = \mathrm{V}(\Pgn_{x,y}) - \mu$.
        Let $i, j \in \{1, 2\}$.
        By Remark~\ref{rem: distancias} and Lemma~\ref{lem: isomorphism preserves all distances}, and using the fact that $\cen(x') = \cen(x) - \mu$, we obtain the following:
        \begin{align}
            v_i(x') &= \cen(x') + d(v_i(x'), x') \cdot \alpha_i / \|\alpha_i\| \\
                    &= \cen(x') + d(f_\phi(v_i(x)), x') \cdot \alpha_i / \|\alpha_i\| \\
                    &= \cen(x) - \mu + d(v_i(x), x) \cdot \alpha_i / \|\alpha_i\| \\
                    &= v_i(x) - \mu.
        \end{align}
        Similarly, we obtain $v_j(y') = v_j(y) - \mu$ and $\cen(y') = \cen(y) - \mu$.
        This proves the equality of polygons: $\Pgn_{x,y} - \mu = \Pgn_{x', y'}$.

        \item $\phi(z_1(x)) = z_2(x')$ and $\phi(z_2(x)) = z_1(x')$.
        Recall that $\sigma$ is an automorphism of the Bruhat order.
        In this case, we compose $\phi$ with the map $z \mapsto \sigma(z)$ to reduce to the previous case.
        We now explain this in more detail.

        Clearly, $f_\phi(v_1(x)) = v_2(x')$ and $f_\phi(v_2(x)) = v_1(x')$.
        By the adjacency statement in Lemma~\ref{lem: iso preserves adjacency}, it follows that $f_\phi(v_1(y)) = v_2(y')$ and $f_\phi(v_2(y)) = v_1(y')$.
        The equality of polygons
        \begin{equation*}
            \sigma \Pgn_{x,y} - \mu = \Pgn_{x', y'}
        \end{equation*}
        follows by an argument entirely analogous to the previous case, using the weight $\mu \coloneqq \cen(\sigma x) - \cen(x')$.\qedhere
    \end{itemize}    
\end{proof}
\subsection{Proof of Theorem~\ref{thm-main intro}}\label{subsec: main result}
Let $\lambda$ be a dominant weight.
Let $\tau_{\lambda}$ be as in Section~\ref{subsec-translations}.
\begin{defn}\label{def-mcm}
    We say that $[x,y]$ and $[x',y']$ are \emph{comparable} if there are  $u,v$ dominant elements and $\lambda, \lambda'$ dominant weights such that
    \begin{itemize}
        \item $\tau_{\lambda}$ induces an isomorphism of posets between $[x,y]$ and $[u,v]$.
        \item $\tau_{\lambda'}$ induces an isomorphism of posets between $[x',y']$ and $[u,v]$.
    \end{itemize}
\end{defn}
The following result is a restatement of Theorem~\ref{thm-main intro}.
\begin{thm}\label{thm: main theorem}
    Let $[x,y]$ and $[x',y']$ be thick intervals.
    Then $[x,y]$ and $[x',y']$ are isomorphic as partially ordered sets if and only if there is $g\in \{\id, \sigma\}$ such that $[x,y]$ and $g[x',y']$ are comparable.
\end{thm}
\begin{proof}
    The necessity is obvious.
    Let us prove the sufficiency.
    By Proposition~\ref{prop: iso implies congruent polygons}, we have $\Pgn_{x,y}+\mu=\Pgn_{x',y'}$ or $\Pgn_{x,y}+\mu=\sigma(\Pgn_{x',y'})$ for some (not necessarily dominant) weight $\mu$.
    Suppose that the first equality is satisfied.
    In this case, we will prove that $[x,y]$ and $[x',y']$ are comparable, i.e., we will take $g=\id$ in the statement of the theorem.
    By Lemma~\ref{lem: iso preserves adjacency}, we have that $x+\mu=x'$ and  $y+\mu = y'$.
    There are three possibilities:
    \begin{itemize}
        \item $\Pgn_{x,y}$ is a hexagon.
        By Proposition~\ref{prop: traslaciones que preservan Pgn}(\ref{item-prop: trasl presev hex}), we have $\mu=0$, which proves this case.
        \item $\Pgn_{x,y}$ is a pentagon.
        Let $i\in \{1,2\}$ be such that $\Pgn_{x,y}$ contains two adjacent vertices lying on~$\mathbbm{R}_{\geq 0}\varpi_i+\a_i$.
        By Proposition~\ref{prop: traslaciones que preservan pentagonos}, $\mu \in \mathbbm{R}\varpi_{i}$.
        We may assume that $\mu$ is dominant.
        If it is not, then $-\mu$ is dominant, in which case we can exchange the roles of~$[x,y]$ and $[x',y']$.
        By Proposition~\ref{prop: trasl-por-dominat}, $\tau_\mu$ is an isomorphism of the posets $[x,y]$ and $[x+\mu,y+\mu]$.
        \item $\Pgn_{x,y}$ are parallelograms.
        As $[x,y]$ and $[x',y']$ are isomorphic, $x$ and $x'$ both belong either to $\Theta$ or to $\Theta^s$.
        This implies that there exist $\lambda,\lambda'$ dominant weights, such that $x+\lambda=x'+\lambda'$.
        This implies that $\Pgn_{x,y}+\lambda=\Pgn_{x',y'}+\lambda'$.
        Consider $u\coloneqq x+\lambda=x'+\lambda'$ and $v \coloneqq y+\lambda=y'+\lambda'$.
        By Proposition~\ref{prop: trasl-por-dominat}, we have $[u,v]$ is isomorphic to $[x,y]$ (resp.\@ $[x',y']$) as posets under the map $\t_\lambda$ (resp.\@ $\t_{\lambda'}$).
        So $[x,y]$ and $[x',y']$ are comparable.
    \end{itemize}
    Finally, suppose that $\Pgn_{x,y}+\mu=\sigma(\Pgn_{x',y'})$.
    In this case, we consider $[\sigma(x'),\sigma(y')]$ which is isomorphic to $[x',y']$ as posets.
    By the way in which the polygons $\Pgn$ are constructed, it is not hard to see that $\Pgn_{\sigma(x'),\sigma(y')}=\sigma(\Pgn_{x',y'})$.
    Now we can use the previous case to get that $[x,y]$ and $[\sigma(x'),\sigma(y')]=\sigma([x',y'])$ are comparable.
\end{proof}

%% file: Sections/Applications.tex
\section{Some applications}\label{section: applications}
In this section, we examine some applications of our results.

\subsection{Non-recursive proof of the Combinatorial Invariance Conjecture for thick intervals in \texorpdfstring{$\widetilde{A}_2$}{A2 tilde}}

Kazhdan--Lusztig polynomials remain among the most profound and enigmatic objects in representation theory.
These polynomials $P_{x,y}(q)$ are defined for any pair $(x, y)$ in a Coxeter system $(W, S)$, and they play a central role in the representation theory of Hecke algebras, Lie algebras, algebraic groups, and the topology of Schubert varieties.

In the 1980s, Lusztig and Dyer independently formulated the following conjecture:
\begin{conj}[Combinatorial Invariance Conjecture]\label{conj: KL-invariance}
    Let $(W,S)$ and $(W',S')$ be Coxeter systems with Bruhat orders $\leq$ and $\leq'$, respectively.
    Given $x,y \in W$ and $x',y' \in W'$ such that the intervals $([x,y], \leq)$ and $([x',y'], \leq')$ are isomorphic as posets, their associated Kazhdan--Lusztig polynomials are equal:
    \begin{equation*}
        P_{x,y}(q) = P_{x',y'}(q).
    \end{equation*}
\end{conj}
Building on our results, we provide a non-recursive proof of the Combinatorial Invariance Conjecture in~$\widetilde{A}_2$ for thick intervals (Proposition~\ref{prop: CIC for thick}).
This complements---and is considerably cleaner than---the recursive proof of the full conjecture in~$\widetilde{A}_2$ given by the first two authors and Plaza:
\begin{thm}[{\cite[Theorem 5.2]{BLP23}}]
    \label{thm: CIC for A2 tilde}
    Let $[x,y]$ and $[x',y']$ be intervals in~$\widetilde{A}_2$.
    If $[x,y]$ and $[x',y']$ are isomorphic as partially ordered sets, then $P_{x,y}(q) = P_{x',y'}(q)$.
\end{thm}
Before presenting our proof, we recall some basic notions from~\cite{BLP23}.

The \emph{Hecke algebra of type~$\widetilde{A}_2$} is an associative unital \( \mathbbm{Z}[v, v^{-1}] \)-algebra generated by the set \( \{H_{s_0}, H_{s_1}, H_{s_2}\} \), subject to the relations
\begin{gather*}
    H_{s_i}^2 = (v^{-1} - v) H_{s_i} + 1, \quad \text{for } i \in \{0, 1, 2\}, \\
    H_{s_i} H_{s_j} H_{s_i} = H_{s_j} H_{s_i} H_{s_j}, \quad \text{for } i \neq j.
\end{gather*}
It has two notable \( \mathbbm{Z}[v, v^{-1}] \)-bases: the standard basis \( \{H_w \mid w \in W\} \) and the Kazhdan--Lusztig basis \( \{\underline{H}_w \mid w \in W\} \)~\cite{KL79}.
We have
\begin{equation}
    \underline{H}_w = \sum_{x \in W} h_{x,w}(v)\, H_x,
\end{equation}
where \( h_{x,w}(v) = v^{\ell(x,w)} P_{x,w}(v^{-2}) \), and \( P_{x,w}(q) \in \mathbbm{Z}[q] \) is the Kazhdan--Lusztig polynomial associated with the pair \( (x, w) \).

For \( x \in W \), define
\begin{equation}
    N_x \coloneqq \sum_{z \leq x} v^{\ell(z,x)} H_z.
\end{equation}
Recall that \( \theta(m-1, n-1) = \theta(m,n) - \rho \).
The following proposition provides a non-recursive formula for the Kazhdan--Lusztig basis elements associated with the sets~$\Theta$ and~$\Theta^s$.
\begin{prop}[{\cite[Proposition 3.1]{BLP23}}]\label{prop: formulas para KLP base}
    Let $m,n$ be non-negative integers.
    Then
    \begin{align}
        \underline{H}_{\theta(m,n)} &= N_{\theta(m,n)} + v^2 \underline{H}_{\theta(m,n) - \rho}, \\
        \underline{H}_{\theta(m,n)s} &= N_{\theta(m,n)} + v \underline{H}_{\theta(m-1,n)} + v \underline{H}_{\theta(m,n-1)}.
    \end{align}
\end{prop}
The first and second identities in this proposition imply the following two corollaries, respectively.
\begin{cor}\label{cor: KLPparatheta}
    Let $m$ and $n$ be non-negative integers, and let $x \in W$ be dominant with $x \leq \theta(m,n)$.
    Then
    \begin{equation}\label{eq: KLP para theta}
        P_{x,\theta(m,n)}(q) = \sum_{i=0}^k q^i,
    \end{equation}
    where
    \begin{equation*}
        k = \max \{ j \geq 0 \mid x \leq \theta(m,n) - j\rho \in D \}.
    \end{equation*}
\end{cor}
\begin{cor}\label{cor: KLPparatheta-s}
    Let $m$ and $n$ be non-negative integers, and let $x \in D$ be dominant with $x \leq \theta^s(m,n)$.
    Then
    \begin{equation}
        P_{x,\theta^s(m,n)}(q) = -1 + \sum_{i=0}^{t} q^i + \sum_{i=0}^{r} q^i,
    \end{equation}
    where
    \begin{align*}
        t &= \max\left( \{0\} \cup \{ j \geq 0 \mid x \leq \theta(m-1,n) - j\rho \in D \} \right), \\
        r &= \max\left( \{0\} \cup \{ j \geq 0 \mid x \leq \theta(m,n-1) - j\rho \in D \} \right).
    \end{align*}
\end{cor}
\begin{lem}\label{lem: KL is translation invariant}
    Let $x,y\in D$ be dominant elements, and let $\lambda$ be a dominant weight.
    If $[x,y]\simeq[x+\lambda,y+\lambda]$, then $P_{x,y}(q)=P_{x+\lambda,y+\lambda}(q)$.
\end{lem}
\begin{proof}
    We need to split the proof into two cases: $y = \theta(m,n)$ or $y = \theta^s(m,n)$ for some $m,n \in \mathbbm{N}$.
    \begin{itemize}
        \item[$\bullet$] Let $y = \theta(m,n)$.
        Define
        \begin{equation*}
            J_{x, y} \coloneqq \{ j \geq 0 \mid x \leq y - j\rho \in D \}.
        \end{equation*}
        By Corollary~\ref{cor: KLPparatheta}, the Kazhdan--Lusztig polynomial \( P_{x, y} \) is determined by the set \( J_{x, y} \).
        Thus, it suffices to show
        \begin{equation*}
            J_{x, y} = J_{x + \lambda, y + \lambda}.
        \end{equation*}

        Since \( y - j\rho \in D \) implies \( y - j\rho \leq y \), it follows that
        \begin{equation*}
            J_{x, y} = \{ j \geq 0 \mid y - j\rho \in \Pgn_{x, y} \cap\, D \}.
        \end{equation*}
        By Lemma~\ref{lem: same cardinal iff same poligon}, we have
        \begin{equation*}
            \Pgn_{x + \lambda, y + \lambda} = \Pgn_{x, y} + \lambda.
        \end{equation*}
        Furthermore, by Proposition~\ref{prop: trasl-por-dominat}, the polygon \( \Pgn_{x, y} \) is either a pentagon or a parallelogram.
        Consider the two cases:
        \begin{itemize}
            \item  If \( \Pgn_{x, y} \) is a pentagon, then the above identity implies
            \begin{equation*}
                \Pgn_{x + \lambda, y + \lambda} \cap\, D = (\Pgn_{x, y} \cap\, D) + \lambda,
            \end{equation*}
            which immediately gives \( J_{x, y} = J_{x + \lambda, y + \lambda} \).
            \item  If $\Pgn_{x,y}$ is  a parallelogram,
            then one can verify---by using the fact that \( y = \theta(m, n) \) and by drawing \( \Pgn_{x, y} \)---that
            \begin{equation*}
                y - j\rho \in \Pgn_{x, y} \cap\, D \iff y - j\rho \in \Pgn_{x, y}.
            \end{equation*}
            Hence, we again conclude that \( J_{x, y} = J_{x + \lambda, y + \lambda} \).
        \end{itemize}
        \item[$\bullet$] Now, suppose that $y = \theta^s(m,n)$.
        The proof is similar to the first bullet but using  Corollary~\ref{cor: KLPparatheta-s} instead of Corollary~\ref{cor: KLPparatheta}.\qedhere
    \end{itemize}
\end{proof}
Finally, we deduce the following result from Theorem~\ref{thm: main theorem} and Lemma~\ref{lem: KL is translation invariant}.
\begin{prop}\label{prop: CIC for thick}
    Let $[x, y]$ and $[x', y']$ be thick intervals.
    If $[x, y]$ and $[x', y']$ are isomorphic as partially ordered sets, then \( P_{x, y}(q) = P_{x', y'}(q) \).
\end{prop}

\subsection{Stabilization of affine Bruhat intervals}\label{subsec: stabilization}
Let  $W = \mathbbm{Z}\Phi^\vee \rtimes W_f$ be an affine Weyl group with root system $\Phi$ and finite Weyl group $W_f$ with longest element $w_f$.
Let $C_+$ be the dominant Weyl chamber corresponding to the set of positive roots $\Phi_+$.

Let $A_x$ be the alcove corresponding to $x$.
The coweight lattice $\Lambda^\vee \supset \mathbbm{Z}\Phi^\vee$ acts on~$W$ via translations: for any $\mu \in \Lambda^\vee$ and $x \in W$, there exists a unique element $x + \mu \in W$ whose alcove is $\mu + A_x$, that is, the translation by $\mu$ of the alcove $A_x$.

An element $x\in W$ is \emph{dominant} if $A_x\subset C_+$.
A coweight $\mu\in\Lambda^\vee$ is \emph{dominant} if $\mu\in \overline{C_+}$ (i.e., $\langle\alpha, \mu \rangle \geq 0$ for all $\alpha \in \Phi_+$.) 
\begin{conj}[Stabilization of dominant Bruhat intervals]
    \label{conj: stabilization of dominant Bruhat intervals}
    For any dominant coweight $\lambda \in \Lambda^\vee$ and $x, y \in W$ such that $x$ and $y$ are dominant, there exists an integer $N_0 = N_0(x, y, \lambda)$ such that for all $N \geq N_0$,
    \begin{equation*}
        [x+N_0\lambda,\, y+N_0\lambda] \cong [x+N\lambda,\, y+N\lambda].
    \end{equation*}
\end{conj}
Using SageMath, we verified this conjecture for several intervals and $\lambda$ in several directions, when $\Phi$ is of type~$A_3$, $B_3$, and~$A_4$.
Including translations, we verified around $50$ interval isomorphisms in each of these types.

The reason we believe this conjecture is true stems from Proposition~\ref{prop: An intro}.
The existence of the parallelotopes $\Par^i_{x,y}$ associated with an interval $[x, y]$ suggests that a statement similar to Theorem~\ref{thm-main intro} may hold in general.

Instead of defining $\Pgn_{x,y}$ as before, one could define
\begin{equation*}
    \Pgn^i_{x,y} \coloneqq R \cap \Par^i_{x,y},
\end{equation*}
where $R$ is a region similar to $F_{\mathrm{id}}$.
Although we do not know precisely what this region should be, a reasonable candidate is
\begin{equation*}
    R = \{v \in E \mid -1 \leq \langle \alpha_i^\vee, v \rangle\}.
\end{equation*}
Suppose the poset class of $[x, y]$ depends only on the collection ${\Pgn^i_{x,y}}$.
Since the members of this collection stabilize after translation by a sufficiently large multiple $N_0\lambda$, this provides further support for the conjecture.

This intuition is reinforced by the fact that each $\Pgn^i_{x,y}$ depends only on the relative position of $x$ and $y$, and on the distances from each vertex of $A_x$ to the walls. (Below, we prove a related result in type~$\widetilde{A}_2$.)

In type~$A$, Proposition~\ref{prop: dominant translations in An tilde intro} offers additional evidence---though only within the dominant region.

We will prove this conjecture in the setting of~$\widetilde{A}_2$.
To proceed, we first establish the following lemmas.
\begin{lem}
    \label{lem: Pgn stabilizes}
    Suppose $x, y$ are dominant elements and $\lambda$ is a dominant weight.
    There exists an integer $N_0$ such that for all $N \geq N_0$, we have
    \begin{equation*}
        \Pgn_{x+N\lambda,\, y+N\lambda}=(N-N_0)\lambda+\Pgn_{x+N_0\lambda,\, y+N_0\lambda}.
    \end{equation*}
    Moreover, there are two cases:
    \begin{itemize}
        \item $\Pgn_{x+N_0\lambda,\, y+N_0\lambda}$ is a parallelogram.
        \item $\Pgn_{x+N_0\lambda,\, y+N_0\lambda}$ is a pentagon with two adjacent vertices on~$\mathbbm{R}\varpi_{i}+\alpha_{i}$ and $\lambda\in \mathbbm{R}_{\geq0}\varpi_i$ for some $i\in\{1,2\}$.
    \end{itemize}
\end{lem}
\begin{proof}
    By Lemma~\ref{lem: vertice de 60 grados}, we have that
    $
        v_i\coloneqq \cen(x)+(\varpi_i,\cen(y)-\cen(x))\alpha_i\in \mathrm{V}(\Par_{x,y}),
    $
    for $1\leq i \leq 2$.
    \newline
    
    \noindent\textbf{Case A.} Suppose that $(\alpha_i, \lambda)>0$ for all $1\leq i \leq 2$.
    Let $j$ such that $\{i,j\}=\{1,2\}$.
    Clearly, there exists an integer $N_i$ such that $(\alpha_i,v_i+N_i\lambda)>-1$.
    We have $(\alpha_j, v_i+N_i\lambda)>0$ and $(\rho,v_i+N_i\lambda)>0$, which implies that $v_i+N_i\lambda\in F_{\operatorname{id}}$ for $1\leq i \leq 2$.
    Let $N_0=\max\{N_1,N_2\}$, $x'=x+N_0\lambda$, $y'=y+N_0\lambda$, and $v'_i\coloneqq v_i+N_0\lambda\in F_{\operatorname{id}}$.
    By Lemma~\ref{lem: vertice de 60 grados}, $v'_i\in \mathrm{V}(\Par_{x',y'})$.
    By the convexity of~$F_{\id}$, we have $\Par_{x',y'}\subset F_{\id}$,  so $\Par_{x',y'}=\Pgn_{x',y'}$ and $\Par_{x',y'}$ is a parallelogram.
    The result follows from Proposition~\ref{prop: traslaciones que preservan Pgn}(\ref{item-prop: trasl presev par}).
    \newline
    
    \noindent\textbf{Case B.}
    As always, we use Definition~\ref{def: label of vertices} for the notation of the vertices of~$\Pgn_{x,y}$.
    Suppose that $(\alpha_1, \lambda)>0$ and $(\alpha_2, \lambda)=0$.
    In this case, $\lambda=a\varpi_1$ for some $a>0$.
    Clearly, there is an integer $N_0$ such that $(\alpha_1,v'_2)>-1$, where $v'_2\coloneqq v_2+N_0\lambda$.
    Since $x$ is dominant, we have $(\alpha_2, v'_2)>-1$ and $(\rho, v'_2)>-1$, so $v'_2\in F_{\operatorname{id}}\setminus \partial F_{\id}$.
    Let $x'=x+N_0\lambda$, $y'=y+N_0\lambda$.
    By Lemma~\ref{lem: vertice de 60 grados}, we have that $v'_2\in \mathrm{V}(\Par_{x',y'})$, so $v'_2=v_2^{x',y'}(x')\in \mathrm{V}(\Pgn_{x',y'})$.
    By Lemma~\ref{lem: vertex in the wall and vertex outside the wall}\ref{lem: vertex in the wall and vertex outside the wall part 1}, we have $v_2^{x',y'}(x')=v_1^{x',y'}(y')$, so $\Pgn_{x',y'}$ is either a parallelogram or a pentagon with two adjacent vertices on~$\mathbbm{R}_{\geq 0}\varpi_{1}+\alpha_{1}$.
    In either case, the result follows from Proposition~\ref{prop: traslaciones que preservan Pgn}(\ref{item-prop: trasl presev par}) or Proposition~\ref{prop: traslaciones que preservan Pgn}(\ref{item-prop: trasl presev pent}).
    \newline

    \noindent\textbf{Case C.} Suppose that $(\alpha_1, \lambda)=0$ and $(\alpha_2, \lambda)>0$.
    This is analogous to the previous case.
\end{proof}
\begin{prop}
    \label{prop: dominant Bruhat intervals stabilize in A2}
    Conjecture~\ref{conj: stabilization of dominant Bruhat intervals} holds for $W$ of type~$\widetilde{A}_2$.
\end{prop}
\begin{proof}
    By Lemma~\ref{lem: Pgn stabilizes}, there is $N_0$ such that if $N\geq N_0$, we have
    \begin{equation*}
        \Pgn_{x+N\lambda,\, y+N\lambda}=(N-N_0)\lambda+\Pgn_{x+N_0\lambda,\, y+N_0\lambda},
    \end{equation*}
    and we have two cases:
    \begin{itemize}
        \item $\Pgn_{x+N_0\lambda,\, y+N_0\lambda}$ is a parallelogram.
        \item $\Pgn_{x+N_0\lambda,\, y+N_0\lambda}$ is a pentagon with two adjacent vertices on~$\mathbbm{R}\varpi_{i}+\alpha_{i}$ and $\lambda\in \mathbbm{R}_{\geq0}\varpi_i$ for some $i\in\{1,2\}$.
    \end{itemize}
    In each case, by Proposition~\ref{prop: trasl-por-dominat} the map 
    \begin{equation*}
        \tau_{(N-N_0)\lambda}\colon [x+N_0\lambda,\, y+N_0\lambda]\to [x+N\lambda,\, y+N\lambda]
    \end{equation*}
    is an isomorphism of posets.

    The second part of the proposition follows from Proposition~\ref{prop: traslaciones que preservan Pgn}.
\end{proof}

\subsection{Classification of lower intervals in \texorpdfstring{$\widetilde{A}_2$}{A2~}} \label{subsection: Classification lower}
Let $W$ be the Weyl group of affine type~$A_2$.
In this section, we classify the lower Bruhat intervals $[\mathrm{id}, y]$ for $y\in W$ up to poset isomorphism.
\begin{prop}\label{prop: main lower} 
    Let $u,v$ be elements in~$W$.
    Then $[\mathrm{id},u]\simeq [\mathrm{id},v]$ if and only if $v$ belongs to the $G$-orbit of~$u$.
\end{prop}
We need to deal first with ``exceptional'' intervals $[\mathrm{id}, y]$ for small values of~$\ell(y)$ which cannot be told apart from their cardinalities or $\operatorname{LC}(\mathrm{id}, y)$.
\begin{lem}\label{lema: casos no isomorfos}
    We have $[\mathrm{id},u]\not\simeq [\mathrm{id},v]$ in the following cases: 
    \begin{enumerate}
         \item\label{item: lem no iso 0}  $\{u,v\}=\{s_0\theta^s(0,0), \mathrm{\mathbf{x}}_{5}\}$.
         \item\label{item: lem no iso 1.2} $\{u,v\}=\{\theta^s(0,1), \mathrm{\mathbf{x}}_{6}\}$.
         \item\label{item: lem no iso 1}  $\{u,v\}=\{\theta^s(1,0), \mathrm{\mathbf{x}}_{6}\}$.
         \item\label{item: lem no iso 2}  $\{u,v\}=\{\theta(1,1), \mathrm{\mathbf{x}}_{7}\}$.
    \end{enumerate}
\end{lem}
\begin{proof}
    Let $\phi\colon[\id,u]\to[\id,v]$ be an isomorphism.
    By Lemma~\ref{lem: iso-preserva-dih}, we have $[a,b]\subset [\id,u]$ is a dihedral subinterval if and only if $\phi([a,b])\subset [\id,v]$ is a dihedral subinterval.
    
    From Proposition~\ref{prop: Dyer dihedral} and Corollary~\ref{cor: characterization of the Bruhat order}, it is an easy (but tedious) task to see that:
    \begingroup
    \allowdisplaybreaks
    \begin{align*}
        \mathcal{D}(s_0\theta^s(0,0)) &= \{121, 123, 213, 312, 313, 321, 323\}, \\
        \mathcal{D}(\mathbf{x}_5) &= \{213, 1213, 1321, 2313, 12312, 12313, 13231, \\
                                  &\phantom{=\{\;} 23121, 23123, 31231, 31321, 31323\}, \\
        \mathcal{D}(\theta^s(0,1)) &= \{321, 1213, 1231, 1321, 1323, 2132, 2313, 3231\}, \\
        \mathcal{D}(\theta^s(1,0)) &= \{312, 1213, 1231, 1323, 2132, 2312, 2313, 3132\}, \\
        \mathcal{D}(\mathbf{x}_6) &= \{132, 213, 1213, 1231, 1323, 2312, 2313, 3123, 3132\}, \\
        \mathcal{D}(\theta(1,1)) &= \{1213, 3121, 12312, 12313, 13231, 21321, 21323, \\
                                     &\phantom{=\{\;} 23121, 23132, 31321\}, \\
        \mathcal{D}(\mathbf{x}_7) &= \{213, 1213, 1321, 2313, 12312, 12313, 13231, \\
                                  &\phantom{=\{\;} 23121, 23123, 31231, 31321, 31323\}.
    \end{align*}
    \endgroup
    In each case, we observe that $|\mathcal{D}(u)|\neq |\mathcal{D}(v)|$, so $[\mathrm{id},u]\not\simeq [\mathrm{id},v]$.
\end{proof}
\begin{lem}\label{lemma: lowerlargo impar}
    If $[\id,u]\simeq [\id,v]$ and $\ell(u)=\ell(v)$ is odd, then $u,v\in \Theta$ or $u,v\in {^s\Theta^s}$ or $u,v\in X^\text{odd}$.
\end{lem}
\begin{proof} 
    By Remark~\ref{rem: facts of theta-partition}\ref{remark-item: length of the elements}, we have $u,v\in \Theta\uplus {^s\Theta^s}\uplus X^\text{odd}$.
    We will suppose by contrapositive that $u, v$ belong to different regions, then we will prove that $[\id,u]\not\simeq [\id,v]$.
    Modulo swapping $u$ and $v$, there are only three possibilities.
    \begin{itemize}
        \item $u\in \Theta$ and $v\in X^\text{odd}$.
        By Remark~\ref{rem: remark tontos}, we can assume that $u=\theta(m,n)$ and $v=\mathrm{\mathbf{x}}_{2(m+n+2)+1}$.
        By Proposition~\ref{prop: cardinal lower}, we have 
        \begin{equation*}
            |[\id,\theta(m,n)]|=|[\mathrm{id}, \mathrm{\mathbf{x}}_{2(m+n+1)+1}]|\implies 6mn-2m-2n-2=0.
        \end{equation*}
        The unique solution over the non-negative integers of that equation is $(m,n)=(1,1)$.
        By Lemma~\ref{lema: casos no isomorfos}(\ref{item: lem no iso 2}), we have $[\id,u]\not\simeq [\id,v]$.
        \item $u\in {^s\Theta^s}$ and $v\in X^\text{odd}$.
        The proof is similar to the previous case.
        \item $u\in \Theta$ and $v\in {^s\Theta^s}$.
        By Remark~\ref{rem: remark tontos}, we can assume that $u=\theta(m,n)$ and $v=s_0\theta^s(h,k)$.
        If $\ell(u)=\ell(v)$ then $m+n=h+k+2$.
        If in addition $|[\id,\theta(m,n)]|=|[\id,s_0\theta^s(h,k)]|$, we have
        \begin{equation*}
            3(mn-hk-h-k)=2,
        \end{equation*}
        which does not have solutions in~$\mathbbm{Z}$, so $[\id,u]\not\simeq [\id,v]$.\qedhere
    \end{itemize}
\end{proof}
The proof of the following lemma is similar to the previous one, so we will skip it.
\begin{lem}\label{lemma: lowerlargo par}
    If $[\id,u]\simeq [\id,v]$ and $\ell(u)$ is even, then $u,v\in \Theta^s$ or $u,v\in X^{\mathrm{even}}$.
\end{lem}
\begin{proof}[Proof of Proposition~\ref{prop: main lower}]
    The only if part follows from the definition of the  Bruhat order and the group $G$.
    Conversely, if $[\mathrm{id},u]\simeq [\mathrm{id},v]$, we have $\ell(u)=\ell(v)$.
    We can suppose that $\ell(u)$ is even, as the proof of the other case is similar.
    By Remark~\ref{rem: facts of theta-partition}\ref{remark-item: length of the elements} and Lemma~\ref{lemma: lowerlargo par}, either $u,v\in X^\mathrm{even}$ or $u,v\in \Theta^s$.
    
    If $u,v\in X^\mathrm{even}$, the equality $\ell(u)=\ell(v)$ immediately implies $v\in G\cdot u$.
    
    If $u,v\in \Theta^s$, by Remark~\ref{rem: remark tontos}, we can assume $u=\theta^s(m,n)$, $v=\theta^s(h,k)$.
    The equality $\ell(u)=\ell(v)$ implies $m+n=h+k$.
    From $|[\id,\theta^s(m,n)]|=|[\id,\theta^s(h,k)]|$ and Proposition~\ref{prop: cardinal lower}, we get $\{m,n\}=\{h,k\}$.
    Then either $u=\sigma v$ or $u=v$.
\end{proof}

%% file: Sections/Towards_a_general_picture.tex
\section{Towards a general picture in affine Weyl groups}\label{section: towards a general picture}
Let $W = \mathbbm{Z}\Phi^\vee \rtimes W_f$ be an arbitrary affine Weyl group, as in Section~\ref{subsec: stabilization}.
We denote by $W_+ \coloneqq \{x \in W \mid A_x \subset C_+\}$ the set of \emph{dominant elements} in~$W$.

By abuse of notation, we write $y - x \in \Lambda^\vee$ if there exists a coweight $\lambda \in \Lambda^\vee$ such that $x + \lambda = y$ (that is, $A_x + \lambda = A_y$).

\subsection{A conjecture for arbitrary affine Weyl groups}\label{subsec: weak conjecture}
Let us recall Conjecture~\ref{conj: weak conj intro}.
\begin{conj}\label{conj: weak generalization Gen} 
    Let $x, x', y, y' \in W_+$ be such that:
    \begin{itemize}
        \item $[x, y]$ and $[x', y']$ are isomorphic as posets,
        \item $[x, y]$ and $[x', y']$ are full,
        \item there exists $g_0 \in G$ such that $x - g_0 x' \in \Lambda^\vee$.
    \end{itemize}
    Then there exists $g \in G$ such that $y - g y' = x - g x' \in \Lambda^\vee$.
\end{conj}
We verified Conjecture~\ref{conj: weak generalization Gen} through exhaustive SageMath checks on all interval pairs within the bounds listed in Table~\ref{tab: conj-evidence}.
\begin{table}[ht]
    \centering
    \renewcommand{\arraystretch}{1.6}
    \begin{tabular}{|c|c|c|}
        \hline
        \textbf{Type} & \textbf{Bound on $\ell(x,y)$} & \textbf{Bound on $\ell(x)$} \\
        \hline
        $\widetilde{A}_2$ & Thick intervals & \textemdash \\
        \hline
        $\widetilde{A}_3$ & $\leq 10$ & $\leq 20$ \\
        \hline
        $\widetilde{A}_4$ & $\leq 12$ & $\leq 10$ \\
        \hline
        $\widetilde{B}_2$ & $\leq 15$ & $\leq 14$ \\
        \hline
        $\widetilde{B}_3$ & $\leq 11$ & $\leq 15$ \\
        \hline
        $\widetilde{C}_3$ & $\leq 11$ & $\leq 12$ \\
        \hline
        $\widetilde{G}_2$ & $\leq 15$ & $\leq 20$ \\
        \hline
    \end{tabular}
    \caption{Verified interval pairs by type and bounds as evidence for Conjecture~\ref{conj: weak conj intro}.
    The $\widetilde{A}_2$ case is fully proven for thick intervals in Theorem~\ref{thm-main intro}.
    All affine types of rank~$\leq 3$ are represented.}
    \label{tab: conj-evidence}
\end{table}

\subsection{Towards a general picture in type~\texorpdfstring{$A$}{A}}\label{subsec: results in An tilde}
In this section, we propose extensions of Theorem~\ref{thm-main intro}, prove Proposition~\ref{prop: An intro}, and provide a partial generalization of Proposition~\ref{prop: trasl-por-dominat} for dominant weights in type~$\widetilde{A}_n$ (Proposition~\ref{prop: dominant translations in An tilde intro}).

The following conjecture extends Theorem~\ref{thm-main intro} to other intervals in type~$\widetilde{A}_2$.
Although we have not been able to prove it, doing so would require a deeper understanding of the cardinalities of the intervals $[x, y]$, which appears to be more subtle than the analysis in Section~\ref{subsec: translations preserving cardinality}.
Nonetheless, the conjecture has been verified using SageMath for all pairs of full intervals with $\ell(x, y) \leq 10$ and $\ell(x) \leq 20$.
\begin{conj}\label{conj: a2general}
    Let \( W \) be an affine Weyl group of type~\( \widetilde{A}_2 \).
    Suppose \( x, x', y, y' \in W \) satisfy \( x, x' \in D \), and that both intervals \( [x, y] \) and \( [x', y'] \) are full and poset isomorphic.
    Then there exists an element \( g \in G \) such that
    \begin{equation*}
        \tau_{\lambda} \colon [x, y] \to [gx', gy']
    \end{equation*}
    induces a poset isomorphism, where \( \lambda = gx' - x \in \Lambda \).
    
    An analogous statement holds for \( x, x' \in s_0 D \), provided the regions used to define \( \tau_{\lambda} \) are suitably modified.
\end{conj}
Henceforth, let $W$ denote the affine Weyl group of type~$\widetilde{A}_n$ with simple roots $\Delta = \{ \alpha_1, \ldots, \alpha_n \}$, so that we may identify $\Lambda$ with the coweight lattice $\Lambda^\vee$.

Let $\Lambda_+ \coloneqq \Lambda \cap \overline{C_+}$ denote the set of dominant weights, equipped with the dominance order $\preceq$.
Let $\{\varpi_1, \dots, \varpi_n\}$ be the fundamental weights, and for convenience, set $\varpi_0 = \mathbf{0}$ to denote the origin of~$E$.
Then the vertices of the fundamental alcove are $\{\varpi_0, -\varpi_1, \dots, -\varpi_n\}\subset\Lambda$.
We let $W$ act on $\Lambda$ from the left in the natural way.

Note that $x$ is dominant if and only if $x(-\varpi_i) \in \Lambda_+$ for all $i \in \{0, 1, \dots, n\}$

The Bruhat order on dominant elements is characterized by the following criterion.
\begin{prop}[{\cite[Proposition 1.1]{CdLP25}}]\label{prop: criterion bruhat-coweight}
    Let $x, y \in W_+$.
    Then $x \leq y$ in the Bruhat order if and only if 
    \begin{equation*}
        x(-\varpi_i) \preceq y(-\varpi_i) \quad \text{for all } i \in \{0,1,\hdots,n\}.
    \end{equation*}
\end{prop}
By Proposition~\ref{prop: criterion bruhat-coweight}, we directly obtain the following result.
\begin{lem}\label{lem: tras in A_n dominante}
    Let $x,y\in W_+$ and  $\lambda\in\Lambda_+\cap \mathbbm{Z}\Phi$.
    Then $x\leq y\iff x+\lambda\leq y+\lambda$.
\end{lem}
For any $\lambda \in E$, we define the \emph{weight polytope associated with $\lambda$} by $P(\lambda)\coloneqq \operatorname{Conv}(\{w\lambda\mid w\in W_f\}).$ 
We have the following lemma.
\begin{lem}{\cite[Proposition~8.44]{Hal15}}\label{lem: equiv dominance order}
    Suppose that $\lambda,\mu\in \Lambda_+$.
    Then $\mu\preceq\lambda$ if and only if $\mu\in P(\lambda)$.
\end{lem}
\begin{defn}\label{def: i-th parallelepiped}
    Let $x,y\in W_+$ be such that $x\leq y$.
    For $i\in \{0,1,\hdots,n\}$, we define the \emph{$i$-th parallelotope $\Par_{x,y}^i$ associated with $[x,y]$} as 
    \begin{equation*}
        \Par_{x,y}^i\coloneqq(x(-\varpi_i)+\mathbbm{R}_{\geq0}\Delta)\cap(y(-\varpi_i)-\mathbbm{R}_{\geq0}\Delta).
    \end{equation*}
\end{defn}
The following result is a reformulation of Proposition~\ref{prop: An intro}.
\begin{prop}\label{prop: An}
    Let $x, y \in W_+$ with $x \leq y$.
    Then $z \in [x, y] \cap W_+$ if and only if $z(-\varpi_i) \in \Par_{x,y}^i \cap\, \overline{C_+}$ for all $i \in \{0, 1, \dots, n\}$.
\end{prop}
\begin{proof}
    By Proposition~\ref{prop: criterion bruhat-coweight}, we have
    \begin{equation}\label{eq: proof general order and geom.}
        z \in [x, y] \cap W_+ 
        \iff 
        x(-\varpi_i) \preceq z(-\varpi_i) \preceq y(-\varpi_i)
        \quad \text{for all } i \in \{0, 1, \dots, n\}.
    \end{equation}
    By Lemma~\ref{lem: equiv dominance order},
    \begin{equation}\label{eq: proof equiv dom order}
        z(-\varpi_i) \preceq y(-\varpi_i)
        \iff 
        z(-\varpi_i) \in P(y(-\varpi_i)).
    \end{equation}
    Observe that
    \begin{equation*}
        P(y(-\varpi_i)) \cap \overline{C_+}
        = \big(y(-\varpi_i) - \mathbb{R}_{\geq 0} \Delta\big) \cap \overline{C_+}.
    \end{equation*}
    Thus, Equation~\eqref{eq: proof equiv dom order} becomes
    \begin{equation}\label{eq1: proof equic dom order}
        z(-\varpi_i) \preceq y(-\varpi_i)
        \iff 
        z(-\varpi_i) \in \big(y(-\varpi_i) - \mathbb{R}_{\geq 0} \Delta\big) \cap \overline{C_+}.
    \end{equation}
    Analogously,
    \begin{align}
        \begin{split}\label{eq2: proof equic dom order}
            x(-\varpi_i) \preceq z(-\varpi_i)
            &\iff 
            x(-\varpi_i) \in \big(z(-\varpi_i) - \mathbb{R}_{\geq 0} \Delta\big) \cap \overline{C_+} \\
            &\iff 
            z(-\varpi_i) \in \big(x(-\varpi_i) + \mathbb{R}_{\geq 0} \Delta\big) \cap \overline{C_+}.
        \end{split}
    \end{align}
    The result follows from Equations~\eqref{eq: proof general order and geom.}, \eqref{eq1: proof equic dom order}, and \eqref{eq2: proof equic dom order}, together with Definition~\ref{def: i-th parallelepiped}.
\end{proof}
The next lemma is immediate.
\begin{lem}\label{lem: tras presv n-parall}
    Let $\lambda \in \Lambda_+ \cap \mathbbm{Z} \Phi$, and let $x, y \in W_+$ with $x \leq y$.
    Then, for every $i \in \{0, 1, \dots, n\}$, we have
    \begin{equation*}
        \Par^i_{x,y} + \lambda = \Par^i_{x+\lambda,\, y+\lambda}.
    \end{equation*}
\end{lem}
The following is a restatement of Proposition \ref{prop: dominant translations in An tilde intro}
\begin{prop}\label{prop: dominant translations in An tilde}
    Let $x, y \in W_+$ be such that $x \leq y$.
    Define the set 
    \begin{equation}
        I(x,y)\coloneqq\{1\leq i \leq n \mid \overline{A_z}\cap H_{\alpha_i,0}\neq \emptyset \mbox{ for some $z\in [x,y]\cap W_+$} \}.
    \end{equation}
    Then for every $\lambda\in \mathbbm{Z}\Phi\cap \bigoplus_{i\not\in I(x,y)}\mathbbm{Z}_{\geq 0}\varpi_i$
    the map
    \begin{align}
        [x, y] \cap W_+ &\to [x+\lambda, y+\lambda] \cap W_+\\
        z&\mapsto z+\lambda
    \end{align}
    is a poset isomorphism.
\end{prop}
\begin{proof}
    By Lemma~\ref{lem: tras in A_n dominante}, the map is an order embedding.
    It remains to show surjectivity.

    Let $H = \bigcap_{i \in I(x,y)} \{v \in E \mid (v, \alpha_i) \geq 0\}$.
    For a subset $X \subset E$, denote $X_H \coloneqq X \cap H$.
    Let $S \subset E$ and $s \in S$.
    Since $(s + \lambda, \alpha_i) = (s, \alpha_i)$ for all $i \in I(x,y)$, it follows that
    \begin{equation*}
        S_H + \lambda = (S + \lambda)_H.
    \end{equation*}
    Let $0 \leq k \leq n$.
    Applying this to $S = \Par_{x,y}^k$, we get
    \begin{equation}\label{eq: prop gnrl Par tras}
        (\Par_{x,y}^k)_H + \lambda = (\Par_{x,y}^k + \lambda)_H.
    \end{equation}
    By convexity of alcoves, we have $i \notin I(x, y)$ if and only if $\Par_{x,y}^k \cap\, H_{\alpha_i, 0} = \emptyset$ for all $0 \leq k \leq n$.
    In particular,
    \begin{equation}\label{eq: prop gnrl Par inter}
        (\Par_{x,y}^k)_H = \Par_{x,y}^k \cap\, \overline{C_+}.
    \end{equation}
    Combining \eqref{eq: prop gnrl Par tras} and \eqref{eq: prop gnrl Par inter}, we obtain
    \begin{equation*}
        (\Par_{x,y}^k + \lambda)_H = (\Par_{x,y}^k \cap\, \overline{C_+}) + \lambda.
    \end{equation*}
    By Lemma~\ref{lem: tras presv n-parall}, we have
    \begin{equation*}
        \Par_{x+\lambda, y+\lambda}^k \cap\, \overline{C_+} 
        = (\Par_{x+\lambda, y+\lambda}^k)_H 
        = (\Par_{x,y}^k + \lambda)_H 
        = (\Par_{x,y}^k \cap\, \overline{C_+}) + \lambda.
    \end{equation*}
    The surjectivity then follows from Proposition~\ref{prop: An}.
\end{proof}